\documentclass[11pt]{amsart}

\usepackage{pdfsync}
\usepackage{amsthm}
\usepackage{amsmath}
\usepackage{amssymb}
\usepackage{mathtools}

\usepackage{verbatim}
\usepackage{geometry}

\usepackage[shortlabels]{enumitem}
\setlist{nosep}
\newcommand{\eps}{\ensuremath{\varepsilon}}

\newcommand{\Prob}[1]{\ensuremath{%
    \mathbb P\left[#1\right]
  }}
\newcommand{\Expect}[1]{\ensuremath{%
    \mathbb E\left[#1\right]
  }}

\usepackage[T1]{fontenc}
\usepackage{amssymb}
\usepackage{amsmath}
\usepackage{amsthm}
\usepackage{fancyhdr}
\usepackage{mathtools}
\usepackage[british]{babel}
\usepackage{geometry}
\usepackage{enumitem}
\usepackage{algpseudocode}
\usepackage{dsfont}
\usepackage{centernot}
\usepackage{xstring}
\usepackage{colortbl}
\usepackage[section]{algorithm}
\usepackage{graphicx}
\usepackage[font={footnotesize}]{caption}
\usepackage[usenames,dvipsnames,table]{xcolor}
\usepackage{pstricks,tikz}

\usetikzlibrary{calc}
\usetikzlibrary{shapes}
\usetikzlibrary{snakes}

\usepackage{xcolor}
\usepackage[h]{esvect}
\usepackage[noadjust]{cite}
\usepackage[
		bookmarksopen=true,
		bookmarksopenlevel=1,
		colorlinks=true,
		linkcolor=darkblue,
        linktoc=page,
		citecolor=darkblue,
]{hyperref}

\definecolor{OwlRed}{RGB}{ 255, 92, 168}
\definecolor{OwlGreen}{RGB}{ 90, 168, 0}
\definecolor{OwlBlue}{RGB}{ 0, 152, 233}
\definecolor{OwlYellow}{RGB}{ 242, 147, 24}
\colorlet{OwlViolet}{OwlRed!50!OwlBlue}
\colorlet{OwlBrown}{OwlRed!50!OwlGreen}
\colorlet{OwlOrange}{OwlRed!50!OwlYellow}
\colorlet{OwlCyan}{OwlGreen!50!OwlBlue}

\title[]{Almost all optimally coloured complete graphs contain a rainbow Hamilton path}
\date{\today}

\author[Stephen~Gould]{Stephen~Gould}
\email{\vspace{-20mm}\{spg377,t.j.kelly,d.kuhn,d.osthus\}@bham.ac.uk}

\author[Tom~Kelly]{Tom~Kelly}

\author[Daniela~K\"uhn]{Daniela~K\"uhn}

\author[Deryk~Osthus]{Deryk~Osthus}

\address{School of Mathematics, University of Birmingham,
Edgbaston, Birmingham, B15 2TT, United Kingdom}

\thanks{This project has received partial funding from the European Research
Council (ERC) under the European Union's Horizon 2020 research and innovation programme (grant agreement no. 786198, D.~K\"uhn and D.~Osthus).
The research leading to these results was also partially supported by the EPSRC, grant nos. EP/N019504/1 (T.~Kelly and D.~K\"uhn) and EP/S00100X/1 (D.~Osthus).}

\newtheorem{theorem}{Theorem}[section]
\newtheorem{prop}[theorem]{Proposition}
\newtheorem{lemma}[theorem]{Lemma}
\newtheorem{cor}[theorem]{Corollary}
\newtheorem{fact}[theorem]{Fact}
\newtheorem{conj}[theorem]{Conjecture}

\theoremstyle{definition}

\newtheorem{defin}[theorem]{Definition}

\newtheoremstyle{claimstyle}{5pt}{5pt}{\em}{5pt}{\em}{:}{5pt}{}
\theoremstyle{claimstyle}
\newtheorem{claim}{Claim}

\newtheoremstyle{stepstyle}{10pt}{5pt}{\em}{0pt}{\em}{:}{5pt}{}
\theoremstyle{stepstyle}

\numberwithin{equation}{section}

\definecolor{darkblue}{rgb}{0,0,0.5}

\def\noproof{{\unskip\nobreak\hfill\penalty50\hskip2em\hbox{}\nobreak\hfill%
       $\square$\parfillskip=0pt\finalhyphendemerits=0\par}\goodbreak}
\def\endproof{\noproof\bigskip}

\def\noclaimproof{{\unskip\nobreak\hfill\penalty50\hskip2em\hbox{}\nobreak\hfill%
       $-$\parfillskip=0pt\finalhyphendemerits=0\par}\goodbreak}

\def\endclaimproof{\noclaimproof\medskip}

\newdimen\margin
\def\textno#1&#2\par{
   \margin=\hsize
   \advance\margin by -4\parindent
          \setbox1=\hbox{\sl#1}
   \ifdim\wd1 < \margin
      $$\box1\eqno#2$$
   \else
      \bigbreak
      \hbox to \hsize{\indent$\vcenter{\advance\hsize by -3\parindent
      \it\noindent#1}\hfil#2$}
      \bigbreak
   \fi}

\def\lateproof#1{\removelastskip\penalty55\medskip\noindent\setcounter{claim}{0}\setcounter{step}{0}{\bf Proof of #1. }} 

\def\claimproof{\removelastskip\penalty55\medskip\noindent{\em Proof of claim: }}

\begin{document}

\newcommand{\new}[1]{\textcolour{red}{#1}}
\def\COMMENT#1{}
\def\TASK#1{}
\newcommand{\APPENDIX}[1]{}
\newcommand{\NOTAPPENDIX}[1]{#1}
\newcommand{\todo}[1]{\begin{center}\textbf{to do:} #1 \end{center}}

\def\eps{{\varepsilon}}
\newcommand{\ex}{\mathbb{E}}
\newcommand{\pr}{\mathbb{P}}
\newcommand{\cB}{\mathcal{B}}
\newcommand{\cA}{\mathcal{A}}
\newcommand{\cE}{\mathcal{E}}
\newcommand{\cS}{\mathcal{S}}
\newcommand{\cF}{\mathcal{F}}
\newcommand{\cG}{\mathcal{G}}
\newcommand{\bL}{\mathbb{L}}
\newcommand{\bF}{\mathbb{F}}
\newcommand{\bZ}{\mathbb{Z}}
\newcommand{\cH}{\mathcal{H}}
\newcommand{\cC}{\mathcal{C}}
\newcommand{\cM}{\mathcal{M}}
\newcommand{\bN}{\mathbb{N}}
\newcommand{\bR}{\mathbb{R}}
\def\O{\mathcal{O}}
\newcommand{\cP}{\mathcal{P}}
\newcommand{\cQ}{\mathcal{Q}}
\newcommand{\cR}{\mathcal{R}}
\newcommand{\cJ}{\mathcal{J}}
\newcommand{\cL}{\mathcal{L}}
\newcommand{\cK}{\mathcal{K}}
\newcommand{\cD}{\mathcal{D}}
\newcommand{\cI}{\mathcal{I}}
\newcommand{\cV}{\mathcal{V}}
\newcommand{\cT}{\mathcal{T}}
\newcommand{\cU}{\mathcal{U}}
\newcommand{\cW}{\mathcal{W}}
\newcommand{\cX}{\mathcal{X}}
\newcommand{\cY}{\mathcal{Y}}
\newcommand{\cZ}{\mathcal{Z}}
\newcommand{\1}{{\bf 1}_{n\not\equiv \delta}}
\newcommand{\eul}{{\rm e}}
\newcommand{\Erd}{Erd\H{o}s}
\newcommand{\cupdot}{\mathbin{\mathaccent\cdot\cup}}
\newcommand{\whp}{whp }
\newcommand{\bX}{\mathcal{X}}
\newcommand{\bV}{\mathcal{V}}
\newcommand{\ordsubs}[2]{(#1)_{#2}}
\newcommand{\unordsubs}[2]{\binom{#1}{#2}}
\newcommand{\ordelement}[2]{\overrightarrow{\mathbf{#1}}\left({#2}\right)}
\newcommand{\ordered}[1]{\overrightarrow{\mathbf{#1}}}
\newcommand{\reversed}[1]{\overleftarrow{\mathbf{#1}}}
\newcommand{\weighting}[1]{\mathbf{#1}}
\newcommand{\weightel}[2]{\mathbf{#1}\left({#2}\right)}
\newcommand{\unord}[1]{\mathbf{#1}}
\newcommand{\ordscript}[2]{\ordered{{#1}}_{{#2}}}
\newcommand{\revscript}[2]{\reversed{{#1}}_{{#2}}}

\newcommand{\doublesquig}{%
  \mathrel{%
    \vcenter{\offinterlineskip
      \ialign{##\cr$\rightsquigarrow$\cr\noalign{\kern-1.5pt}$\rightsquigarrow$\cr}%
    }%
  }%
}

\newcommand{\defn}{\emph}

\newcommand\restrict[1]{\raisebox{-.5ex}{$|$}_{#1}}

\newcommand{\probfc}[1]{\mathrm{\mathbb{P}}_{F^{*}}\left[#1\right]}
\newcommand{\probd}[1]{\mathrm{\mathbb{P}}_{D}\left[#1\right]}
\newcommand{\probf}[1]{\mathrm{\mathbb{P}}_{F}\left[#1\right]}
\newcommand{\prob}[1]{\mathrm{\mathbb{P}}\left[#1\right]}
\newcommand{\probb}[1]{\mathrm{\mathbb{P}}_{b}\left[#1\right]}
\newcommand{\expn}[1]{\mathrm{\mathbb{E}}\left[#1\right]}
\newcommand{\expnb}[1]{\mathrm{\mathbb{E}}_{b}\left[#1\right]}
\newcommand{\probxj}[1]{\mathrm{\mathbb{P}}_{x(j)}\left[#1\right]}
\newcommand{\expnxj}[1]{\mathrm{\mathbb{E}}_{x(j)}\left[#1\right]}
\def\gnp{G_{n,p}}
\def\G{\mathcal{G}}
\def\lflr{\left\lfloor}
\def\rflr{\right\rfloor}
\def\lcl{\left\lceil}
\def\rcl{\right\rceil}

\newcommand{\qbinom}[2]{\binom{#1}{#2}_{\!q}}
\newcommand{\binomdim}[2]{\binom{#1}{#2}_{\!\dim}}

\newcommand{\grass}{\mathrm{Gr}}

\newcommand{\brackets}[1]{\left(#1\right)}
\def\sm{\setminus}
\newcommand{\Set}[1]{\{#1\}}
\newcommand{\set}[2]{\{#1\,:\;#2\}}
\newcommand{\krq}[2]{K^{(#1)}_{#2}}
\newcommand{\ind}[1]{$\mathbf{S}(#1)$}
\newcommand{\indcov}[1]{$(\#)_{#1}$}
\def\In{\subseteq}

%

\maketitle
\begin{abstract}
    \vspace{-0.4cm}A subgraph~$H$ of an edge-coloured graph is called rainbow if all of the edges of~$H$ have different colours. In 1989, Andersen conjectured that every proper edge-colouring of~$K_{n}$ admits a rainbow path of length $n-2$. We show that almost all optimal edge-colourings of~$K_{n}$ admit both~(i) a rainbow Hamilton path and~(ii) a rainbow cycle using all of the colours.  This result demonstrates that Andersen's Conjecture holds for almost all optimal edge-colourings of $K_n$ and answers a recent question of Ferber, Jain, and Sudakov.  Our result also has applications to the existence of transversals in random symmetric Latin squares.
\end{abstract}
\section{Introduction}
\subsection{Extremal results on rainbow colourings}
We say that a subgraph~$H$ of an edge-coloured graph is \textit{rainbow} if all of the edges of~$H$ have different colours.  An \textit{optimal edge-colouring} of a graph is a proper edge-colouring using the minimum possible number of colours.
In this paper we study the problem of finding a rainbow Hamilton path in large optimally edge-coloured complete graphs.  

The study of finding rainbow structures within edge-coloured graphs has a rich history.
For example, the problem posed by Euler on finding orthogonal $n\times n$ Latin squares can easily be seen to be equivalent to that of finding an optimal edge-colouring of the complete bipartite graph~$K_{n,n}$ which decomposes into edge-disjoint rainbow perfect matchings.
It transpires that there are optimal colourings of~$K_{n,n}$ without even a single rainbow perfect matching, if~$n$ is even. 
However, an important conjecture, often referred to as the  Ryser-Brualdi-Stein Conjecture, posits that one can always find an almost-perfect rainbow matching, as follows.

\begin{conj}[{Ryser~\cite{R67}, Brualdi-Stein~\cite{BR91,S75}}]\label{RBS}
Every optimal edge-colouring of~$K_{n,n}$ admits a rainbow matching of size~$n-1$ and, if~$n$ is odd, a rainbow perfect matching.
\end{conj}
Currently, the strongest result towards this conjecture for arbitrary optimal edge-colourings is due to Keevash, Pokrovskiy, Sudakov, and Yepremyan~\cite{KPSY20}, who showed that there is always a rainbow matching of size $n-O(\log n / \log \log n)$.
This result improved earlier bounds of Woolbright~\cite{W78}, Brouwer, de Vries, and Wieringa~\cite{BdVW78}, and Hatami and Shor~\cite{HS08}.

It is natural to search for spanning rainbow structures in the non-partite setting as well; that is, what spanning rainbow substructures can be found in properly edge-coloured complete graphs~$K_{n}$?  
It is clear that one can always find a rainbow spanning tree -- indeed, simply take the star rooted at any vertex.
Kaneko, Kano, and Suzuki~\cite{KKS02} conjectured that for $n>4$, in any proper edge-colouring of~$K_{n}$, one can find~$\lfloor n/2\rfloor$ edge-disjoint rainbow spanning trees, thus decomposing~$K_{n}$ if~$n$ is even, and almost decomposing~$K_{n}$ if~$n$ is odd.
This conjecture was recently proved approximately by Montgomery, Pokrovskiy, and Sudakov~\cite{MPS19}, who showed that in any properly edge-coloured~$K_{n}$, one can find $(1-o(1))n/2$ edge-disjoint rainbow spanning trees.

For \textit{optimal} edge-colourings, even more is known.
Note firstly that if~$n$ is even and~$K_{n}$ is optimally edge-coloured, then the colour classes form a $1$\textit{-factorization} of~$K_{n}$; that is, a decomposition of~$K_{n}$ into perfect matchings.
Throughout the paper, we will use the term $1$-factorization synonymously with an edge-colouring whose colour classes form a $1$-factorization.
It is clear that if a $1$-factorization of~$K_{n}$ exists, then~$n$ is even.
Very recently, Glock, K\"{u}hn, Montgomery, and Osthus~\cite{GKMO20} showed that for sufficiently large even~$n$, there exists a tree~$T$ on~$n$ vertices such that any $1$-factorization of~$K_{n}$ decomposes into edge-disjoint rainbow spanning trees isomorphic to~$T$, thus resolving conjectures of Brualdi and Hollingsworth~\cite{BH96}, and Constantine~\cite{C02,C05}.
See e.g.~\cite{PS18,MPS19,KKKO20} for previous work on these conjectures.

The tree~$T$ used in~\cite{GKMO20} is a path of length $n-o(n)$, together with~$o(n)$ short paths attached to it. 
Thus it might seem natural to ask if one can find a rainbow Hamilton path in any $1$-factorization of~$K_{n}$.  Note that such a path would contain all of the colours used in the 1-factorization, so it is not possible to find a rainbow Hamilton cycle in a 1-factorization of $K_n$.  However, in 1984 Maamoun and Meyniel~\cite{MM84} 
proved the existence of a $1$-factorization of~$K_{n}$ (for~$n\geq4$ being any power of~$2$) without a rainbow Hamilton path.  
Sharing parallels with Conjecture~\ref{RBS} for the non-partite setting, Andersen~\cite{A89} conjectured in 1989 that all proper edge-colourings of~$K_{n}$ admit a rainbow path which omits only one vertex.

\begin{conj}[{Andersen~\cite{A89}}]\label{Andersen}
All proper edge-colourings of~$K_{n}$ admit a rainbow path of length $n-2$.
\end{conj}

Several variations of Andersen's Conjecture have been proposed.  In 2007, Akbari, Etesami, Mahini, and Mahmoody~\cite{AEMM07} conjectured that all $1$-factorizations of~$K_{n}$ admit a Hamilton cycle whose edges collectively have at least $n-2$ colours.  They  also conjectured that all $1$-factorizations of~$K_{n}$ admit a rainbow cycle omitting only two vertices.

Although now known to be false, the following stronger form of Conjecture~\ref{Andersen} involving the `sub-Ramsey number' of the Hamilton path was proposed by Hahn~\cite{H80}.  Every (not necessarily proper) edge-colouring of~$K_n$ with at most~$n/2$ edges of each colour admits a rainbow Hamilton path.  In light of the aforementioned construction of Maamoun and Meyniel~\cite{MM84}, in 1986 Hahn and Thomassen~\cite{HT86} suggested the following slightly weaker form of Hahn's Conjecture, that all edge-colourings of~$K_{n}$ with strictly fewer than~$n/2$ edges of each colour admit a rainbow Hamilton path.  However, even this weakening of Hahn's Conjecture is false -- Pokrovskiy and Sudakov~\cite{PS19} proved the existence of such edge-colourings of $K_n$ in which the longest rainbow Hamilton path has length at most $n - \ln n / 42$.  

Andersen's Conjecture has led to a number of results, generally focussing on increasing the length of the rainbow path or cycle that one can find in an arbitrary $1$-factorization or proper edge-colouring of~$K_{n}$ (see e.g.~\cite{CL15, GM10, GM12, GRSS11}).
Alon, Pokrovskiy, and Sudakov~\cite{APS17} proved that all proper edge-colourings of~$K_{n}$ admit a rainbow path with length $n-O(n^{3/4})$, and the error bound has since been improved to $O(\sqrt{n}\cdot\log n)$ by Balogh and Molla~\cite{BM19}.
Further support for Conjecture~\ref{Andersen} and its variants was provided by Montgomery, Pokrovskiy, and Sudakov~\cite{MPS19} as well as Kim, K\" uhn, Kupavskii, and Osthus~\cite{KKKO20}, who showed that if we consider proper edge-colourings where no colour class is larger than $n/2 - o(n)$, then we can even find $n/2 -o(n)$ edge-disjoint rainbow Hamilton cycles.

\subsection{Random colourings}
It is natural to consider these problems in a probabilistic setting, that is to consider random edge-colourings as well as random Latin squares.  However, the `rigidity' of the underlying structure makes these probability spaces very challenging to analyse.  Recently significant progress was made by Kwan~\cite{K16}, who showed that almost all Latin squares contain a transversal, or equivalently, that almost all optimal edge-colourings of $K_{n,n}$ admit a rainbow perfect matching.  His analysis was carried out in a hypergraph setting, which also yields the result that almost all Steiner triple systems contain a perfect matching.  Recently, this latter result was strengthened by Ferber and Kwan~\cite{FK20}, who showed that almost all Steiner triple systems have an approximate decomposition into edge-disjoint perfect matchings.  Here we show that Hahn's original conjecture (and thus Andersen's Conjecture as well) holds for almost all 1-factorizations, answering a recent question of Ferber, Jain, and Sudakov~\cite{FJS20}.  In what follows, we say a property holds
`with high probability' if it holds with a probability that tends to~$1$ as the number of vertices~$n$ tends to infinity.

\begin{theorem}\label{mainthm}
Let~$\phi$ be a uniformly random optimal edge-colouring of~$K_{n}$.
Then with high probability,
\begin{enumerate}[(i)]
\item\label{mainthm:rainbow-path}~$\phi$ admits a rainbow Hamilton path, and
\item\label{mainthm:rainbow-cycle}~$\phi$ admits a rainbow cycle $F$ containing all of the colours.
\end{enumerate}
In particular, if~$n$ is odd, then $F$ is a rainbow Hamilton cycle.
\end{theorem}

As discussed in Section~\ref{corollary-section}, there is a well-known correspondence between rainbow 2-factors in $n$-edge-colourings of $K_n$ and transversals in symmetric Latin squares, as a transversal in a Latin square corresponds to a permutation $\sigma$ of $[n]$ such that the entries in positions $(i, \sigma(i))$ are distinct for all $i\in [n]$.  Based on this, we use Theorem~\ref{mainthm}\ref{mainthm:rainbow-cycle} to show that random symmetric Latin squares of odd order contain a \textit{Hamilton transversal} with high probability.  Here we say a transversal is \textit{Hamilton} if the underlying permutation $\sigma$ is an $n$-cycle.

\begin{cor}\label{oddsym}
Let~$n$ be an odd integer and~$\mathbf{L}$ a uniformly random symmetric~$n\times n$ Latin square.
Then with high probability~$\mathbf{L}$ contains a Hamilton transversal.
\end{cor}
Further results on random Latin squares were recently obtained by Kwan and Sudakov~\cite{KS18}, who gave estimates on the number of intercalates in a random Latin square as well as their likely discrepancy.  After the completion of the initial version of this paper, additional results on intercalates in random Latin squares were obtained by Kwan, Sah, and Sawhney~\cite{KSS21}, which, together with the results of~\cite{KS18}, resolve an old conjecture of McKay and Wanless~\cite{MW99}.  In addition, Gould and Kelly~\cite{GK21} showed that an analogue of Corollary~\ref{oddsym} also holds when $\mathbf{L}$ is a uniformly random (not necessarily symmetric) $n \times n$ Latin square, strengthening the aforementioned result of Kwan~\cite{K16}.

\section{Notation}
In this section, we collect some definitions and notation that we will use throughout the paper.

For a graph~$G$ and (not necessarily distinct) vertex sets $A,B\subseteq V(G)$, we define~$E_{G}(A,B)\coloneqq\{e=ab\in E(G)\colon a\in A, b\in B\}$.
We often simply write~$E(A,B)$ when~$G$ is clear from the context.
We define $e(A,B)\coloneqq |E(A,B)|$.
For a vertex $v\in V(G)$, we define~$\partial_{G}(v)$ to be the set of edges of~$G$ which are incident to~$v$.
For a proper colouring $\phi\colon E(G)\rightarrow\bN$ and a colour $c\in\bN$, we define $E_{c}(G)\coloneqq\{e\in E(G)\colon \phi(e)=c\}$ and say that an edge $e\in E_{c}(G)$ is a $c$\textit{-edge} of~$G$.
For a vertex $v\in V(G)$, if~$e$ is a $c$-edge in~$G$ incident to~$v$, then we say that the non-$v$ endpoint of~$e$ is the\COMMENT{Unique since~$\phi$ proper.} $c$\textit{-neighbour} of~$v$.
For a vertex~$v\in V(G)$ and three colours $c_{1}, c_{2}, c_{3}\in\bN$, we say that the $c_{3}$-neighbour of the $c_{2}$-neighbour of the $c_{1}$-neighbour of~$v$ is the \textit{end} of the $c_{1}c_{2}c_{3}$\textit{-walk starting at}~$v$, if all such edges exist.\COMMENT{Again unique since~$\phi$ is proper.}
For a set of colours $D\subseteq\bN$, we define~$N_{D}(v)\coloneqq\{w\in N_{G}(v)\colon \phi(vw)\in D\}$.
For sets $A,B\subseteq V(G)$ and a colour $c\in\bN$, we define $E_{c}(A,B)\coloneqq\{e\in E(A,B)\colon\phi(e)=c\}$.
If~$G$ is not clear from the context, we sometimes also write~$E_{G}^{c}(A,B)$.
For any subgraph $H\subseteq G$, we define $\phi(H)\coloneqq\{\phi(e)\colon e\in E(H)\}$.
For a set of colours $D\subseteq [n-1]$, let~$\cG_{D}^{\text{col}}$ be the set of pairs~$(G, \phi_{G})$, where~$G$ is a $|D|$-regular graph on a vertex set~$V$ of size~$n$, and~$\phi_{G}$ is a $1$-factorization of~$G$ with colour set~$D$.  
Often, we abuse notation and write $G\in\cG_{D}^{\text{col}}$, and in this case we let~$\phi_{G}$ denote the implicit $1$-factorization of~$G$, sometimes simply writing~$\phi$ when~$G$ is clear from the context.
For $G\in\cG_{[n-1]}^{\text{col}}$ and a set of colours $D\subseteq [n-1]$, we define the \textit{restriction of}~$G$ \textit{to}~$D$, denoted~$G|_{D}$, to be the spanning subgraph of~$G$ containing precisely those edges of~$G$ which have colour in~$D$.
Observe that $G|_{D}\in\cG_{D}^{\text{col}}$.
A subgraph $H\subseteq G\in \cG_{D}^{\text{col}}$ inherits the colours of its edges from~$G$.
Observe that uniformly randomly choosing a $1$-factorization~$\phi$ of~$K_{n}$ on vertex set~$V$ and colour set~$[n-1]$ is equivalent to uniformly randomly choosing $G\in\cG_{[n-1]}^{\text{col}}$.
For any $D\subseteq [n-1]$, $G\in\cG_{D}^{\text{col}}$, and sets $V'\subseteq V$, $D'\subseteq D$, we define $E_{V', D'}(G)\coloneqq\{e=xy\in E(G)\colon \phi_{G}(e)\in D', x,y\in V'\}$, and we define $e_{V', D'}(G)\coloneqq |E_{V',D'}(G)|$.
For a hypergraph~$\cH$, we write~$\Delta^{c}(\cH)$ to denote the maximum codegree of~$\cH$; that is, the maximum number of edges containing any two fixed vertices of~$\cH$.

For a set~$D$ of size~$n$ and a partition~$\cP$ of~$D$ into~$m$ parts, we say that~$\cP$ is \textit{equitable} to mean that all parts~$P$ of~$\cP$ satisfy $|P|\in\{\lfloor n/m\rfloor, \lceil n/m\rceil\}$, and when it does not affect the argument, we will assume all parts of an equitable partition have size precisely\COMMENT{Feels a bit cheeky}~$n/m$.
For a set $S$ and a real number $p \in [0, 1]$, a \textit{$p$-random} subset $T\subseteq S$ is a random subset in which each element of $S$ is included in $T$ independently with probability $p$.
A \textit{$\beta$-random} subgraph of a graph~$G$ is a spanning subgraph of~$G$ where the edge-set is a $\beta$-random subset of~$E(G)$.
For an event~$\cE$ in any probability space, we write~$\overline{\cE}$ to denote the complement of~$\cE$.
For real numbers $a,b,c$ such that $b>0$, we write $a=(1\pm b)c$ to mean that the inequality $(1-b)c\leq a\leq (1+b)c$ holds.
For a natural number $n\in\bN$, we define $[n]\coloneqq\{1,2,\dots, n\}$, and $[n]_{0}\coloneqq [n]\cup\{0\}$.
We write~$x\ll y$ to mean that for any $y\in (0,1]$ there exists an $x_{0}\in (0,1)$ such that for all $0<x\leq x_{0}$ the subsequent statement holds.
Hierarchies with more constants are defined similarly and should be read from the right to the left.
Constants in hierarchies will always be real numbers in $(0,1]$.
We assume large numbers to be integers if this does not affect the argument.

\section{Overview of the proof}\label{overview}

In this section, we provide an overview of the proof of Theorem~\ref{mainthm}.  In Section~\ref{sketch} we prove Theorem~\ref{mainthm} in the case when~$n$ is even assuming two key lemmas which we prove in later sections.
In particular, we assume that~$n$ is even in Sections~\ref{overview}--\ref{switching-section}, so that the optimal edge-colouring~$\phi$ we work with is a $1$-factorization of~$K_{n}$.
In Section~\ref{corollary-section} we derive Theorem~\ref{mainthm} in the case when~$n$ is odd from the case when $n$ is even.
We will also deduce Corollary~\ref{oddsym} from Theorem~\ref{mainthm}\ref{mainthm:rainbow-cycle} in Section~\ref{corollary-section}.
Throughout the proof we work with constants $\eps, \gamma, \eta,$ and $\mu$ satisfying the following hierarchy:\COMMENT{The relationship $\eps \ll \gamma$ depends in part on the Pippenger-Spencer Theorem (Theorem~\ref{hypergraph-matching-thm}).  The relationships $\gamma \ll \eta$ and $\eta \ll \mu$ could simply be written as polynomial relationships, where these polynomials are evident from our proof.}
\begin{equation}\label{constant-heirarchy}
  1/n \ll \eps \ll \gamma \ll \eta \ll \mu \ll 1.
\end{equation}
Our proof uses the absorption method as well as switching techniques.
Note that the latter is a significant difference to~\cite{K16,FK20}, which rely on the analysis of the random greedy triangle removal process, as well as modifications of arguments in~\cite{LL13,Kee18} which bound the number of Steiner triple systems.
Our main objective is to show that with high probability, in a random 1-factorization, we can find an absorbing structure inside a random subset of $\Theta(\mu n)$ reserved vertices, using a random subset of $\Theta(\mu n)$ reserved colours. A recent result~\cite[Lemma 16]{GKMO20}, based on hypergraph matchings, enables us to find a long rainbow path avoiding these reserved vertices and colours, and using our absorbing structure, we extend this path to a rainbow Hamilton path.
More precisely, we randomly `reserve' $\Theta(\mu n)$ vertices and colours and show that with high probability we can find an absorbing structure.
This absorbing structure consists of a subgraph~$G_{\text{abs}}$ containing only reserved vertices and colours and all but at most $\gamma n$ of them.
Moreover~$G_{\text{abs}}$ contains `flexible' sets of vertices and colours $V_{\mathrm{flex}}$ and $C_{\mathrm{flex}}$ each of size $\eta n$, with the following crucial property:
  \begin{enumerate}[($\dagger$)]
  \item for any pair of equal-sized subsets $X\subseteq V_{\mathrm{flex}}$ and $Y\subseteq C_{\mathrm{flex}}$ of size at most $\eta n/2$, the graph $G_{\text{abs}} - X$ contains a spanning rainbow path whose colours avoid~$Y$.\label{crucial-absorber-property}
  \end{enumerate}
  In fact, this spanning rainbow path has the same end vertices, regardless of the choice of $X$ and $Y$.  Given this absorbing structure, we find a rainbow Hamilton path in the following three steps:
\begin{enumerate}
\item \textbf{Long path step:}
  Apply~\cite[Lemma 16]{GKMO20} to obtain a long rainbow path $P_1$ containing only non-reserved vertices and colours.  Moreover, $P_1$ contains all but at most $\gamma n$ of them.
\item\label{covering-step} \textbf{Covering step:}
  `Cover' the vertices and colours not in $G_{\text{abs}}$ or $P_1$ using the flexible sets, by greedily constructing a path $P_2$ containing them as well as sets $X\subseteq V_{\mathrm{flex}}$ and $Y\subseteq C_{\mathrm{flex}}$ of size at most $\eta n / 2$.
\item \textbf{Absorbing step:}
  `Absorb' the remaining vertices and colours, by letting $P_3$ be the rainbow path guaranteed by~\ref{crucial-absorber-property}.
\end{enumerate}
In the covering step, we can ensure that $P_2$ shares one end with $P_1$ and one end with $P_3$ so that $P_1\cup P_2 \cup P_3$ is a rainbow Hamilton path, as desired.  These steps are fairly straightforward, so the majority of the paper is devoted to building the absorbing structure, that is, the subgraph $G_{\mathrm{abs}}$ which satisfies \ref{crucial-absorber-property} with respect to `flexible' sets $V_{\mathrm{flex}}$ and $C_{\mathrm{flex}}$.  This argument is split into two parts.  Lemma~\ref{main-absorber-lemma}, proved in Section~\ref{absorption-section}, asserts that, subject to some quasirandomness conditions, we can build our absorbing structure using our randomly reserved vertices and colours; Lemma~\ref{main-switching-lemma}, proved in Section~\ref{switching-section}, asserts that a typical $1$-factorization of $K_n$ has these quasirandom properties.

\subsection{Absorption}\label{absorption-details-section}

To design our absorbing structure, we employ a strategy sometimes called `distributed absorption', first introduced by Montgomery~\cite{M18}.  The details of this are presented in Section~\ref{sketch}, but we provide an overview now.  Our absorbing structure consists of many `gadgets' pieced together in a particular way.  In particular, for a vertex $v$ and colour $c$, a \textit{$(v, c)$-absorber} (see Definition~\ref{def:absorbing-gadget} and Figure~\ref{(v,c)-absorber-fig}) is a small subgraph containing both $v$ and an edge coloured $c$, with the following property: It contains a rainbow path which is spanning and which uses one of each colour assigned to its edges, and it also contains a rainbow path which includes all of its vertices except $v$ and an edge of every one if its colours except $c$; moreover, these paths have the same end vertices.  We refer to the former path as the \textit{$(v, c)$-absorbing path} and the latter as the \textit{$(v, c)$-avoiding path} (again, see Definition~\ref{def:absorbing-gadget}).  

\begin{figure}
  \centering
  \begin{tikzpicture}[scale = 1.5]
    \tikzstyle{vtx}=[draw, fill, circle, scale=.5];
    \tikzstyle{link}=[snake=zigzag];

    \node[label=above:{\Large $v$}, vtx] (v) at (0, 0) {};

    \node[vtx] (t1) at ($(v) + (30:1)$) {};
    \node[vtx] (b1) at ($(v) + (-30:1)$) {};

    \node[vtx] (t2) at ($(t1) + (1, 0)$) {};
    \node[vtx] (b2) at ($(b1) + (1, 0)$) {};

    \node[vtx] (t3) at ($(t2) + (1, 0)$) {};
    \node[vtx] (b3) at ($(b2) + (1, 0)$) {};
    \node[above] at (0.4,0.21) {$1$};
    \node[below] at (0.4,-0.21) {$2$};
    \node[right] at (0.82,0) {$3$};
    \node[left] at (1.9,0) {$3$};
    \node[above] at (2.35,0.47) {$1$};
    \node[below] at (2.35,-0.47) {$2$};
    \node[right] at (2.84,0) {$c$};

    \draw[ultra thick, OwlRed] (v) -- (t1);
    \draw[ultra thick, OwlGreen] (t1) -- (b1);
    \draw[ultra thick, OwlBlue] (b1)-- (v);

    \draw[ultra thick, OwlGreen] (t2) -- (b2);
    \draw[ultra thick, OwlRed] (t2) -- (t3);
    \draw[ultra thick, OwlBlue] (b2) -- (b3);
    \draw[ultra thick, OwlViolet] (t3) -- (b3);

    \draw[link] (t1) -- (b2);
    \draw[link] (t2) -- (b3);
  \end{tikzpicture}
  \caption{A $(v, c)$-absorber, where $\phi(e_i) = i$.  The paths $P_1$ and $P_2$ are drawn as zigzags.}
  \label{(v,c)-absorber-fig}
\end{figure}
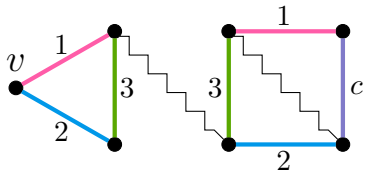


We build our absorbing structure out of $(v, c)$-absorbers, along with short rainbow paths linking them together, using an auxiliary bipartite graph $H$ as a template (see Definition~\ref{def:absorber} and Figure~\ref{absorber-fig}), where one part of $H$ is a set of vertices (including $V_{\mathrm{flex}}$) and the other part is a set of colours (including $C_{\mathrm{flex}}$).  For every edge $vc\in E(H)$, we will have a $(v, c)$-absorber in the absorbing structure.  When proving~\ref{crucial-absorber-property}, if $v$ or $c$ is in $X$ or $Y$, then the spanning rainbow path in $G_{\text{abs}}-X$ contains the $(v, c)$-avoiding path, and otherwise it may contain the $(v, c)$-absorbing path.  More precisely, we find a perfect matching of $H - (X \cup Y)$, and we use the $(v, c)$-absorbing path for every matched pair of vertex $v$ and colour $c$.

\begin{figure}
  \centering
    \begin{tikzpicture}[scale = 1.5]
    \tikzstyle{vtx}=[draw, fill, circle, scale = .5];
    \tikzstyle{link}=[snake=zigzag];

    \node[label=above:{\Large $v_1$}, vtx] (v) at (0, 0) {};

    \node[vtx] (t1) at ($(v) + (30:1)$) {};
    \node[vtx] (b1) at ($(v) + (-30:1)$) {};

    \node[vtx] (t2) at ($(t1) + (1, 0)$) {};
    \node[vtx] (b2) at ($(b1) + (1, 0)$) {};

    \node[vtx] (t3) at ($(t2) + (1, 0)$) {};
    \node[vtx] (b3) at ($(b2) + (1, 0)$) {};

    \draw (v) -- (t1);
    \draw (t1) -- (b1);
    \draw (b1)-- (v);

    \draw (t2) -- (b2);
    \draw (t2) -- (t3);
    \draw (b2) -- (b3);
    \draw (t3) -- (b3);
    \node[left] at (-2.83,0) {$c_{2}$};

    \draw[link] (t1) -- (b2);
    \draw[link] (t2) -- (b3);

    \node[vtx] (t'1) at ($(v) + (150:1)$) {};
    \node[vtx] (b'1) at ($(v) + (-150:1)$) {};

    \node[vtx] (t'2) at ($(t'1) + (-1, 0)$) {};
    \node[vtx] (b'2) at ($(b'1) + (-1, 0)$) {};

    \node[vtx] (t'3) at ($(t'2) + (-1, 0)$) {};
    \node[vtx] (b'3) at ($(b'2) + (-1, 0)$) {};

    \draw (v) -- (t'1);
    \draw (t'1) -- (b'1);
    \draw (b'1) -- (v);

    \draw (t'2) -- (b'2);
    \draw (t'2) -- (t'3);
    \draw (b'2) -- (b'3);
    \draw (t'3) -- (b'3);
    \node[left] at (-2.83,-1.5) {$c_{2}$};

    \draw[link] (b'1) -- (t'2);
    \draw[link] (b'2) -- (t'3);

    \node[label=above:{\Large $v_2$}, vtx] (bv) at (0, -1.5) {};

    \node[vtx] (bt1) at ($(bv) + (30:1)$) {};
    \node[vtx] (bb1) at ($(bv) + (-30:1)$) {};

    \node[vtx] (bt2) at ($(bt1) + (1, 0)$) {};
    \node[vtx] (bb2) at ($(bb1) + (1, 0)$) {};

    \node[vtx] (bt3) at ($(bt2) + (1, 0)$) {};
    \node[vtx] (bb3) at ($(bb2) + (1, 0)$) {};

    \draw (bv) -- (bt1);
    \draw (bt1) -- (bb1);
    \draw (bb1)-- (bv);

    \draw (bt2) -- (bb2);
    \draw (bt2) -- (bt3);
    \draw (bb2) -- (bb3);
    \draw (bt3) -- (bb3);
    \node[right] at (2.83,0) {$c_{1}$};

    \draw[link] (bt1) -- (bb2);
    \draw[link] (bt2) -- (bb3);

    \node[vtx] (bt'1) at ($(bv) + (150:1)$) {};
    \node[vtx] (bb'1) at ($(bv) + (-150:1)$) {};

    \node[vtx] (bt'2) at ($(bt'1) + (-1, 0)$) {};
    \node[vtx] (bb'2) at ($(bb'1) + (-1, 0)$) {};

    \node[vtx] (bt'3) at ($(bt'2) + (-1, 0)$) {};
    \node[vtx] (bb'3) at ($(bb'2) + (-1, 0)$) {};

    \draw (bv) -- (bt'1);
    \draw (bt'1) -- (bb'1);
    \draw (bb'1) -- (bv);

    \draw (bt'2) -- (bb'2);
    \draw (bt'2) -- (bt'3);
    \draw (bb'2) -- (bb'3);
    \draw (bt'3) -- (bb'3);
    \node[right] at (2.83,-1.5) {$c_{1}$};

    \draw[link] (bb'1) -- (bt'2);
    \draw[link] (bb'2) -- (bt'3);

    \draw[decorate, decoration=zigzag] (b'3) .. controls ($(b'3) + (0, -.5)$) and ($(bt'1) + (0, .5)$) .. (bt'1);
    \draw[decorate, decoration=zigzag] (b1) .. controls ($(b1) + (0, -.5)$) and ($(bt3) + (0, .5)$) .. (bt3);
    \draw[decorate, decoration=zigzag] (t'1) .. controls ($(t'1) + (0, .5)$) and ($(t3) + (0, .5)$) .. (t3);
  \end{tikzpicture}
  \caption{An $H$-absorber where $H\cong K_{2,2}$ with bipartition $\left(\{v_1, v_2\}, \{c_1, c_2\}\right)$.}
  \label{absorber-fig}
\end{figure}
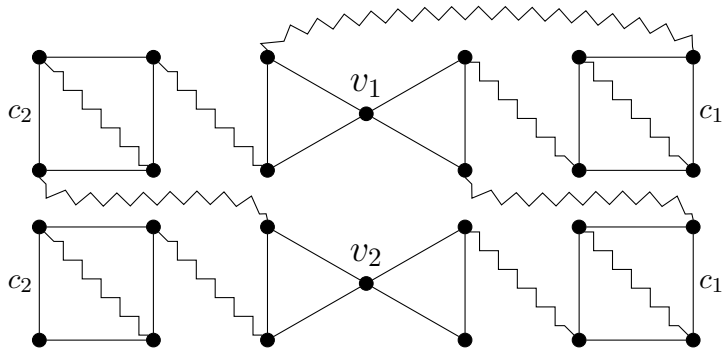

A naive approach would be to use the complete bipartite graph with parts $V_{\mathrm{flex}}$ and $C_{\mathrm{flex}}$ as our template $H$; however, this would require too many absorbing gadgets.  Instead, we choose a much sparser template graph $H$ that is \textit{robustly matchable} with respect to $V_{\mathrm{flex}}$ and $C_{\mathrm{flex}}$ (see Definition~\ref{def:rmbg}); we use a result of Montgomery~\cite[Lemma~10.7]{M18} to construct a robustly matchable bipartite graph with maximum degree $O(1)$.  
Thus, we only need $\Theta(\eta n)$ absorbing gadets to build an absorbing structure satisfying \ref{crucial-absorber-property}.  
Using that $\eta \ll \mu$, we can build such an absorbing structure inside the random subset of $\Theta(\mu n)$ vertices and $\Theta(\mu n)$ colours (see Lemma~\ref{greedy-absorber-lemma}).  However, our absorbing structure needs to contain all but at most $\gamma n$ of the reserved vertices and colours.  To that end, we attach a long rainbow path using almost all of the remaining reserved vertices and colours that we call a \textit{tail} (see Definition~\ref{def:absorber}); this is accomplished using the semi-random method, implemented via hypergraph matchings results (see Lemma~\ref{linking-lemma}).  We use a similar approach in the long path step.

\subsection{Analysing a random $1$-factorization of $K_n$}

To build the absorbing structure described in Section~\ref{absorption-details-section}, we need to show that a typical 1-factorization of $K_n$ satisfies some quasirandom properties.  We call these properties \textit{local edge-resilience} and \textit{robust gadget-resilience} (see Definitions~\ref{def:edge-resilient} and \ref{spread}), and we prove they hold for typical 1-factorizations in Lemma~\ref{main-switching-lemma}.  Standard arguments can be used to show that these properties hold with high probability for a (not necessarily proper) edge-colouring of $K_n$ where each edge is assigned one of $n$ colours independently and uniformly at random; however, it is much more challenging to prove this for a random $1$-factorization.  
We prove Lemma~\ref{main-switching-lemma} using a `coloured version' of switching arguments that are commonly used to study random regular graphs.
Unfortunately, $1$-factorizations of the complete graph~$K_{n}$ are `rigid' structures, in the sense that it is difficult to make local changes without global ramifications on such a $1$-factorization.
Thus, instead of analysing switchings between graphs in~$\cG_{[n-1]}^{\text{col}}$, we will analyse switchings between graphs in~$\cG_{D}^{\text{col}}$ for appropriately chosen $D\subsetneq [n-1]$.
In the setting of random Latin squares, this approach was used by McKay and Wanless~\cite{MW99} and further developed by Kwan and Sudakov~\cite{KS18}, and we build on their ideas.

We use results on the number of $1$-factorizations of dense regular graphs due to Kahn and Lov\'{a}sz (see Theorem~\ref{linlur}) and Ferber, Jain, and Sudakov (see Theorem~\ref{Ferb}) to study the number of completions of a graph $H\in\cG_{D}^{\text{col}}$ to a graph $G\in\cG_{[n-1]}^{\text{col}}$, and we use this information to compare the probability space corresponding to a uniform random choice of $\mathbf{H}\in\cG_{D}^{\text{col}}$, with the probability space corresponding to a uniform random choice of $\mathbf{G}\in\cG_{[n-1]}^{\text{col}}$.
In particular, if a uniformly random $\mathbf{H} \in \cG^{\text{col}}_D$ is extended uniformly at random to obtain a colouring $\mathbf{H'} \in \cG^{\text{col}}_{[n - 1]}$, then $\mathbf{H'}$ is not chosen uniformly at random from $\cG^{\text{col}}_{[n - 1]}$, since different choices of $H \in \cG_{D}^{\text{col}}$ have different numbers of extensions; however, $\mathbf{H'}$ can be compared to a uniformly random $\mathbf{G} \in \cG^{\text{col}}_{[n - 1]}$ as follows (see also Corollary~\ref{wf}).  For an absolute constant $C$, and for each $K \in \cG^{\text{col}}_{[n-1]}$,
\begin{equation*}
    \Prob{\mathbf{G} = K} = \Prob{\mathbf{H'} = K}\cdot \exp(\pm n^{2 - 1/C}).
\end{equation*}
Therefore, any property that holds for $\mathbf{H}$ with probability at least $1 - \exp(-\Omega(n^2))$ also holds with high probability for $\mathbf{G}|_D$.  Our switching arguments yield local edge-resilience and robust gadget-resilience for $\mathbf{H}$ with high enough probability (see Lemmas~\ref{localedge} and \ref{masterswitch}) to apply Corollary~\ref{wf}.

\section{Proving Theorem~\ref{mainthm}}\label{sketch}

In this section, let $\phi$ be a 1-factorization of $K_n$ with vertex set $V$ and colour set $C=[n-1]$.  We first present the details of our absorbing structure, and in Section~\ref{proof-section}, we prove Theorem~\ref{mainthm} (in the case when~$n$ is even) subject to its existence.
We begin by introducing our absorbing gadgets in the following definition (see also Figure~\ref{(v,c)-absorber-fig}).

\begin{defin}\label{def:absorbing-gadget}
  For every $v\in V$ and $c\in C$, a \textit{$(v, c)$-absorbing gadget} is a subgraph of $K_n$ of the form $A = T\cup Q$ such that the following holds:
  \begin{itemize}
  \item $T \cong K_3$ and $Q\cong C_4$, 
  \item $T$ and $Q$ are vertex-disjoint,
  \item $v\in V(T)$ and there is a unique edge $e \in E(Q)$ such that $\phi(e) = c$, 
  \item if $e_1,e_2 \in E(T)$ are the edges incident to $v$, then there is matching $\{e'_1, e'_2\}$ in $Q$ not containing $e$ such that $\phi(e_i) = \phi(e'_i)$ for $i\in\{1, 2\}$,
  \item if $e_3 \in E(T)$ is the edge not incident to $v$, then there is an edge $e'_3 \in E(Q)$ such that $\{e'_3, e\}$ is a matching in $Q$ and $\phi(e_3) = \phi(e'_3) \neq c$.
  \end{itemize}
 In this case, a pair~$P_{1}$,~$P_{2}$ of paths \textit{completes} the $(v, c)$-absorbing gadget $A = T \cup Q$ if
  \begin{itemize}
  \item the ends of $P_1$ are non-adjacent vertices in $Q$,
  \item one end of $P_2$ is in $Q$ but not incident to $e$ and the other end of $P_2$ is in $V(T)\setminus\{v\}$,
  \item $P_1$ and $P_2$ are vertex-disjoint and both $P_1$ and $P_2$ are internally vertex-disjoint from $A$,
  \item $P_1 \cup P_2$ is rainbow, and
  \item $\phi(P_1 \cup P_2) \cap \phi(A) = \varnothing$,
  \end{itemize}
  and we say $A' \coloneqq A \cup P_1 \cup P_2$ is a \textit{$(v, c)$-absorber}.  We also define the following.
  \begin{itemize}
  \item The path $P$ with edge-set $E(P_1)\cup E(P_2) \cup \{e'_1, e'_2, e_3\}$ is the \textit{$(v, c)$-avoiding} path in $A'$, and the path $P'$ with edge-set $E(P_1) \cup E(P_2) \cup \{e_1, e_2, e'_3, e\}$ is the \textit{$(v, c)$-absorbing} path in $A'$.
  \item A vertex in $V(A)\setminus \{v\}$, a colour in $\phi(A) \setminus \{c\}$, or an edge in $E(A)$ is \textit{used} by the $(v, c)$-absorbing gadget $A$.
  \end{itemize}
\end{defin}


It is convenient for us to distinguish between a $(v, c)$-absorbing gadget and a $(v, c)$-absorber, because when we build our absorbing structure, we first find a $(v, c)$-absorbing gadget for every $vc \in E(H)$ and then find the paths completing each absorbing gadget.  We also find an additional set of paths that `links' the gadgets together, as in the following definition.

\begin{defin}\label{def:absorber}
  Let $H$ be a bipartite graph with bipartition $(V', C')$ where $V'\subseteq V$ and $C'\subseteq C$, and suppose $\mathcal A = \{A_{v, c} : vc \in E(H)\}$ where $A_{v, c}$ is a $(v, c)$-absorbing gadget.
  \begin{itemize}
  \item We say $\mathcal A$ \textit{satisfies} $H$ if whenever $A_{v,c}, A_{v',c'}\in\cA$ for some $(v,c)\neq (v',c')$, no vertex in $V(A_{v,c})$ or colour in $\phi(A_{v, c})$ is used by $A_{v', c'}$.
  \item If $\mathcal P$ is a collection of vertex-disjoint paths of length 4, then we say $\mathcal P$ \textit{completes} $\mathcal A$ if the following holds:
    \begin{itemize}
    \item $\bigcup_{P\in\cP}P$ is rainbow,
    \item no colour that is either in $C'$ or is used by a $(v, c)$-absorbing gadget $A_{v, c}\in \mathcal A$ appears in a path $P\in \mathcal P$,
    \item no vertex that is either in $V'$ or is used by a $(v, c)$-absorbing gadget $A_{v,c}\in\mathcal A$ is an internal vertex of a path $P \in \mathcal P$,
    \item for every $(v, c)$-absorbing gadget $A_{v, c}\in\mathcal A$ there is a pair of paths $P_1, P_2\in \mathcal P$ such that $P_1$ and $P_2$ complete $A_{v, c}$ to a $(v, c)$-absorber $A'_{v, c}$, and
    \item the graph $\left(\bigcup_{A\in\mathcal A}A\cup\bigcup_{P\in\mathcal P}P\right) \setminus V'$ is connected and has maximum degree three, and $\mathcal P$ is minimal subject to this property.
    \end{itemize}

  \item We say $(\mathcal A, \mathcal P)$ is an \textit{$H$-absorber} if $\mathcal A$ satisfies $H$ and is completed by $\mathcal P$.  See Figure~\ref{absorber-fig}.
  \item We say a rainbow path $T$ is a \textit{tail} of an $H$-absorber $(\mathcal A, \mathcal P)$ if
    \begin{itemize}
    \item one of the ends of $T$, say $x$, is in a $(v, c)$-absorbing gadget $A_{v,c} \in \mathcal A$ such that $x \neq v$,
    \item $V(T)\cap V(A) \subseteq\{x\}$ for all $A\in \mathcal A$ and $V(T) \cap V(P) = \varnothing$ for all $P \in \mathcal P$, and
    \item $\phi(T) \cap \phi(A) = \varnothing$ for all $A \in \mathcal A$ and $\phi(T) \cap \phi(P) = \varnothing$ for all $P \in \mathcal P$.
    \end{itemize}
  \item For every matching $M$ in $H$, we define the \textit{path absorbing $M$} in $(\mathcal A, \mathcal P, T)$ to be the rainbow path $P$ such that
    \begin{itemize}
    \item $P$ contains $\bigcup_{P'\in \mathcal P}P' \cup T$ and
    \item for every $vc \in E(H)$, if $vc\in E(M)$, then $P$ contains the $(v, c)$-absorbing path in the $(v, c)$-absorber $A'_{v, c}$ and $P$ contains the $(v, c)$-avoiding path otherwise (that is, $V(P)\cap V' = V(M)\cap V'$ and $\phi(P) \cap C' = V(M) \cap C'$).
    \end{itemize}

  \end{itemize}
\end{defin}

Note that if~$\cP$ completes~$\cA$, then some of the paths in~$\cP$ will complete absorbing gadgets in~$\cA$ to absorbers, while the remaining set of paths $\cP'\subseteq\cP$ will be used to connect all the absorbing gadgets in~$\cA$.
More precisely, there is an enumeration $A_{1},\dots,A_{|\cA|}$ of~$\cA$ and an enumeration $P_{1},\dots,P_{|\cA|-1}$ of~$\cP'$ such that each~$P_{i}$ joins~$A_{i}$ to~$A_{i+1}$.
In particular, for each $i\in[|\cA|]\setminus\{1,|\cA|\}$, each vertex in~$A_{i}\setminus V'$ is the endpoint of precisely one path in~$\cP$ (and thus has degree three in $\bigcup_{A\in\cA}A\cup\bigcup_{P\in\cP}P$), while both~$A_{1}\setminus V'$ and~$A_{|\cA|}\setminus V'$ contain precisely one vertex which is not the endpoint of some path in~$\cP$ (and thus these two vertices have degree two in $\bigcup_{A\in\cA}A\cup\bigcup_{P\in\cP}P$).
Any tail~$T$ of an $H$-absorber~$(\cA,\cP)$ has to start at one of these two vertices.
Altogether this means that, given any matching~$M$ of~$H$, the path absorbing~$M$ in~$(\cA,\cP,T)$ in the definition above actually exists.

Our absorbing structure is essentially an $H$-absorber $(\mathcal A, \mathcal P)$ with a tail $T$ and flexible sets $V_{\mathrm{flex}}, C_{\mathrm{flex}}\subseteq V(H)$ for an appropriately chosen template $H$.  If $H - (X\cup Y)$ has a perfect matching $M$, then the path absorbing $M$ in $(\mathcal A, \mathcal P, T)$ satisfies~\ref{crucial-absorber-property}.  This fact motivates the property of $H$ that we need in the next definition.


\begin{defin}\label{def:rmbg}
  Let $H$ be a bipartite graph with bipartition $(A, B)$ such that $|A| = |B|$, and let $A'\subseteq A$ and $B'\subseteq B$ such that $|A'| = |B'|$.
  \begin{itemize}
  \item We say $H$ is \textit{robustly matchable} with respect to $A'$ and $B'$ if for every pair of sets $X$ and $Y$ where $X\subseteq A'$, $Y\subseteq B'$, and $|X| = |Y| \leq |A'| / 2$, there is a perfect matching in $H - (X\cup Y)$.
  \item In this case, we say $A'$ and $B'$ are \textit{flexible} and $A\setminus A'$ and \textit{$B\setminus B'$} are \textit{buffer sets}.
  \end{itemize}
\end{defin}
This concept was first introduced by Montgomery~\cite{M18}. 
If $H$ is robustly matchable with respect to $V_{\mathrm{flex}}$ and $C_{\mathrm{flex}}$, then an $H$-absorber $(\mathcal A, \mathcal P)$ with tail $T$ satisfies~\ref{crucial-absorber-property}.  The last property of our absorbing structure that we need is that the flexible sets allow us to execute the covering step, which we capture in the following definition.

\begin{defin}
  Let $V_{\mathrm{flex}} \subseteq V$, let $C_{\mathrm{flex}}\subseteq C$, and let $G_{\mathrm{flex}}$ be a spanning subgraph of $K_n$.
  \begin{itemize}
  \item If $u,v\in V$ and $c\in C$, and $P\subseteq G_{\mathrm{flex}}$ is a rainbow path of length four such that
    \begin{itemize}
    \item $u$ and $v$ are the ends of $P$,
    \item $u', w, v' \in V_{\mathrm{flex}}$, where $uu', u'w, wv', vv' \in E(P)$,
    \item $\phi(uu'), \phi(wv'), \phi(vv') \in C_{\mathrm{flex}}$, and
    \item $\phi(u'w) = c$,
    \end{itemize}
    then $P$ is a $(V_{\mathrm{flex}}, C_{\mathrm{flex}}, G_{\mathrm{flex}})$-\textit{cover} of $u, v,$ and $c$.
  \item If $P$ is a rainbow path such that $P = \bigcup_{i = 1}^k P_i$ where $P_i$ is a $(V_{\mathrm{flex}}, C_{\mathrm{flex}}, G_{\mathrm{flex}})$-cover of $v_i, v_{i+1},$ and $c_i$, then $P$ is a $(V_{\mathrm{flex}}, C_{\mathrm{flex}}, G_{\mathrm{flex}})$-\textit{cover} of $\{v_1, \dots, v_{k+1}\}$ and $\{c_1, \dots, c_k\}$.
  \item If $H$ is a regular bipartite graph with bipartition $(V', C')$ where $V'\subseteq V$ and $C'\subseteq C$ such that
    \begin{itemize}
    \item $H$ is robustly matchable with respect to $V_{\mathrm{flex}}$ and $C_{\mathrm{flex}}$ where $|V_{\mathrm{flex}}|, |C_{\mathrm{flex}}| \geq \delta n$, and
    \item for every $u,v \in V$ and $c\in C$, there are at least $\delta n^2$ $(V_{\mathrm{flex}}, C_{\mathrm{flex}}, G_{\mathrm{flex}})$-covers of $u,v$, and $c$, 
    \end{itemize}
    then $H$ is a \textit{$\delta$-absorbing template with flexible sets $(V_{\mathrm{flex}}, C_{\mathrm{flex}}, G_{\mathrm{flex}})$}.
  \item If $(\mathcal A, \mathcal P)$ is an $H$-absorber where $H$ is a $\delta$-absorbing template and $T$ is a tail for $(\mathcal A, \mathcal P)$, then $(\mathcal A, \mathcal P, T, H)$ is a \textit{$\delta$-absorber}.
  \end{itemize}
\end{defin}

A $36\gamma$-absorber has the properties we need to execute both the covering step and the absorbing step, which we make formal with the next proposition.

First, we introduce the following convenient convention.
Given $V''\subseteq V'\subseteq V$ and $C''\subseteq C'\subseteq C$, we say that~$(V'',C'')$ is \textit{contained} in~$(V',C')$ with $\delta$\textit{-bounded remainder} if $V''\subseteq V'$, $C''\subseteq C'$, and $|V'\setminus V''|,|C'\setminus C''|\leq\delta n$.
If $G$ is a spanning subgraph of $K_n$, $V' \subseteq V$, and $C'\subseteq C$, then we say a graph $G'$ is \textit{contained} in $(V', C', G)$ with \textit{$\delta$-bounded remainder} if~$(V(G'),\phi(G'))$ is contained in~$(V',C')$ with $\delta$-bounded remainder and $G'\subseteq G$.
\begin{prop}\label{main-absorbing-proposition}
  If $(\mathcal A, \mathcal P, T, H)$ is a $\delta$-absorber and $P'$ is a rainbow path contained in $(V\setminus V', C\setminus C', G')$ with $\delta / 18$-bounded remainder where
  \begin{itemize}
  \item $V' = \bigcup_{A\in\mathcal A}V(A) \cup \bigcup_{P\in \mathcal P}V(P) \cup V(T)$
  \item $C' = \bigcup_{A\in\mathcal A}\phi(A) \cup \bigcup_{P\in \mathcal P}\phi(P) \cup \phi(T)$, and
  \item $G'$ is the complement of $\bigcup_{A\in\mathcal A}A \cup \bigcup_{P\in \mathcal P}P \cup T$,
  \end{itemize}
  then there is both a rainbow Hamilton path containing $P'$ and a rainbow cycle containing $P'$ and all of the colours in $C$.
\end{prop}
\begin{proof}
  Order the colours in $C\setminus(\phi(P')\cup C')$ as $c_1, \dots, c_k$, and note that $k \leq \delta n / 18$.  Order the vertices in $V\setminus (V(P') \cup V')$ as $v_1, \dots, v_\ell$.\COMMENT{
  \begin{align*}
    |V(P')| &= |\phi(P')| + 1,\\
    |V'| &= |C'| + 1,\\
    |V| &= |C| + 1,\\
    \ell &= |V| - |V'| - |V(P')|,\ \mathrm{and}\\
    k &= |C| - |C'| - |\phi(P')|,
  \end{align*}}
  Using that~$H$ is regular, it is easy to see that $|V'|=|C'|+1$, and thus $\ell = k - 1$.  Let $v_0$ and $v'_0$ be the ends of $P'$, let $u$ be the end of $T$ not in $V(A)$ for any $A \in \mathcal A$, and let $u'$ be the unique vertex in $\bigcup_{A\in\mathcal A}V(A)\setminus V(T)$ of degree two in $\bigcup_{A\in\mathcal A}A \cup \bigcup_{P\in \mathcal P}P$.
  Let $(V_{\mathrm{flex}}, C_{\mathrm{flex}}, G_{\mathrm{flex}})$ be the flexible sets of $H$.

    First we show that there is a rainbow Hamilton path containing $P'$.
  We claim there is a $(V_{\mathrm{flex}}, C_{\mathrm{flex}}, G_{\mathrm{flex}})$-cover $P''$ of $\{v_0, \dots, v_k\}$ and $\{c_1, \dots, c_k\}$, where $v_k \coloneqq u$.
  Suppose for $j \in [k - 1]$ and $i <  j$ that $P_i$ is a $(V_{\mathrm{flex}}, C_{\mathrm{flex}}, G_{\mathrm{flex}})$-cover of $v_i, v_{i + 1}$, and $c_{i + 1}$ such that $\bigcup_{i < j} P_i$ is a rainbow path.  We show that there exists a $(V_{\mathrm{flex}}, C_{\mathrm{flex}}, G_{\mathrm{flex}})$-cover $P_j$ of $v_j, v_{j + 1}$, and $c_{j + 1}$ that is internally-vertex- and colour-disjoint from $\bigcup_{i < j}P_i$, which implies that $\bigcup_{i \leq j} P_i$ is a rainbow path, and thus we can choose the path $P''$ greedily, proving the claim.  Since each vertex in $V_{\mathrm{flex}}$ and each colour in $C_{\mathrm{flex}}$ is contained in at most $3n$ $(V_{\mathrm{flex}}, C_{\mathrm{flex}}, G_{\mathrm{flex}})$-covers of $v_j, v_{j + 1}$, and $c_{j + 1}$, and since $H$ is a $\delta$-absorbing template, there are at least $\delta n^2 - 18n\cdot j$ $(V_{\mathrm{flex}}, C_{\mathrm{flex}}, G_{\mathrm{flex}})$-covers of $v_j, v_{j + 1}$, and $c_{j + 1}$ not containing a vertex or colour from $\bigcup_{i < j}P_i$.  Thus, since $j < k \leq \delta n / 18$, there exists a $(V_{\mathrm{flex}}, C_{\mathrm{flex}}, G_{\mathrm{flex}})$-cover $P_j$ of $v_j, v_{j + 1}$, and $c_{j + 1}$ such that $\bigcup_{i\leq j} P_i$ is a rainbow path, as desired, and consequently we can choose the path $P''$ greedily, as claimed.

  Now let $X \coloneqq V(P'') \cap V_{\mathrm{flex}}$, and let $Y \coloneqq \phi(P'')\cap C_{\mathrm{flex}}$.  Since $|X| = |Y| = 3k \leq |V_{\mathrm{flex}}|/2$ and $H$ is robustly matchable with respect to $V_{\mathrm{flex}}$ and $C_{\mathrm{flex}}$, there is a perfect matching $M$ in $H - (X\cup Y)$.  Let $P'''$ be the path absorbing $M$ in $(\mathcal A, \mathcal P, T)$.  Then $P'\cup P'' \cup P'''$ is a rainbow Hamilton path, as desired.  
  
  Now we show that there is a rainbow cycle containing $P'$ and all of the colours in $C$.  By the same argument as before, there is a $(V_{\mathrm{flex}}, C_{\mathrm{flex}}, G_{\mathrm{flex}})$-cover $P''_1$ of $\{v_0, \dots, v_{\ell-1}, u\}$ and $\{c_1, \dots, c_{k-1}\}$ as well as a $(V_{\mathrm{flex}}, C_{\mathrm{flex}}, G_{\mathrm{flex}})$-cover $P''_2$ of $v'_0$, $u'$, and $c_k$ such that $P''_1$ and $P''_2$ are vertex-and colour-disjoint.  Letting $X \coloneqq V(P''_1\cup P''_2) \cap V_{\mathrm{flex}}$ and $Y\coloneqq \phi(P''_1\cup P''_2) \cap C_{\mathrm{flex}}$, letting $M$ be a perfect matching in $H - (X\cup Y)$ and $P'''$ be the path absorbing $M$ in $(\mathcal A, \mathcal P, T)$ as before, $P'\cup P''_1\cup P''_2\cup P'''$ is a rainbow cycle using all the colours in $C$, as desired.
\end{proof}

\subsection{The proof of Theorem~\ref{mainthm} when~$n$ is even}\label{proof-section}

In this subsection, we prove the~$n$ even case of Theorem~\ref{mainthm} subject to two lemmas, Lemmas~\ref{main-switching-lemma} and~\ref{main-absorber-lemma}, which we prove in Sections~\ref{switching-section} and~\ref{absorption-section}, respectively.
The first of these lemmas, Lemma~\ref{main-switching-lemma}, states that almost all 1-factorizations have two key properties, introduced in the next two definitions.  Lemma~\ref{main-absorber-lemma} states that if a 1-factorization has both of these properties, then we can build an absorber using the reserved vertices and colours with high probability.

Recall the hierarchy of constants $\eps, \gamma, \eta, \mu$ from~\eqref{constant-heirarchy}.
%
Firstly, we will need to show that if $G\in\cG_{[n-1]}^{\text{col}}$ is chosen uniformly at random, then with high probability, for any $V'\subseteq V$, $C'\subseteq C$ that are not too small,~$G$ admits many edges with colour in~$C'$ and both endpoints in~$V'$.
This property will be used in the construction of the tail of our absorber.
\begin{defin}\label{def:edge-resilient}
For $D\subseteq C= [n-1]$, we say that~$G\in\cG_{D}^{\text{col}}$ is $\eps$\textit{-locally edge-resilient} if for all sets of colours $D'\subseteq D$ and all sets of vertices $V'\subseteq V$ of sizes $|V'|, |D'|\geq\eps n$, we have that $e_{V',D'}(G)\geq\eps^{3}n^{2}/100$.
\end{defin}
Secondly, we will need that almost all $G\in \cG_{[n-1]}^{\text{col}}$ contain many $(v,c)$-absorbing gadgets for all $v\in V$, $c\in C$.
\begin{defin}\label{spread}
Let $D\subseteq C=[n-1]$.
\begin{itemize}
\item For $G\in\cG_{D}^{\text{col}}$, $x\in V$, $c\in D$, and $t\in\bN_{0}$, we say that a collection~$\cA_{(x,c)}$ of $(x,c)$-absorbing gadgets in~$G$ is $t$\textit{-well-spread} if
\begin{itemize}
    \item for all $v\in V$, there are at most~$t$ $(x,c)$-absorbing gadgets in~$\cA_{(x,c)}$ using~$v$;
    \item for all $e\in E(G)$, there are at most~$t$ $(x,c)$-absorbing gadgets in~$\cA_{(x,c)}$ using~$e$;
    \item for all $d\in D$, there are at most~$t$ $(x,c)$-absorbing gadgets in~$\cA_{(x,c)}$ using~$d$.
\end{itemize}
(Note that by definition of `using' (see Definition~\ref{def:absorbing-gadget}), there are no $(x,c)$-absorbing gadgets using~$x$ or~$c$.)
\item We say that $G\in\cG_{[n-1]}^{\text{col}}$ is $\mu$\textit{-robustly gadget-resilient} if for all $x\in V$ and all $c\in C$, there is a $5\mu n/4$-well-spread collection of at least~$\mu^{4}n^{2}/2^{23}$ $(x,c)$-absorbing gadgets in~$G$.
\end{itemize}
\end{defin}

\begin{lemma}\label{main-switching-lemma}
  Suppose $1/n\ll \eps,\mu \ll1$.
  If~$\phi$ is a 1-factorization of~$K_n$ chosen uniformly at random, then~$\phi$ is $\eps$-locally edge-resilient and $\mu$-robustly gadget-resilient with high probability.
\end{lemma}

As discussed, we prove Lemma~\ref{main-switching-lemma} in Section~\ref{switching-section} using switching arguments.  The next lemma is used to construct an absorber using the reserved vertices and colours.

\begin{lemma}\label{main-absorber-lemma}
  Suppose $1/n\ll\eps\ll\gamma\ll\eta\ll\mu\ll1$, and let $p = q = \beta = 5\mu + 26887\eta/2 + \gamma/3 - 26880\eps$.
  If $\phi$ is an $\eps$-locally edge-resilient and $\mu$-robustly gadget-resilient 1-factorization of $K_n$ with vertex set $V$ and colour set $C$ and
  \begin{itemize}
  \item [(R1)] $V'$ is a $p$-random subset of $V$, 
  \item [(R2)] $C'$ is a $q$-random subset of $C$, and 
  \item [(R3)] $G'$ is a $\beta$-random subgraph of $K_n$,
  \end{itemize}
  then with high probability there is a $36\gamma$-absorber $(\mathcal A, \mathcal P, T, H)$ such that $\bigcup_{A\in\mathcal A}A \cup \bigcup_{P\in \mathcal P}P \cup T$  is contained in $(V', C', G')$ with $\gamma$-bounded remainder.
\end{lemma}

The final ingredient in the proof of Theorem~\ref{mainthm} is the following lemma which follows from~\cite[Lemma 16]{GKMO20}, that enables us to find the long rainbow path whose leftover we absorb using the absorber from Lemma~\ref{main-absorber-lemma}.

\begin{lemma}\label{long-rainbow-path-lemma}
  Suppose $1/n \ll \gamma \ll p$, and let $q = \beta = p$.  For every 1-factorization $\phi$ of $K_n$ with vertex set $V$ and colour set $C$, if
  \begin{itemize}
  \item $V'$ is a $p$-random subset of $V$, 
  \item $C'$ is a $q$-random subset of $C$, and
  \item $G$ is a $\beta$-random subgraph of $K_n$,
  \end{itemize}
  then with high probability there is a rainbow path contained in $(V', C', G)$ with $\gamma$-bounded remainder.
\end{lemma}

We conclude this section with a proof of Theorem~\ref{mainthm} in the case that~$n$ is even, assuming Lemmas~\ref{main-switching-lemma} and~\ref{main-absorber-lemma}.

\lateproof{Theorem~\ref{mainthm}, $n$ even case}
  By Lemma~\ref{main-switching-lemma}, it suffices to prove that if $\phi$ is an $\eps$-locally edge-resilient and $\mu$-robustly gadget-resilient 1-factorization, then there is a rainbow Hamilton path and a rainbow cycle containing all of the colours.

  Let $p = q = \beta$ as in Lemma~\ref{main-absorber-lemma}, let $V_1, V_2$ be a random partition of $V$ where $V_1$ is $p$-random and $V_2$ is $(1 - p)$-random, let $C_1, C_2$ be a random partition of $C$ where $C_1$ is $q$-random and $C_2$ is $(1 - q)$-random, and let $G_1$ and $G_2$ be $\beta$-random and $(1 - \beta)$-random subgraphs of $K_n$ such that $E(G_1)$ and $E(G_2)$ partition the edges of $K_n$.  By Lemma~\ref{main-absorber-lemma} applied with $V' = V_1$, $C' = C_1$, and $G' = G_1$, and by Lemma~\ref{long-rainbow-path-lemma} applied with $V' = V_2$, $C' = C_2$, and $G = G_2$, the following holds with high probability.  There exists
  \begin{enumerate}[(i)]
  \item a $36\gamma$-absorber $(\mathcal A, \mathcal P, T, H)$ such that $\bigcup_{A\in\mathcal A}A \cup \bigcup_{P\in \mathcal P}P \cup T$  is contained in $(V_1, C_1, G_1)$ with $\gamma$-bounded remainder, and
  \item a rainbow path $P'$ contained in $(V_2, C_2, G_2)$ with $\gamma$-bounded remainder.
  \end{enumerate}
  Now we fix an outcome of the random partitions $(V_1, V_2)$, $(C_1, C_2)$, and $(G_1, G_2)$ so that (i) and (ii) hold.  By Proposition~\ref{main-absorbing-proposition}, there is both a rainbow Hamilton path containing $P'$ and a rainbow cycle containing $P'$ and all of the colours in $C$, as desired.
\endproof

\section{Tools}
In this section, we collect some results that we will use throughout the paper.
\subsection{Probabilistic tools}
We will use the following standard probabilistic estimates.
\begin{lemma}[Chernoff Bound]\label{chernoff bounds}
  Let~$X$ have binomial distribution with parameters~$n,p$.
  Then for any $0 < t \leq np$,
  \begin{equation*}
    \Prob{|X - np| > t} \leq 2\exp\left(\frac{-t^2}{3np}\right).
  \end{equation*}
\end{lemma}
  Let $X_1, \dots, X_m$ be independent random variables taking values in $\mathcal X$, and let $f : \mathcal X^m \rightarrow \mathbb R$.
  If for all $i\in [m]$ and $x'_i, x_1, \dots, x_m \in \mathcal X$, we have
    \begin{equation*}
      |f(x_1, \dots, x_{i - 1}, x_i, x_{i + 1}, \dots, x_m) - f(x_1, \dots, x_{i - 1}, x'_i, x_{i + 1}, \dots, x_m)| \leq c_i,
    \end{equation*}
    then we say $X_i$ \textit{affects} $f$ by at most $c_i$.
\begin{theorem}[McDiarmid's Inequality]\label{mcd}
  If $X_1, \dots, X_m$ are independent random variables taking values in $\mathcal X$ and $f : \mathcal X^m \rightarrow \mathbb R$ is such that $X_i$ affects $f$ by at most $c_i$ for all $i\in [m]$, then for all $t > 0$,
  \begin{equation*}
    \Prob{|f(X_1, \dots, X_m) - \Expect{f(X_1, \dots, X_m)}| \geq t} \leq \exp\left(-\frac{t^2}{\sum_{i=1}^m c^2_i}\right).
  \end{equation*}
\end{theorem}

\subsection{Hypergraph matchings}
When we build our absorber in the proof of Lemma~\ref{main-absorber-lemma}, we seek to efficiently use the vertices, colours, and edges of our random subsets $V'\subseteq V$, $C'\subseteq C$, $E'\subseteq E$, and to do this we make use of the existence of large matchings in almost-regular hypergraphs with small codegree.
In fact, we will need the stronger property that there exists a large matching in such a hypergraph which is well-distributed with respect to a specified collection of vertex subsets.
We make this precise in the following definition.
  Given a hypergraph~$\cH$ and a collection of subsets~$\cF$ of~$V(\cH)$, we say a matching~$\cM$ in~$\cH$ is $(\gamma,\cF)$\textit{-perfect} if for each $F\in\cF$, at most $\gamma\cdot\max\{|F|,|V(\cH)|^{2/5}\}$ vertices of~$\cF$ are left uncovered by~$\cM$.
The following theorem is a consequence of Theorem 1.2 in~\cite{AY05}, and is based on a result of Pippenger and Spencer~\cite{PS89}.
\begin{theorem}\label{hypergraph-matching-thm}
  Suppose $1/n \ll \eps \ll \gamma \ll 1/r$.  Let $\mathcal H$ be an $r$-uniform hypergraph on $n$ vertices such that for some $D\in\mathbb N$, we have $d_{\mathcal H}(x) = (1 \pm \eps)D$ for all $x\in V(\mathcal H)$ and $\Delta^c(\mathcal H) \leq D / \log^{9r}n$.  If $\mathcal F$ is a collection of subsets of $V(\mathcal H)$ such that $|\mathcal F| \leq n^{\log n}$, then there exists a $(\gamma, \mathcal F)$-perfect matching.
\end{theorem}
We will use Theorem~\ref{hypergraph-matching-thm} in the final step of constructing an absorber (see Lemma~\ref{linking-lemma}).
We construct an auxiliary hypergraph~$\cH$ whose edges represent structures we wish to find, and a large well-distributed matching in~$\cH$ corresponds to an efficient allocation of vertices, colours, and edges of the $1$-factorization to construct almost all of these desired structures.
We remark that this is also a key strategy in the proof of Lemma~\ref{long-rainbow-path-lemma}, and was first used in~\cite{KKKO20}.
\subsection{Robustly matchable bipartite graphs of constant degree}
In this subsection, we prove that there exist large bipartite graphs which are robustly matchable as in Definition~\ref{def:rmbg}, and have constant maximum degree.
\begin{defin}
  Let $m \in \mathbb N$.
  \begin{itemize}
  \item An $RMBG(3m, 2m, 2m)$ is a bipartite graph $H$ with bipartition $(A, B_1 \cup B_2)$ where $|A| = 3m$ and $|B_1| = |B_2| = 2m$ such that for any $B' \subseteq B_1$ of size $m$, there is a perfect matching in $H - B'$.  In this case, we say $H$ is \textit{robustly matchable} with respect to $B_1$, and that~$B_{1}$ is the identified \textit{flexible set}.
  \item A $2RMBG(7m, 2m)$ is a bipartite graph $H$ with bipartition $(A, B)$ where $|A| = |B| = 7m$ such that $H$ is robustly matchable with respect to sets $A'\subseteq A$ and $B'\subseteq B$ where $|A'| = |B'| = 2m$.
  \end{itemize}
\end{defin}

By~\cite[Lemma~10.7]{M18}, for all sufficiently large~$m$ there exists an $RMBG(3m, 2m, 2m)$ with maximum degree at most 100.
We use a one-sided (there is one flexible set) $RMBG(3m,2m,2m)$ exhibited in~\cite[Corollary~10]{GKMO20} in which each of the vertex classes are regular, to construct a $256$-regular two-sided (in that we identify a flexible set on each side of the vertex bipartition) $2RMBG(7m,2m)$.
\begin{lemma}\label{2rmbg-lemma}
  For all sufficiently large $m$, there is a $2RMBG(7m, 2m)$ that is $256$-regular.
\end{lemma}
\begin{proof}
Suppose that $m\in\bN$ is sufficiently large.
  By~\cite[Corollary~10]{GKMO20}, there exists an $RMBG(3m, 2m, 2m)$ that is (256, 192)-regular (i.e.\ all vertices in the first vertex class have degree~$256$ and all vertices in the second vertex class of have degree~$192$). Let~$H$ and~$H'$ be two vertex-disjoint isomorphic copies of a $(256, 192)$-regular $RMBG(3m, 2m, 2m)$, and let $(A, B_1\cup B_2)$ and $(A', B'_1 \cup B'_2)$ be the bipartitions of $H$ and $H'$ respectively such that $H$ is robustly matchable with respect to $B_1$ and $H'$ is robustly matchable with respect to $B'_1$.

  Let $H''$ be a 64-regular bipartite graph with bipartition $(B_1 \cup B_2, B'_1 \cup B'_2)$ such that $H''[B_1 \cup B'_1]$ contains a perfect matching $M$.  We claim that $H\cup H'\cup H''$ is robustly matchable with respect to $B_1$ and $B'_1$.  To that end, let $X\subseteq B_1$ and $Y\subseteq B'_1$ such that $|X| = |Y| \leq m$.  It suffices to show that $H\cup H'\cup H'' - (X\cup Y)$ has a perfect matching.  Since $H''[B_1\cup B'_1]$ contains a perfect matching, $H''[B_1\cup B'_1] - (X\cup Y)$ contains a matching of size at least $2m - |X| - |Y| = 2(m - |X|)$.  Thus, there exists a matching $M'$ in $H''[B_1\cup B'_1] - (X\cup Y)$ of size $m - |X|$.  Let $X' \coloneqq X \cup (B_1\cap V(M'))$ and $Y' \coloneqq Y \cup (B'_1 \cap V(M'))$, and note that $|X'| = |Y'| = m$.  Since $H$ is an $RMBG(3m, 2m, 2m)$, $H - X'$ has a perfect matching $M_1$, and similarly $H' - Y'$ has a perfect matching $M_2$.  Now $M' \cup M_1 \cup M_2$ is a perfect matching in $H\cup H'\cup H'' - (X\cup Y)$, as required.  Since $H\cup H'\cup H''$ is 256-regular, the result follows.
\end{proof}

\section{Constructing the absorber: proof of Lemma~\ref{main-absorber-lemma}}
\label{absorption-section}
Throughout this section, let $\phi$ be an $\eps$-locally edge-resilient and $\mu$-robustly gadget resilient 1-factorization of $K_n$ with vertex set $V$ and colour set $C$, let $E\coloneqq E(K_{n})$, and recall
\begin{equation*}
  1/n \ll \eps  \ll \gamma \ll \eta \ll \mu \ll 1.
\end{equation*}

Let $\tilde H$ be a 256-regular $2RMBG(7m, 2m)$  where $2m = (\eta - 2\eps)n$, which exists by Lemma~\ref{2rmbg-lemma}.
We define the following probabilities:

\begin{minipage}{.5\linewidth}
  \begin{eqnarray}
    p_{\mathrm{flex}} &\coloneqq& \eta,\nonumber\\
    p_{\mathrm{buff}} &\coloneqq& 5\eta / 2,\nonumber\\
    p_{\mathrm{abs}} &\coloneqq& 6|E(\tilde H)|/n + 2\mu,\label{slicearray}\\
    p_{\mathrm{link}} &\coloneqq& 9|E(\tilde H)|/n + 3\mu,\nonumber\\
    p_{\mathrm{link}}' &\coloneqq& \gamma/3,\nonumber
  \end{eqnarray}
\end{minipage}%
\begin{minipage}{.5\linewidth}
  \begin{eqnarray*}
    q_{\mathrm{flex}} &\coloneqq& \eta,\nonumber\\
    q_{\mathrm{buff}} &\coloneqq& 5\eta / 2,\nonumber\\
    q_{\mathrm{abs}} &\coloneqq& 3|E(\tilde H)|/n + \mu,\nonumber\\
    q_{\mathrm{link}} &\coloneqq& 12|E(\tilde H)|/n + 4\mu,\nonumber\\  
    q_{\mathrm{link}}' &\coloneqq& \gamma/3,\nonumber\\
  \end{eqnarray*}
\end{minipage}

\noindent and we let $p_{\mathrm{main}} \coloneqq 1 - p_{\mathrm{flex}} - p_{\mathrm{buff}}- p_{\mathrm{abs}} - p_{\mathrm{link}} - p'_{\mathrm{link}}$ and $q_{\mathrm{main}} \coloneqq 1 - q_{\mathrm{flex}} -q_{\mathrm{buff}}- q_{\mathrm{abs}} - q_{\mathrm{link}} - q'_{\mathrm{link}}$.  Note that $p_{\mathrm{main}} = q_{\mathrm{main}}$, and let $\beta \coloneqq 1 - p_{\mathrm{main}}$.

\begin{defin}
An \textit{absorber partition of~$V$,~$C$, and~$K_{n}$} is defined as follows:
\begin{equation}\label{absorber-random-partition}
  \begin{split}
    &V = V_{\mathrm{main}}\, \dot{\cup}\, V_{\mathrm{flex}} \,\dot{\cup}\, V_{\mathrm{buff}}\, \dot{\cup}\, V_{\mathrm{abs}} \,\dot{\cup}\, V_{\mathrm{link}} \,\dot{\cup}\, V_{\mathrm{link}}',\ \text{and}\\
    &C = C_{\mathrm{main}} \,\dot{\cup}\, C_{\mathrm{flex}} \,\dot{\cup}\, C_{\mathrm{buff}} \,\dot{\cup}\, C_{\mathrm{abs}} \,\dot{\cup}\, C_{\mathrm{link}} \,\dot{\cup}\, C_{\mathrm{link}}', 
  \end{split}
\end{equation}
where $V_{\mathrm{main}}$ is $p_{\mathrm{main}}$-random,~$V_{\mathrm{flex}}$ is $p_{\mathrm{flex}}$-random etc, and the sets of colours are defined analogously.  Let $V' \coloneqq V\setminus V_{\mathrm{main}}$, $C' \coloneqq C\setminus C_{\mathrm{main}}$, and let $G'$ be a $\beta$-random subgraph of~$K_{n}$.
\end{defin}
Note that~$V'$,~$C'$, and~$G'$ satisfy~(R1)--(R3) in the statement of Lemma~\ref{main-absorber-lemma}.\COMMENT{Note that $|E(\Tilde{H})|=256\cdot7m=1792m=896\cdot2m=896(\eta-2\eps)n=896\eta n-1792\eps n$. Write $p \coloneqq 1-p_{\text{main}}$ and $q\coloneqq 1-q_{\text{main}}$. Then clearly $p=q=\beta$,~$V'$ is $p$-random,~$C'$ is $q$-random,~$G'$ is $\beta$-random, and $p=\eta+5\eta/2+(6|E(\Tilde{H})|/n+2\mu)+(9|E(\Tilde{H})|/n+3\mu)+\gamma/3=5\mu+7\eta/2+15|E(\Tilde{H})|/n+\gamma/3=5\mu+7\eta/2+13440\eta-26880\eps +\gamma/3=5\mu+26887\eta/2+\gamma/3-26880\eps$.}
\subsection{Overview of the proof}\label{overviewsec}

We now overview our strategy for proving Lemma~\ref{main-absorber-lemma}.  First we need the following definitions.
A \textit{link} is a rainbow path of length 4 with internal vertices in $V_{\mathrm{link}} \cup V'_{\mathrm{link}}$, ends in $V_{\mathrm{abs}}$, and colours and edges in $C_{\mathrm{link}} \cup C'_{\mathrm{link}}$ and $G'$, respectively. 
A link with internal vertices in $V_{\mathrm{link}}$ and colours in $C_{\mathrm{link}}$ is a \textit{main link}, and a link with internal vertices in $V'_{\mathrm{link}}$ and colours in $C'_{\mathrm{link}}$ is a \textit{reserve link}.
If $M$ is a matching and $\mathcal P = \{P_e\}_{e\in E(M)}$ is a collection of vertex-disjoint links such that $\bigcup_{P\in\mathcal P}P$ is rainbow and $P_{uv}$ has ends $u$ and $v$ for every $uv\in E(M)$, then $\mathcal P$ \textit{links} $M$.

We aim to build a $36\gamma$-absorber $(\mathcal A, \mathcal P, T, H)$ such that $\bigcup_{A\in \mathcal A} A \cup \bigcup_{P\in\mathcal P} P \cup T$ is contained in $(V', C', G')$ with $\gamma$-bounded remainder and $H\cong \tilde H$.  
First, we show (see Lemma~\ref{absorbing-template-lemma}) that with high probability there is a $36\gamma$-absorbing template $H\cong \tilde H$, where
\begin{itemize}
\item $H$ has flexible sets $(V'_{\mathrm{flex}}, C'_{\mathrm{flex}}, G')$ and $(V'_{\text{flex}},C'_{\text{flex}})$ is contained in $(V_{\mathrm{flex}}, C_{\mathrm{flex}})$ with $3\eps$-bounded remainder, and
\item $H$ has buffer sets $V'_{\mathrm{buff}}$ and $C'_{\mathrm{buff}}$ where $(V'_{\mathrm{buff}}, C'_{\mathrm{buff}})$ is contained in $(V_{\mathrm{buff}}, C_{\mathrm{buff}})$ with $6\eps$-bounded remainder.
\end{itemize}

Then, we show that with high probability, there exists an $H$-absorber $(\mathcal A, \mathcal P)$ where
\begin{itemize}
\item for every $vc\in E(H)$, the $(v, c)$-absorbing gadget $A_{v, c} \in \mathcal A$ uses vertices, colours, and edges in $V_{\mathrm{abs}}$, $C_{\mathrm{abs}}$, and $G'$, respectively, and
\item every $P\in\mathcal P$ is a link.
\end{itemize}
In particular, if $\mathcal A = \{A_1, \dots, A_k\}$, where~$A_{i}$ is a $(v_{i},c_{i})$-absorbing gadget, then $\mathcal P$ links the matching $M_1\cup M_2 \cup M_3$, where~$V(M_1)$,~$V(M_2)$, and~$V(M_3)$ are pairwise vertex-disjoint, and
\begin{itemize}
\item [(M1)] $M_{1}=\{r_{1}s_{1},\dots,r_{k}s_{k}\}$, where~$r_{i}$ and~$s_{i}$ are non-adjacent vertices of the $4$-cycle in~$A_{i}$, for each $i\in[k]$,
\item [(M2)] $M_{2}=\{w_{1}x_{1},\dots,w_{k}x_{k}\}$, where~$w_{i}$ is a non-$v_{i}$ vertex of the triangle in~$A_{i}$ and~$x_{i}$ is a vertex of the $4$-cycle in~$A_{i}$, for each $i\in[k]$, and
\item [(M3)] $M_{3}=\{y_{1}z_{2},\dots,y_{k-1}z_{k}\}$, where~$y_{i}$ is a non-$v_{i}$ vertex of the triangle in~$A_{i}$ for each $i\in[k-1]$, and~$z_{i}$ is a vertex of the $4$-cycle in~$A_{i}$ for each $i\in[k]\setminus\{1\}$.
\end{itemize}

Finally, letting $V'_{\mathrm{abs}}$ and $C'_{\mathrm{abs}}$ 
be the vertices and colours 
in $V_{\mathrm{abs}}$ and $C_{\mathrm{abs}}$ 
not used by any $(v, c)$-absorbing gadget in $\mathcal A$, we show that with high probability there is a tail $T$ for $(\mathcal A, \mathcal P)$ where $T$ is the union of 
\begin{itemize}
\item a rainbow matching $M$ contained in ($V'_{\mathrm{abs}}, C'_{\mathrm{abs}}, G')$ with $6\eps$-bounded remainder and
\item a collection $\mathcal T$ of vertex-disjoint links where all but one vertex in~$V(M)$ is the end of precisely one link.
\end{itemize}
In particular, if $E(M) = \{a_{1}b_{1}, \dots, a_{\ell} b_{\ell}\}$, then $\mathcal T$ links $M_4$, where
\begin{itemize}
\item [(M4)] $M_4$ is a matching of size~$\ell$ with edges $b_{i} a_{i + 1}$ for every $i \in [\ell - 1]$ and an edge $va_{1}$ where $v$ is one of the two vertices used by a gadget in $\mathcal A$ that is not in a link in $\mathcal P$.
\end{itemize}
\begin{figure}
  \centering
    \begin{tikzpicture}[scale = 0.9]
    \tikzstyle{vtx}=[draw, fill, circle, scale = .5];
    \tikzstyle{link}=[snake=zigzag];

    \node[label=left:{\Large $v_1$}, vtx] (v) at (0, 0) {};

    \node[vtx] (t1) at ($(v) + (30:1)$) {};
    \node[vtx] (b1) at ($(v) + (-30:1)$) {};

    \node[vtx] (t2) at ($(t1) + (1, 0)$) {};
    \node[vtx] (b2) at ($(b1) + (1, 0)$) {};

    \node[vtx] (t3) at ($(t2) + (1, 0)$) {};
    \node[vtx] (b3) at ($(b2) + (1, 0)$) {};

    \draw (v) -- (t1);
    \draw (t1) -- (b1);
    \draw (b1)-- (v);

    \draw (t2) -- (b2);
    \draw (t2) -- (t3);
    \draw (b2) -- (b3);
    \draw[OwlRed] (t3) [ultra thick] -- (b3);

    \draw[link] (t1) -- (b2);
    \draw[link] (t2) -- (b3);

    \node[vtx] (t'1) at ($(t3) + (1, 0)$) {};
    \node[vtx] (b'1) at ($(b3) + (1, 0)$) {};

    \node[vtx] (t'2) at ($(t'1) + (1, 0)$) {};
    \node[vtx] (b'2) at ($(b'1) + (1, 0)$) {};

    \node[vtx] (t'3) at ($(t'2) + (1, 0)$) {};
    \node[vtx] (b'3) at ($(b'2) + (1, 0)$) {};

    \draw (v) .. controls ($(v) + (0, 1)$) and ($(t'1) + (0, .5)$) .. (t'1);
    \draw  (v) .. controls ($(v) + (0, -1)$) and ($(b'1) + (0, -.5)$) .. (b'1);
    \draw (t'1) -- (b'1);

    \draw (t'2) -- (b'2);
    \draw (t'2) -- (t'3);
    \draw (b'2) -- (b'3);
    \draw (t'3) [OwlBlue, ultra thick] -- (b'3);

    \draw[link] (t'1) -- (b'2);
    \draw[link] (t'2) -- (b'3);

    \draw[link] (t3) -- (b'1);

    \node[vtx] (t''1) at ($(t'3) + (2.5, 0)$) {};
    \node[vtx] (b''1) at ($(b'3) + (2.5, 0)$) {};

    \node[vtx] (t''2) at ($(t''1) + (1, 0)$) {};
    \node[vtx] (b''2) at ($(b''1) + (1, 0)$) {};

    \node[vtx] (t''3) at ($(t''2) + (1, 0)$) {};
    \node[vtx] (b''3) at ($(b''2) + (1, 0)$) {};

    \draw (v) .. controls ($(v) + (0, 1.5)$) and ($(t''1) + (0, .5)$) .. (t''1);
    \draw  (v) .. controls ($(v) + (0, -1.5)$) and ($(b''1) + (0, -.5)$) .. (b''1);
    \draw (t''1) -- (b''1);

    \draw (t''2) -- (b''2);
    \draw (t''2) -- (t''3);
    \draw (b''2) -- (b''3);
    \draw (t''3) [ultra thick, OwlCyan] -- (b''3);

    \draw[link] (t''1) -- (b''2);
    \draw[link] (t''2) -- (b''3);

    \node (ttm) at ($(t'3)!.5!(t''1)$) {};
    \node (tbm) at ($(b'3)!.5!(b''1)$) {};
    \node at ($(ttm)!.5!(tbm)$) {\Huge $\cdots$};

    \draw[link] (t'3) -- ($(t'3) + (.5, -.5)$);
    \draw[link] (b''1) -- ($(b''1) + (-.5, .5)$);


    \node[label=left:{\Large $v_{7m}$}, vtx] (bv) at (0, -4) {};

    \node[vtx] (bt1) at ($(bv) + (30:1)$) {};
    \node[vtx] (bb1) at ($(bv) + (-30:1)$) {};

    \node[vtx] (bt2) at ($(bt1) + (1, 0)$) {};
    \node[vtx] (bb2) at ($(bb1) + (1, 0)$) {};

    \node[vtx] (bt3) at ($(bt2) + (1, 0)$) {};
    \node[vtx] (bb3) at ($(bb2) + (1, 0)$) {};

    \draw (bv) -- (bt1);
    \draw (bt1) -- (bb1);
    \draw (bb1)-- (bv);

    \draw (bt2) -- (bb2);
    \draw (bt2) -- (bt3);
    \draw (bb2) -- (bb3);
    \draw (bt3) [ultra thick, OwlGreen] -- (bb3);

    \draw[link] (bt1) -- (bb2);
    \draw[link] (bt2) -- (bb3);

    \node[vtx] (bt'1) at ($(bt3) + (1, 0)$) {};
    \node[vtx] (bb'1) at ($(bb3) + (1, 0)$) {};

    \node[vtx] (bt'2) at ($(bt'1) + (1, 0)$) {};
    \node[vtx] (bb'2) at ($(bb'1) + (1, 0)$) {};

    \node[vtx] (bt'3) at ($(bt'2) + (1, 0)$) {};
    \node[vtx] (bb'3) at ($(bb'2) + (1, 0)$) {};

    \draw (bv) .. controls ($(bv) + (0, 1)$) and ($(bt'1) + (0, .5)$) .. (bt'1);
    \draw (bv) .. controls ($(bv) + (0, -1)$) and ($(bb'1) + (0, -.5)$) .. (bb'1);
    \draw (bt'1) -- (bb'1);

    \draw (bt'2) -- (bb'2);
    \draw (bt'2) -- (bt'3);
    \draw (bb'2) -- (bb'3);
    \draw (bt'3)[ultra thick, OwlOrange] -- (bb'3);

    \draw[link] (bt'1) -- (bb'2);
    \draw[link] (bt'2) -- (bb'3);

    \draw[link] (bt3) -- (bb'1);

    \node[vtx] (bt''1) at ($(bt'3) + (2.5, 0)$) {};
    \node[vtx] (bb''1) at ($(bb'3) + (2.5, 0)$) {};

    \node[vtx] (bt''2) at ($(bt''1) + (1, 0)$) {};
    \node[vtx] (bb''2) at ($(bb''1) + (1, 0)$) {};

    \node[vtx] (bt''3) at ($(bt''2) + (1, 0)$) {};
    \node[vtx] (bb''3) at ($(bb''2) + (1, 0)$) {};

    \draw (bv) .. controls ($(bv) + (0, 1.5)$) and ($(bt''1) + (0, .5)$) .. (bt''1);
    \draw  (bv) .. controls ($(bv) + (0, -1.5)$) and ($(bb''1) + (0, -.5)$) .. (bb''1);
    \draw (bt''1) -- (bb''1);

    \draw (bt''2) -- (bb''2);
    \draw (bt''2) -- (bt''3);
    \draw (bb''2) -- (bb''3);
    \draw (bt''3) [ultra thick, OwlViolet]-- (bb''3);

    \draw[link] (bt''1) -- (bb''2);
    \draw[link] (bt''2) -- (bb''3);

    \node (bttm) at ($(bt'3)!.5!(bt''1)$) {};
    \node (btbm) at ($(bb'3)!.5!(bb''1)$) {};
    \node at ($(bttm)!.5!(btbm)$) {\Huge $\cdots$};

    \draw[link] (bt'3) -- ($(bt'3) + (.5, -.5)$);
    \draw[link] (bb''1) -- ($(bb''1) + (-.5, .5)$);

    \node (2b''3) at ($(b''3)+(0, -2)$) {};
    \draw[decorate, decoration=zigzag] (b1) .. controls ($(b1) + (0, -1)$) .. ($(b1)!.5!(2b''3)$);
    \node (2bt1) at ($(bt1)+(0, 2)$) {};
    \draw[decorate, decoration=zigzag] (bt''3) .. controls ($(bt''3) + (0, 1)$) .. ($(bt''3)!.5!(2bt1)$);

    \node[label=left:{\Huge $\vdots$}] at ($(v)!.5!(bv)$) {};


    \node[vtx] (m1) at ($(t''3) + (1, 0)$) {};
    \node[vtx] (m'1) at ($(m1) + (.5, 0)$) {};
    \draw (m1) -- (m'1);
    \node[vtx] (2m1) at ($(m'1) + (.5, 0)$) {};
    \node[vtx] (2m'1) at ($(2m1) + (.5, 0)$) {};
    \draw (2m1) -- (2m'1);
    \node[vtx] (m2) at ($(b''3) + (1, 0)$) {};
    \node[vtx] (m'2) at ($(m2) + (.5, 0)$) {};
    \draw (m2) -- (m'2);
    \node[vtx] (2m2) at ($(m'2) + (.5, 0)$) {};
    \node[vtx] (2m'2) at ($(2m2) + (.5, 0)$) {};
    \draw (2m2) -- (2m'2);
    \node[vtx] (m3) at ($(bt''3) + (1, 0)$) {};
    \node[vtx] (m'3) at ($(m3) + (.5, 0)$) {};
    \draw (m3) -- (m'3);
    \node[vtx] (2m3) at ($(m'3) + (.5, 0)$) {};
    \node[vtx] (2m'3) at ($(2m3) + (.5, 0)$) {};
    \draw (2m3) -- (2m'3);
    \node[vtx] (m4) at ($(bb''3) + (1, 0)$) {};
    \node[vtx] (m'4) at ($(m4) + (.5, 0)$) {};
    \draw (m4) -- (m'4);
    \node[vtx] (2m4) at ($(m'4) + (.5, 0)$) {};
    \node[vtx] (2m'4) at ($(2m4) + (.5, 0)$) {};
    \draw (2m4) -- (2m'4);

    \draw[link] (t''3) -- (m1);
    \draw[link] (m'1) -- (m2);
    \draw[link] (m'2) -- ($(m'2) + (-.25, -.5)$);
    \draw[link] (m3) -- ($(m3) + (.25, .5)$);
    \draw[link] (m'3) -- (m4);
    \draw[decorate, decoration=zigzag] (m'4) .. controls ($(m'4) + (0, .5)$) and ($(2m4) + (0, .5)$) .. (2m4);
    \draw[link] (2m'4) -- (2m3);
    \draw[link] (2m'3) -- ($(2m'3) + (-.25, .5)$);
    \draw[link] (2m2) -- ($(2m2) + (.25, -.5)$);
    \draw[link] (2m'2) -- (2m1);

    \node at ($(m'2)!.5!(m'3) + (.25, 0)$) {\Huge $\vdots$};

    \node (tl-corner) at ($(t1) + (-2, 1)$) {};
    \node (top-edge1) at ($(t1) + (-.5, 1)$) {};
    \node (top-edge2) at ($(t'1) + (-.5, 1)$) {};
    \node (top-edge3) at ($(t'3) + (.5, 1)$) {};
    \node (top-edge4) at ($(t''1) + (-.5, 1)$) {};
    \node (top-edge5) at ($(t''3) + (.5, 1)$) {};
    \node (top-edge6) at ($(2m'1) + (.5, 1)$) {};
    \node (bottom-edge1) at ($(bb1) + (-.5, -1)$) {};
    \node (bottom-edge2) at ($(bb'1) + (-.5, -1)$) {};
    \node (bottom-edge3) at ($(bb'3) + (.5, -1)$) {};
    \node (bottom-edge4) at ($(bb''1) + (-.5, -1)$) {};
    \node (bottom-edge5) at ($(bb''3) + (.5, -1)$) {};
    \node (bottom-edge6) at ($(2m'4) + (.5, -1)$) {}; 

    \draw (tl-corner) rectangle ($(2m'4) + (.5, -1)$);
    
    \draw (top-edge1.center) -- (bottom-edge1.center);
    \draw (top-edge2.center) -- (bottom-edge2.center);
    \draw (top-edge3.center) -- (bottom-edge3.center);
    \draw (top-edge4.center) -- (bottom-edge4.center);
    \draw (top-edge5.center) -- (bottom-edge5.center);
    
    \node[label=below:1, inner sep = 0pt] at ($(top-edge1)!.5!(top-edge2)$) {};
    \node[label=below:2, inner sep = 0pt] at ($(top-edge2)!.5!(top-edge3)$) {};
    \node[label=below:$\cdots$, inner sep = 0pt] at ($(top-edge3)!.5!(top-edge4)$) {};
    \node[label=below:256, inner sep = 0pt] at ($(top-edge4)!.5!(top-edge5)$) {};
    \node[label=below:$T$, inner sep = 0pt] at ($(top-edge5)!.5!(top-edge6)$) {};


  \end{tikzpicture}
    \caption{
    An absorber $(\mathcal A, \mathcal P, T, H)$, where $\mathcal P$ links $\bigcup_{i=1}^3 M_i$ and $T = M\cup \bigcup_{P\in\mathcal T} P$, where $\mathcal T$ links $M_4$.  Links are drawn as zigzags.  
  }
  \label{fig:building-absorber}
\end{figure}
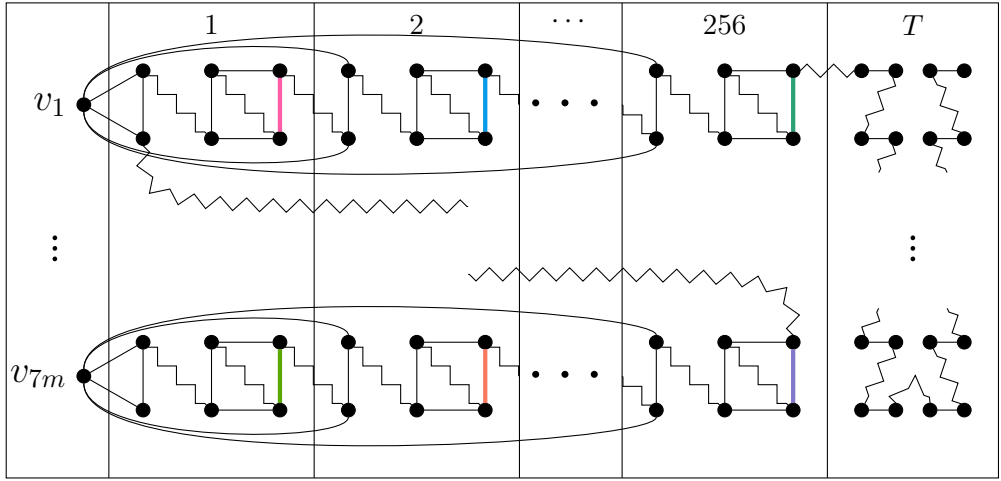

See Figure~\ref{fig:building-absorber}.
\begin{fact}\label{factref}
Suppose that~$\cA$ satisfies~$H$.
If~$\cP\cup\cT$ links $M_{1}\cup\dots\cup M_{4}$, where~$\cP$ links $M_{1}\cup M_{2}\cup M_{3}$ and~$\cT$ links~$M_{4}$, then~$\cP$ completes~$\cA$ and thus $(\cA,\cP)$ is an $H$-absorber.
Moreover, $T\coloneqq M\cup\bigcup_{P\in\cT}P$ is a tail of~$(\cA,\cP)$.
Thus~$(\cA,\cP,T,H)$ is a $36\gamma$-absorber.
\end{fact}
We find these structures in the following steps.
For Steps 1 and 2, see Lemma~\ref{greedy-absorber-lemma}, and for Steps 3 and 4, see Lemma~\ref{linking-lemma}.
\begin{enumerate}[1)]
\item First, we find the collection $\mathcal A$ of absorbing gadgets greedily, using the robust gadget-resilience property of $\phi$,
\item then we greedily construct the matching $M$, using the local edge-resilience property of $\phi$.
\item Next, we construct an auxiliary hypergraph in which each hyperedge corresponds to a main link and apply Theorem~\ref{hypergraph-matching-thm} to choose most of the links in $\mathcal P$, and
\item finally we greedily choose the remainder of the links in $\mathcal P$ from the reserve links.
\end{enumerate}

\subsection{The absorbing template}

\begin{lemma}\label{absorbing-template-lemma}
Consider an absorber partition of~$V$,~$C$, and~$K_{n}$.
  With high probability, there exists a $36\gamma$-absorbing template $H\cong \tilde H$, where
  \begin{enumerate}[(\ref{absorbing-template-lemma}.1), topsep = 6pt]
  \item\label{template-flexible-sets} $H$ has flexible sets $(V'_{\mathrm{flex}}, C'_{\mathrm{flex}}, G')$ where~$(V'_{\text{flex}},C'_{\text{flex}})$ is contained in $(V_{\mathrm{flex}}, C_{\mathrm{flex}})$ with $3\eps$-bounded remainder, and
  \item\label{template-buffer-sets} $H$ has buffer sets $V'_{\mathrm{buff}}$ and $C'_{\mathrm{buff}}$ where $(V'_{\mathrm{buff}}, C'_{\mathrm{buff}})$ is contained in $(V_{\mathrm{buff}}, C_{\mathrm{buff}})$ with $6\eps$-bounded remainder.
  \end{enumerate}
\end{lemma}
\begin{proof}
  For convenience, let $p \coloneqq p_{\mathrm{flex}}$ and $q \coloneqq q_{\mathrm{flex}}$.
  We claim that the following holds with high probability:
  \begin{enumerate}[(a)]
  \item\label{flexible-sets-right-size} $|V_{\mathrm{flex}}|, |C_{\mathrm{flex}}| = (\eta \pm \eps)n$,
  \item\label{buffer-sets-right-size} $|V_{\mathrm{buff}}|, |C_{\mathrm{buff}}| = (5\eta/2 \pm \eps)n$, and
  \item\label{many-covers-survive} for every distinct $u, v\in V$ and $c \in C$, there are at least $p^3q^3\beta^4 n^2 / 4$ $(V_{\mathrm{flex}}, C_{\mathrm{flex}}, G')$-covers of $u$, $v$, and $c$.
  \end{enumerate}
  Indeed,~\ref{flexible-sets-right-size} and~\ref{buffer-sets-right-size} follow from the Chernoff Bound (Lemma~\ref{chernoff bounds}). To prove~\ref{many-covers-survive}, for each $u, v$, and $c$, we apply McDiarmid's Inequality (Theorem~\ref{mcd}). Consider the random variable $f$ counting the number of $(V_{\mathrm{flex}}, C_{\mathrm{flex}}, G')$-covers of~$u$,~$v$, and~$c$. Note that~$f$ is determined by the following independent binomial random variables: $\{X_z\}_{z\in V}$, where $X_z$ indicates if $z\in V_{\mathrm{flex}}$, $\{X_{c'}\}_{c'\in C}$, where $X_{c'}$ indicates if $c'\in C_{\mathrm{flex}}$, and for each edge $e$, the random variable $X_e$ which indicates if $e\in E(G')$.  We claim there are at least $2(n/2 - 2)(n - 7)$ $(V, C, K_{n})$-covers of $u$, $v$, and $c$.  To that end, let $u'w$ be a $c$-edge with $u',w\in V\setminus\{u,v\}$.  There are at least $n - 7$ vertices $v'\in V\setminus \{u, v, u', w\}$ such that $\phi(vv'), \phi(wv')\notin \{\phi(uu'), c\}$, and for each such vertex $v'$ the path~$uu'wv'v$ is a $(V, C, K_{n})$-cover of $u$, $v$, and $c$.
  Similarly, there are at least~$n-7$ $(V,C,K_{n})$-covers of the form~$uwu'v'v$.
  Altogether this gives at least $2(n/2 - 2)(n - 7) \geq n^2/2$ $(V, C, K_{n})$-covers of $u$, $v$, and $c$, as claimed.  Therefore $\Expect{f} \geq p^3q^3\beta^4 n^2/2$.  For each $z\in V$, $X_z$ affects $f$ by at most $3n$, and $X_{uz}$, and $X_{vz}$ each affect $f$ by at most $n$, and for each $c'\in C$, $X_{c'}$ affects $f$ by at most $3n$.  For each edge $e$ not incident to $u$ or $v$, if $e$ is a $c$-edge, then $X_e$ affects $f$ by at most $2n$, and otherwise $e$ affects $f$ by at most two.  Thus, by McDiarmid's Inequality applied with $t = \Expect{f}/2$, there are at least $p^3q^3\beta^4 n^2/4$ $(V_{\mathrm{flex}}, C_{\mathrm{flex}}, G')$-covers of $u$, $v$, and $c$ with probability at least $1 - \exp\left(- p^{6}q^{6}\beta^{8}n^4 / O(n^3)\right)$. Thus by a union bound,~\ref{many-covers-survive} also holds with high probability.

  Now we assume~\ref{flexible-sets-right-size}--\ref{many-covers-survive} holds, and we show there exists a $36\gamma$-absorbing template $H\cong \tilde H$ satisfying~\ref{template-flexible-sets} and~\ref{template-buffer-sets}.

  Since $m= (\eta/2 - \eps)n$, by~\ref{flexible-sets-right-size} and~\ref{buffer-sets-right-size}, there exists $V'_{\mathrm{flex}} \subseteq V_{\mathrm{flex}}$, $C'_{\mathrm{flex}} \subseteq C_{\mathrm{flex}}$, $V'_{\mathrm{buff}} \subseteq V_{\mathrm{buff}}$, and $C'_{\mathrm{buff}} \subseteq C_{\mathrm{buff}}$, such that $|V'_{\mathrm{flex}}|, |C'_{\mathrm{flex}}| = 2m$ and $|V'_{\mathrm{buff}}|, |C'_{\mathrm{buff}}| = 5m$, which we choose arbitrarily, and moreover, $|V_{\mathrm{flex}}\setminus V'_{\mathrm{flex}}|, |C_{\mathrm{flex}}\setminus C'_{\mathrm{flex}}| \leq 3\eps n$ and $|V_{\mathrm{buff}}\setminus V'_{\mathrm{buff}}|, |C_{\mathrm{buff}}\setminus C'_{\mathrm{buff}}| \leq 6\eps n$, as required.  Choose bijections from $V'_{\mathrm{flex}}, C'_{\mathrm{flex}}$, $V'_{\mathrm{buff}}$, and $C'_{\mathrm{buff}}$ to the flexible sets and the buffer sets of $\tilde H$ arbitrarily, and let $H\cong \tilde H$ be the corresponding graph.  Now $H$ satisfies~\ref{template-flexible-sets} and~\ref{template-buffer-sets}, as required, so it remains to show that $H$ is a $36\gamma$-absorbing template.  Since each vertex or colour in $V_{\mathrm{flex}}$ or $C_{\mathrm{flex}}$ is in at most $3n$ $(V_{\mathrm{flex}}, C_{\mathrm{flex}}, G')$-covers of $u$, $v$, and $c$,~\ref{flexible-sets-right-size} and~\ref{many-covers-survive} imply that there are at least $p^3q^3\beta^4 n^2 / 4 - 18\eps n^2 \geq 36\gamma n^2$ $(V'_{\mathrm{flex}}, C'_{\mathrm{flex}}, G')$-covers of $u$, $v$, and $c$, so $H$ is a $36\gamma$-absorbing template, as desired.
\end{proof}

\subsection{Greedily building an $H$-absorber}

\begin{lemma}\label{greedy-absorber-lemma}
Consider an absorber partition of~$V$,~$C$, and~$K_{n}$.
  The following holds with high probability. Suppose $V_{\text{res}}\subseteq V_{\text{flex}}\cup V_{\text{buff}}$ and $C_{\text{res}}\subseteq C_{\text{flex}}\cup C_{\text{buff}}$. For every graph $H\cong \tilde H$ with bipartition $(V_{\mathrm{res}}, C_{\mathrm{res}})$, there exists
  \begin{enumerate}[(\ref{greedy-absorber-lemma}.1), topsep = 6pt]
  \item\label{greedily-choosing-absorber} a collection $\cA=\{A_{vc}\colon vc\in E(H)\}$ such that~$\cA$ satisfies~$H$ and such that for all $A_{vc}\in\cA$ we have that~$A_{vc}$ uses vertices, colours, and edges in $V_{\mathrm{abs}}$, $C_{\mathrm{abs}}$, and $G'$ respectively, and
  \item\label{greedily-choosing-matching} a rainbow matching $M$ contained in $(V'_{\mathrm{abs}}, C'_{\mathrm{abs}}, G')$ with $5\eps$-bounded remainder, where $V'_{\text{abs}}$ and~$C'_{\text{abs}}$ are the sets of vertices and colours in~$V_{\text{abs}}$ and~$C_{\text{abs}}$ not used by any absorbing gadget in~$\cA$.
  \end{enumerate}
\end{lemma}
\begin{proof}
  For convenience, let $p \coloneqq p_{\mathrm{abs}}$ and $q \coloneqq q_{\mathrm{abs}}$ 
  in this proof.

Since~$\phi$ is $\mu$-robustly gadget-resilient, for every $v\in V$, $c\in C$, there is a collection~$\cA_{v,c}$ of precisely\COMMENT{Helps to simplify the calculation where we bound $\prob{\cE_{v,c}}$, as~$|\cA_{v,c}|$ appears on both top and bottom.}~$2^{-23}\mu^{4}n^{2}$ $(v,c)$-absorbing gadgets such that every vertex, every colour, and every edge is used by at most~$5\mu n/4$ of the $A\in\cA_{v,c}$. (Recall from Definition~\ref{def:absorbing-gadget} that a $(v,c)$-absorbing gadget does not use~$v$ and~$c$.)
Fix $v\in V$, $c\in C$.
The expected number of the $(v,c)$-absorbing gadgets in~$\cA_{v,c}$ using only vertices in~$V_{\text{abs}}$, colours in~$C_{\text{abs}}$, and edges in~$G'$ is~$p^{6}q^{3}\beta^{7}|\cA_{v,c}|$.
Let~$\cE_{v,c}$ be the event that fewer than~$p^{6}q^{3}\beta^{7}|\cA_{v,c}|/2$ of the $(v,c)$-absorbing gadgets in~$\cA_{v,c}$ use only vertices in~$V_{\text{abs}}$, colours in~$C_{\text{abs}}$ and edges in~$G'$.
We claim that $\prob{\cE_{v,c}}\leq\exp(-2^{-51}p^{12}q^{6}\beta^{14}\mu^{6}n)$.

To see this, for each $u\in V$, $d\in C$, $e\in E$, let~$m_{u}$,~$m_{d}$, and~$m_{e}$ denote the number of $(v,c)$-absorbing gadgets in~$\cA_{v,c}$ using~$u$,~$d$, and~$e$, respectively.
We will apply McDiarmid's Inequality (Theorem~\ref{mcd}) to the function~$f_{v,c}$ which counts the number of $A\in\cA_{v,c}$ using only vertices in~$V_{\text{abs}}$, colours in~$C_{\text{abs}}$, and edges in~$G'$.
We use independent indicator random variables $\{X_{u}\}_{u\in V}\cup\{X_{d}\}_{d\in C}\cup\{X_{e}\}_{e\in E}$ which indicate whether or not a vertex~$u$ is in~$V_{\text{abs}}$, a colour~$d$ is in~$C_{\text{abs}}$, and an edge~$e$ is in~$G'$.
Each random variable~$X_{u}$,~$X_{d}$,~$X_{e}$ affects~$f_{v,c}$ by at most~$m_{u}$,~$m_{d}$,~$m_{e}$, respectively.
Since $m_{u}\leq 5\mu n/4$ for all $u\in V$ and $m_{d}\leq 5\mu n/4$ for all $d\in C$, we have $\sum_{u\in V}m_{u}^{2}$, $\sum_{d\in C}m_{d}^{2}\leq 25\mu^{2}n^{3}/16$.
Since $\sum_{e\in E}m_{e}=7|\cA_{v,c}|$ and $m_{e}\leq 5\mu n/4$ for all $e\in E$, it follows that $\sum_{e\in E}m_{e}^{2}\leq 35\mu n|\cA_{v,c}|/4$.
Therefore, by McDiarmid's Inequality, we have
\[
\prob{\cE_{v,c}}\leq\exp\left(-\frac{p^{12}q^{6}\beta^{14}|\cA_{v,c}|^{2}/4}{25\mu^{2}n^{3}/8 + 35\mu n|\cA_{v,c}|/4}\right)\leq\exp(-2^{-51}p^{12}q^{6}\beta^{14}\mu^{6}n),
\]
as claimed.
Thus, by a union bound, the probability that there exist $v\in V$, $c\in C$ such that~$\cE_{v,c}$ holds is at most $\exp(-2^{-52}p^{12}q^{6}\beta^{14}\mu^{6}n)$.

  We claim the following holds with high probability:
  \begin{enumerate}[(a)]
  \item\label{absorbing-sets-right-size} $|V_{\mathrm{abs}}| = (p \pm \eps)n$ and $|C_{\mathrm{abs}}| = (q \pm \eps)n$;
  \item\label{many-gadgets-survive} for every $v\in V$, $c\in C$, the event~$\cE_{v,c}$ does not hold;
  \item\label{weak-pseudorandomness-in-slice} for every $V^\circ\subseteq V_{\mathrm{abs}}$ and $C^\circ \subseteq C_{\mathrm{abs}}$ such that $|V^\circ|, |C^\circ| \geq \eps n$, there are at least $\beta \eps^3 n^2/ 200$ edges in $G'$ with both ends in $V^\circ$ and a colour in $C^\circ$.
  \end{enumerate}
  Indeed,~\ref{absorbing-sets-right-size} holds by the Chernoff Bound (Lemma~\ref{chernoff bounds}), we have already shown~\ref{many-gadgets-survive}, and since $\phi$ is $\eps$-locally edge-resilient,~\ref{weak-pseudorandomness-in-slice} holds by applying the Chernoff Bound for each $V^\circ$ and $C^\circ$ and using a union bound.

  Now we assume that~\ref{absorbing-sets-right-size}--\ref{weak-pseudorandomness-in-slice} hold\COMMENT{i.e. fix an outcome of all the random slicing in which~\ref{absorbing-sets-right-size}-\ref{weak-pseudorandomness-in-slice} hold. If there are no $H\cong \tilde H$ with bipartition $(V_{\mathrm{res}}, C_{\mathrm{res}})$ contained in $(V_{\mathrm{flex}}\cup V_{\mathrm{buff}}, C_{\mathrm{flex}} \cup C_{\mathrm{buff}})$, then the desired conclusion holds vacuously for such an outcome. Otherwise, fix any $H\cong \tilde H$ with bipartition $(V_{\mathrm{res}}, C_{\mathrm{res}})$ contained in $(V_{\mathrm{flex}}\cup V_{\mathrm{buff}}, C_{\mathrm{flex}} \cup C_{\mathrm{buff}})$, and we use~\ref{absorbing-sets-right-size}-\ref{weak-pseudorandomness-in-slice} to find the necessary extra structure for~$H$.}, we suppose $H\cong \tilde H$ has bipartition $(V_{\mathrm{res}}, C_{\mathrm{res}})$ contained in $(V_{\mathrm{flex}}\cup V_{\mathrm{buff}}, C_{\mathrm{flex}}\cup C_{\mathrm{buff}})$, and we show that~\ref{greedily-choosing-absorber} and~\ref{greedily-choosing-matching} hold. Arbitrarily order the edges of $H$ as $e_1, \dots, e_{|E(H)|}$.
  Let $i\in[|E(H)|]$ and suppose that for each $j<i$ we have found a $(v_{j},c_{j})$-absorbing gadget~$A_{j}$, where $e_{j}=v_{j}c_{j}$, and further, the collection~$\{A_{1},\dots,A_{i-1}\}$ satisfies the spanning subgraph of~$H$ containing precisely the edges~$e_{1},\dots,e_{i-1}$.  
  Writing $e_{i}=v_{i}c_{i}$, by~\ref{many-gadgets-survive} there is a collection~$\cA_{v_{i},c_{i}}^{\text{abs}}$ of at least $2^{-24}p^{6}q^{3}\beta^{7}\mu^{4}n^{2}$ $(v_{i},c_{i})$-absorbing gadgets each using only $V_{\text{abs}}$-vertices,~$C_{\text{abs}}$-colours, and~$G'$-edges, and moreover, each vertex in~$V_{\text{abs}}$, colour in~$C_{\text{abs}}$, and edge in~$G'$ is used by at most~$5\mu n/4$ of the $A\in\cA_{v_{i},c_{i}}^{\text{abs}}$.
  Thus, at most\COMMENT{Each of the $A_{j}\in\{A_{1},\dots,A_{i-1}\}$ chosen thus far has $6$ $V_{\text{abs}}$-vertices, $3$ $C_{\text{abs}}$-colours, and $7$ $G'$-edges. Thus for each $A_{j}$ we must remove at most $(6+3+7)\cdot 5\mu n/4=20\mu n$ of the gadgets in $\cA_{v_{i},c_{i}}^{\text{abs}}$ from consideration.} $20\mu n\cdot i\leq 20\mu n|E(H)|\leq 17920\eta\mu n^{2}$ of the $(v_{i},c_{i})$-absorbing gadgets in~$\cA_{v_{i},c_{i}}^{\text{abs}}$ use a vertex, colour, or edge used by any of the~$A_{j}$ for $j<i$.
  Since $|\cA_{v_{i},c_{i}}^{\text{abs}}|\geq 2^{-24}p^{6}q^{3}\beta^{7}\mu^{4}n^{2}$, we conclude\COMMENT{Recall $\eta\ll\mu$ and that $p,q,\beta=\Theta(\mu)$.} that there is at least one $(v_{i},c_{i})$-absorbing gadget $A\in\cA_{v_{i},c_{i}}^{\text{abs}}$ using vertices, colours, and edges which are disjoint from the vertices, colours, and edges used by~$A_{j}$, for all $j<i$.
  We arbitrarily choose such an~$A$ to be~$A_{i}$.
  Continuing in this way, it is clear that $\cA\coloneqq\{A_{i}\}_{i=1}^{|E(H)|}$ satisfies~$H$, so~\ref{greedily-choosing-absorber} holds.


  Now we prove~\ref{greedily-choosing-matching}.  
Let $V'_{\mathrm{abs}}$ and $C'_{\mathrm{abs}}$ be the vertices, colours, and edges in $V_{\mathrm{abs}}$ and $C_{\mathrm{abs}}$ not used by any $(v, c)$-absorbing gadget in $\mathcal A$.  By~\ref{absorbing-sets-right-size} and~(\ref{slicearray}), we have $|V'_{\mathrm{abs}}| = (2\mu \pm \eps)n$ and $|C'_{\mathrm{abs}}| = (\mu \pm \eps)n$.  Thus, by~\ref{weak-pseudorandomness-in-slice}, we can greedily choose a rainbow matching $M$ in $(V'_{\mathrm{abs}}, C'_{\mathrm{abs}}, G')$ of size at least $(\mu - 2\eps )n$, and $M$ satisfies~\ref{greedily-choosing-matching}.
\end{proof}

\subsection{Linking}

Lastly, we need the following lemma, inspired by~\cite[Lemma 20]{GKMO20}, which we use to both complete the set of absorbing gadgets obtained by Lemma~\ref{greedy-absorber-lemma} to an $H$-absorber and also construct its tail.
Recall that links were defined at the beginning of Section~\ref{overviewsec}.
\begin{lemma}\label{linking-lemma}
Consider an absorber partition of~$V$,~$C$, and~$K_{n}$.
  The following holds with high probability.  For every matching $M$ such that $V(M)\subseteq V_{\mathrm{abs}}$ and $|V_{\mathrm{abs}}\setminus V(M)| \leq \eps n$,
  there exists a collection $\mathcal P$ of links in~$G'$ such that
  \begin{enumerate}[(\ref{linking-lemma}.1), topsep = 6pt]
  \item\label{linking-matching} $\mathcal P$ links $M$ and
  \item\label{links-bounded-remainder} $\bigcup_{P\in\cP}P\setminus V(M)$ is contained in $(V_{\mathrm{link}} \cup V'_{\mathrm{link}}, C_{\mathrm{link}} \cup C'_{\mathrm{link}}, G')$ with $\gamma/2$-bounded remainder.
  \end{enumerate}
\end{lemma}
\begin{proof}
  We choose a new constant $\delta$ such that $\eps \ll \delta \ll \gamma$.  For convenience, let $p \coloneqq p_{\mathrm{link}}$ and $q \coloneqq q_{\mathrm{link}}$, let~$G_1$ be the spanning subgraph of~$G'$ consisting of edges with a colour in $C_{\mathrm{link}}$, and let $G_2$ be the spanning subgraph of~$G'$ consisting of edges with a colour in $C'_{\mathrm{link}}$.  
  First we claim that with high probability the following holds:
  \begin{enumerate}[(a)]
  \item\label{linking-sets-right-size} $|V_{\mathrm{link}}| = (p \pm \eps)n$, $|C_{\mathrm{link}}| = (q \pm \eps)n$, $|V'_{\mathrm{link}}| = (\gamma/3 \pm \eps)n$, and $|C'_{\mathrm{link}}| = (\gamma/3 \pm \eps)n$,
  \item\label{absorbing-vtcs-right-size} $|V_{\mathrm{abs}}| = (1 \pm \eps)p_{\mathrm{abs}} n = (1 \pm \eps)2pn/ 3$, 
  \item\label{linking-neighborhoods-right-size} for all $v\in V$, we have
    \begin{enumerate}[(i)]
    \item $|N_{G_1}(v) \cap V_{\mathrm{abs}}| = (1 \pm \eps)p_{\mathrm{abs}}\beta qn = (1 \pm \eps)2p\beta qn / 3$ and
    \item $|N_{G_1}(v) \cap V_{\mathrm{link}}| = (1 \pm \eps)p\beta q n$,
    \end{enumerate}
  \item\label{right-number-of-c-edges} for all $c \in C$, we have
    \begin{enumerate}[(i)]
    \item $|E_{G'}^c(V_{\mathrm{abs}}, V_{\mathrm{link}})| = (1 \pm \eps)p_{\mathrm{abs}}p\beta n = (1 \pm \eps)2p^2\beta n/3$ and
    \item $|E^c(G'[V_{\mathrm{link}}])| = (1 \pm \eps)p^2\beta n/2$,
    \end{enumerate}
  \item\label{common-linking-nbrhood-right-size}for all distinct $u,v\in V$, we have $|N_{G_1}(u)\cap N_{G_1}(v)\cap V_{\mathrm{link}}| = (1 \pm \eps)p\beta^2q^2 n$, and
  \item\label{common-linking-nbrhood-large-reserve}for all $u,v\in V$ we have $|N_{G_2}(u) \cap N_{G_2}(v) \cap V'_{\mathrm{link}}| \geq  \gamma^6 n$.
  \end{enumerate}

  Indeed~\ref{linking-sets-right-size}--\ref{right-number-of-c-edges} follow from~(\ref{slicearray}) and the Chernoff Bound.  We prove~\ref{common-linking-nbrhood-right-size} and~\ref{common-linking-nbrhood-large-reserve} using McDiarmid's Inequality.  To prove~\ref{common-linking-nbrhood-right-size}, for each $u, v \in V$, we apply McDiarmid's Inequality to the random variable $f$ counting $|N_{G_{1}}(u)\cap N_{G_{1}}(v)\cap V_{\mathrm{link}}|$ with respect to independent binomial random variables $\{X_w, X_{uw}, X_{vw}\}_{w\in V}$ and $\{X_c\}_{c\in C}$, where $X_w$ indicates if $w \in V_{\mathrm{link}}$, $X_{uw}$ and $X_{vw}$ indicate if the edges $uw$ and $vw$ respectively are in~$G'$, and~$X_{c}$ indicates if $c \in C_{\mathrm{link}}$. For each $w\in V$, $X_w$, $X_{uw}$, and $X_{vw}$ affect $f$ by at most one, and for each $c\in C$, $X_c$ affects $f$ by at most two.
  Thus, by McDiarmid's Inequality with $t = \eps \Expect{f}/2$, we have $|N_{G_{1}}(u) \cap N_{G_{1}}(v)\cap V_{\text{link}}| = (1 \pm \eps)p\beta^2q^2 n$ with probability at least $1 - \exp\left(-(\eps p\beta^2q^2 n/2)^2 / 7n\right)$. By a union bound,~\ref{common-linking-nbrhood-right-size} also holds with high probability.  The proof of~\ref{common-linking-nbrhood-large-reserve} is similar, so we omit it.

  Now we assume~\ref{linking-sets-right-size}--\ref{common-linking-nbrhood-large-reserve} hold\COMMENT{i.e. fix an outcome of all the random slicing in which~\ref{linking-sets-right-size}-\ref{common-linking-nbrhood-large-reserve} hold. If there is no~$M$ as in the hypothesis of the lemma then the conclusion holds vacuously. Otherwise we fix any such~$M$ and continue.}, we suppose $M$ is a matching such that $V(M)\subseteq V_{\mathrm{abs}}$ and $|V_{\mathrm{abs}}\setminus V(M)| \leq \eps n$, and we show that~\ref{linking-matching} and~\ref{links-bounded-remainder} hold with respect to $M$.  Since $|V_{\mathrm{abs}}\setminus V(M)| \leq \eps n$,~\ref{absorbing-vtcs-right-size} implies that
  \begin{equation}
    \label{eq:matching-right-size}
    |V(M)| = (1 \pm \sqrt\eps)2pn/3.
  \end{equation}
  
  We apply Theorem~\ref{hypergraph-matching-thm} to the following 8-uniform hypergraph $\mathcal H$: the vertex-set is $E(M) \cup V_{\mathrm{link}} \cup C_{\mathrm{link}}$, and for every $xy \in E(M)$, $v_1, v_2, v_3 \in V_{\mathrm{link}}$, and $c_1, c_2, c_3, c_4 \in C_{\mathrm{link}}$, $\mathcal H$ contains the hyperedge $\{xy, v_1, v_2, v_3, c_1, c_2, c_3, c_4\}$ if there is a main link $P$ such that
  \begin{itemize}
  \item $P$ has ends $x$ and $y$,
  \item $v_1$, $v_2$, and $v_3$ are the internal vertices in $P$, and
  \item $\phi(P) = \{c_1, c_2, c_3, c_4\}$.
  \end{itemize}

\begin{claim}\label{linklemc1}
$d_{\mathcal H}(v) = (1 \pm 2\sqrt{\eps})p^3\beta^4q^4 n^3$ for all $v\in V(\mathcal H)$.
\end{claim}
\claimproof{}
  Let $xy \in E(M)$.  
  By~\ref{linking-sets-right-size}, there are $(1 \pm \eps)pn$ vertices $v_1 \in V_{\mathrm{link}}$ that can be in a link $P$ with ends $x$ and $y$ corresponding to a hyperedge in $\mathcal H$, where $v_1$ is not adjacent to $x$ or $y$.  For each such $v_1\in V_{\mathrm{link}}$, by~\ref{common-linking-nbrhood-right-size}, there are $(1 \pm \eps)p\beta^2 q^2 n$ choices for the vertex in $V_{\mathrm{link}}$ adjacent to both $x$ and $v_1$ in $P$, and for each such $v_2 \in V_{\mathrm{link}}$, again by~\ref{common-linking-nbrhood-right-size}, there are $(1 \pm 2\eps)p\beta^2 q^2 n$ choices for the vertex in $V_{\mathrm{link}}$ adjacent to both $v_1$ and $y$ in $P$ such that $P$ is a main link.  Thus, $d_{\mathcal H}(xy) = (1 \pm 5\eps)p^3\beta^4 q^4 n^3$, as required.

  Now let $v_1 \in V_{\mathrm{link}}$.  First, we count the number of hyperedges in $\mathcal H$ containing $v_1$ corresponding to a link $P$ where $v_1$ is adjacent to one of the ends.  By~\ref{linking-neighborhoods-right-size}, and since $|V_{\mathrm{abs}}\setminus V(M)| \leq \eps n$, there are $(1 \pm \sqrt\eps)2p\beta qn/3$ choices of the vertex $x \in V(M)$ adjacent to $v_1$ in $P$.  For each such $x$, again by~\ref{linking-neighborhoods-right-size}, there are $(1 \pm 2\eps)p\beta q n$ choices of the vertex $v_2 \in V_{\mathrm{link}}$ adjacent to $y$ in $P$ where $xy\in E(M)$.  For each such $v_2 \in V_{\mathrm{link}}$, by~\ref{common-linking-nbrhood-right-size}, there are $(1 \pm 2\eps)p\beta^2 q^2 n$ choices of the vertex $v_3\in V_{\mathrm{link}}$ adjacent to both $v_1$ and $v_2$ in $P$.  Thus, the number of hyperedges in $\mathcal H$ containing $v_1$ corresponding to a link where $v_1$ is adjacent to one of the ends is $(1 \pm 2\sqrt\eps)2p^3\beta^4 q^4n^3/3$.

  Next, we count the number of hyperedges in $\mathcal H$ containing $v_1$ corresponding to a link $P$ where $v_1$ is not adjacent to one of the ends.  By~\eqref{eq:matching-right-size}, there are $(1 \pm \sqrt{\eps})pn/3$ choices for the edge $xy \in E(M)$ where $x$ and $y$ are the ends of $P$.  For each such $xy\in E(M)$, by~\ref{common-linking-nbrhood-right-size}, there are $(1 \pm \eps)p\beta^2 q^2 n$ choices of the vertex $v_2 \in V_{\mathrm{link}}$ such that $v_2$ is adjacent to $x$ and $v_1$ in $P$, and again by~\ref{common-linking-nbrhood-right-size}, for each such $v_2\in V_{\mathrm{link}}$, there are $(1 \pm 2\eps)p\beta^2 q^2n$ choices of the vertex $v_3\in V_{\mathrm{link}}$ adjacent to both $y$ and $v_1$ in $P$.  Thus, the number of hyperedges in $\mathcal H$ containing $v_1$ corresponding to a link where $v_1$ is not adjacent to one of the ends is $(1 \pm 2\sqrt{\eps})p^3\beta^4q^4 n^3/3$, so
  \begin{equation*}
    d_{\mathcal H}(v_1) = (1 \pm 2\sqrt\eps)(2p^3\beta^4 q^4 n^3/3) + (1 \pm 2\sqrt{\eps})p^3\beta^4q^4 n^3/3 = (1 \pm 2\sqrt{\eps})p^3\beta^4q^4 n^3,
  \end{equation*}
  as required.

  Now let $c_1 \in C_{\mathrm{link}}$.  First we count the number of hyperedges in $\mathcal H$ containing $c_1$ corresponding to a link $P$ where $c_1$ is the colour of one of the edges incident to an end of $P$.  By~\ref{right-number-of-c-edges}, and since $|V_{\mathrm{abs}}\setminus V(M)| \leq \eps n$, there are $(1 \pm \sqrt\eps)2p^2\beta n/3$ choices of the edge $xv_1$ in $P$ where $x\in V(M)$ is an end of $P$ and $\phi(xv_1) = c_1$.  For each such edge $xv_1$, by~\ref{linking-neighborhoods-right-size}, there are $(1\pm 2\eps) p\beta q n$ choices of the vertex $v_2\in V_{\mathrm{link}}$ adjacent to $y$ in $P$ where $xy \in E(M)$.  For each such vertex $v_2$, by~\ref{common-linking-nbrhood-right-size}, there are $(1 \pm 2\eps)p\beta^2q^2 n$ choices of the vertex $v_3$ adjacent to both $v_1$ and $v_2$ in $P$.  Thus, the number of hyperedges in $\mathcal H$ containing $c_1$ corresponding to a link where $c_1$ is the colour of one of the edges incident to an end of $P$ is $(1 \pm 2\sqrt\eps)2p^4\beta^4 q^3 n^3/ 3$.

  Next, we count the number of hyperedges in $\mathcal H$ containing $c_1$ corresponding to a link $P$ where $c_1$ is the colour of one of the edges with both ends in $V_{\mathrm{link}}$.  By~\ref{right-number-of-c-edges}, there are $(1 \pm \eps)p^2 \beta n/2$ choices for the edge $v_1v_2$ in $P$ such that $\phi(v_1v_2) = c_1$, and thus $(1\pm \eps)p^2 \beta n$ choices for the edge if we assume $v_1$ is adjacent to an end in $P$.  For each such edge $v_1v_2$, by~\ref{linking-neighborhoods-right-size}, and since $|V_{\mathrm{abs}}\setminus V(M)| \leq \eps n$, there are $(1 \pm \sqrt\eps)2p\beta q n/3$ choices for the vertex $x \in V(M)$ adjacent to $v_1$ in $P$.  For each such vertex $x$, by~\ref{common-linking-nbrhood-right-size}, there are $(1 \pm 2\eps)p\beta^2 q^2n$ choices for the vertex $v_3$ adjacent to both $y$ and $v_2$ in $P$, where $xy \in E(M)$.  Thus, the number of hyperedges in $\mathcal H$ containing $c_1$ corresponding to a link where $c_1$ is the colour of one of the edges with both ends in $V_{\mathrm{link}}$ is $(1 \pm 2\sqrt\eps)2p^4 \beta^4 q^3 n^3/ 3$, so by~(\ref{slicearray})
  \begin{equation*}
    d_{\mathcal H}(c_1) = (1\pm 2\sqrt{\eps})4p^4\beta^4 q^3 n^3 / 3 = (1\pm 2\sqrt{\eps})p^3\beta^4 q^4 n^3,
  \end{equation*}
  as required to prove Claim~\ref{linklemc1}.
  \endclaimproof{}
\begin{claim}\label{linklemc2}
 $\Delta^c(\mathcal H) \leq 100n^2$.
 \end{claim}
 This can be proved similarly as above (with room to spare).\COMMENT{\claimproof{}
 Let $xy \in E(M)$, $v_1, v_2 \in V_{\mathrm{link}}$, and $c_1, c_2 \in C_{\mathrm{link}}$.  There are at most $3n^2$ links with ends $x$ and $y$ containing $v_1$, so the codegree of pairs in $E(M)\times V_{\mathrm{link}}$ is at most $3n^2$, as required.  Similarly, there are at most $2n^2$ links with ends $x$ and $y$ such that the colour of the edge incident to $x$ or $y$ is $c_1$, and there are at most $2n^2$ links with ends $x$ and $y$ such that the colour of one of the edges not incident to $x$ or $y$ is $c_1$, so the codegree of pairs in $E(M)\times C_{\mathrm{link}}$ is at most $4n^2$, as required.  The codegree of pairs in $E(M)\times E(M)$ is zero.
  Now we count the number of hyperedges in $\mathcal H$ containing $v_1$ and $c_1$ corresponding to a link~$P$ where an edge $e$ in $P$ incident to $v_1$ is coloured $c_1$.  If $e$ is incident to an end $x$ of $P$, then the other end of $P$ is $y$ where $xy \in E(M)$.  In this case, there are at most $n^2$ choices for the internal vertices of $P$ other than $v_1$.  If $e$ is not incident to an end of $P$, then let $v_2$ denote the end of $e$ that is not $v_1$.  There are at most $|E(M)|$ choices for the ends $x$ and $y$ of $P$ and at most $n$ choices for the other internal vertex $v_3$ of $P$.  For each such choice of $xy \in E(M)$ and $v_3\in V_{\mathrm{link}}$, there are at most six links with these vertices.  Since $|E(M)| \leq n/2$, there are at most $4n^2$ links $P$ containing $v_1$ where an edge in $P$ incident to $v_1$ is coloured $c_1$.  Now we count the number of hyperedges in $\mathcal H$ containing $v_1$ and $c_1$ corresponding to a link $P$ containing $v_1$ and an edge $e$ coloured $c_1$ that is not incident to $v_1$.  If $e$ is incident to an end of $P$, say $x$, then the other end of $P$ is $y$ where $xy\in E(M)$.  In this case, there are at most~$n$ choices for the other internal vertex of $P$, and since there are at most $n/2$ possibilities for the edge $e$ coloured $c_1$, there are at most~$n^2$ such links.  If $e$ is not incident to an end of $P$, then the internal vertices of $P$ are determined. Since there are at most $n/2$ possibilities for the edge $e$ and $|E(M)| \leq n/2$, there are at most $n^2$ such links.  Thus, there are at most $2n^2$ links $P$ containing $v_1$ such that there is an edge in $P$ not incident to $v_1$ that is coloured $c_1$, so the codegree of pairs in $V_{\mathrm{link}}\times C_{\mathrm{link}}$ is at most $6n^2$, as required.
  Now we count the number of hyperedges in $\mathcal H$ containing $v_1$ and $v_2$, where $v_{1},v_{2}\in V_{\text{link}}$.  If $P$ is a link containing $v_1$ and $v_2$, then there are at most $n$ choices for the other internal vertex $v_3$ of $P$ and at most $n/2$ choices for the ends $x$ and $y$ of $P$.  For each such choice of $v_3, x$, and $y$, there are at most six links containing these vertices, so there at most $3n^2$ links containing $v_1$ and $v_2$.  Thus, the codegree of pairs in $V_{\mathrm{link}}\times V_{\mathrm{link}}$ is at most $3n^2$, as required.
  Next, we count the number of hyperedges in $\mathcal H$ containing $c_1$ and $c_2$ corresponding to a link $P$ containing edges $e_1$ and $e_2$ such that $\phi(e_i) = c_i$ for $i\in\{1, 2\}$ such that $e_1$ and $e_2$ do not share an end.  At least one of $e_1$ and $e_2$, say $e_1$, is incident to an end $x$ of $P$, and this determines the other end~$y$ of~$P$, where $xy\in E(M)$.  If $e_2$ is incident to $y$, then two of the internal vertices of $P$ are determined by $e_1$ and $e_2$, and there are at most $n$ choices for the other internal vertex.  If $e_2$ is not incident to $y$, then $e_1$ and $e_2$ determine the internal vertices of $P$, but there are up to $n/2$ possible choices for the edge $e_2$.  Since there are at most $n/2$ possible choices for the edge $e_1$, there are at most $2n^2$ such links.
    Finally, we count the number of hyperedges in $\mathcal H$ containing $c_1$ and $c_2$ corresponding to a link $P$ containing edges $e_1$ and $e_2$ such that $\phi(e_i) = c_i$ for $i\in\{1, 2\}$ such that $e_1$ and $e_2$ share an end $v_1$.  If $v_1$ is adjacent to an end $x$ of $P$, then the other end of $P$ is $y$ where $xy\in E(M)$, and there are at most $n$ choices for the internal vertex of $P$ not adjacent to $v_1$.  If $v_1$ is not adjacent to an end of $P$, then there are $|E(M)|$ choices for the ends of $P$.  Since there are at most $n$ choices for the vertex $v_1$, the total number of such links is at most $3n^2$.  Therefore the codegree of pairs in $C_{\mathrm{link}}\times C_{\mathrm{link}}$ is at most $5n^2$, as required.
 \endclaimproof{}}
    Let $\mathcal F \coloneqq \{E(M), V_{\mathrm{link}}, C_{\mathrm{link}}\}$. By Theorem~\ref{hypergraph-matching-thm}, $\mathcal H$ has a $(\delta, \mathcal F)$-perfect matching $\mathcal M$.  Let $\mathcal P_1$ be the collection of links corresponding to $\mathcal M$, and let $M'$ be the matching consisting of all those $xy\in E(M)$ that are not covered by $\mathcal M$.  To complete the proof, we greedily find a collection $\mathcal P_2$ of reserve links that links $M'$.
    
    Write $E(M') = \{x_1y_1, \dots, x_ky_k\}$, and suppose $P_i$ is a reserve link with ends $x_i$ and $y_i$ for $i < j$, where $j \in [k]$.  We show that that there is a reserve link $P_j$ that is vertex- and colour-disjoint from $\bigcup_{i < j}P_i$, which implies that $\bigcup_{i=1}^j P_i$ links $\{x_1y_1, \dots, x_jy_j\}$, and thus we can choose $\mathcal P_2$ greedily.  Since $k \leq \delta n$ and each link has at most three vertices in $V'_{\mathrm{link}}$, by~\ref{linking-sets-right-size}, there is a vertex $v \in V'_{\mathrm{link}}\setminus \bigcup_{i<j}V(P_i)$.  By~\ref{common-linking-nbrhood-large-reserve}, there are at least $\gamma^6 n - 11j$ vertices $v_1 \in (N_{G_2}(x_{j})\cap N_{G_2}(v) \cap V'_{\mathrm{link}})\setminus \bigcup_{i < j}V(P_i)$ such that $\phi(x_{j}v_1), \phi(v_{1}v) \notin \bigcup_{i < j}\phi(P_i)$, and since $j/n \leq \delta \ll \gamma$, we may let $v_1$ be such a vertex.  Similarly, by~\ref{common-linking-nbrhood-large-reserve}, there is a vertex $v_2 \in (N_{G_2}(y_{j})\cap N_{G_2}(v) \cap V'_{\mathrm{link}}) \setminus \bigcup_{i < j}V(P_i)$ such that $\phi(y_{j}v_2), \phi(v_{2}v) \notin \bigcup_{i < j}\phi(P_i) \cup \{\phi(x_{j}v_1), \phi(v_{1}v)\}$.  Now there is a reserve link $P_j$ with ends $x_j$ and $y_j$ and internal vertices $v$, $v_1$, and $v_2$ that is vertex- and colour-disjoint from $\bigcup_{i < j}P_{i}$, as claimed, and therefore there exists a collection $\mathcal P_2$ of reserve links that links $M'$.  Now $\mathcal P_1 \cup \mathcal P_2$ links $M$, so~\ref{linking-matching} holds.  By~\ref{linking-sets-right-size}, and since $\mathcal M$ is $(\delta, \mathcal F)$-perfect,~\ref{links-bounded-remainder} holds, as required.
\end{proof}

\subsection{Proof}
We now have all the tools we need to prove Lemma~\ref{main-absorber-lemma}.

\lateproof{Lemma~\ref{main-absorber-lemma}}
Consider an absorber partition of~$V$,~$C$, and~$K_{n}$.
  By Lemmas~\ref{absorbing-template-lemma},~\ref{greedy-absorber-lemma}, and~\ref{linking-lemma}, there exists an outcome of the absorber partition satisfying the conclusions of these lemmas simultaneously.  In particular, by Lemmas~\ref{absorbing-template-lemma} and~\ref{greedy-absorber-lemma} there exists $H, \mathcal A$, and $M$ such that, writing~$(V_{\text{res}},C_{\text{res}})$ for the bipartition of~$H$,
  \begin{itemize}
  \item $H\cong \tilde H$ is a $36\gamma$-absorbing template satisfying~\ref{template-flexible-sets} and~\ref{template-buffer-sets},
  \item $\mathcal A$ and $H$ satisfy \ref{greedily-choosing-absorber}, and
  \item $M$ satisfies \ref{greedily-choosing-matching}.
  \end{itemize}
  Write $\mathcal A = \{A_1, \dots, A_k\}$ and $E(M) = \{a_1b_1, \dots, a_\ell b_\ell\}$.
  Consider $M_1\,\dot{\cup}\, M_2\,\dot{\cup}\, M_3\,\dot{\cup}\, M_4$, where~$M_{i}$ is a matching satisfying~(Mi) for $i\in[4]$ (see Section~\ref{overviewsec}).
  By~\ref{greedily-choosing-matching} we have $|V_{\text{abs}}\setminus V(M_{1}\cup\dots\cup M_{4})|\leq 5\eps n+1\leq 6\eps n$.
  Thus by Lemma~\ref{linking-lemma} there exist collections of links $\mathcal P$ and $\mathcal T$ in~$G'$ such that
  \begin{itemize}
  \item $\mathcal P\cup\mathcal T$ is a collection of links satisfying~\ref{linking-matching} with respect to $\bigcup_{i=1}^4M_i$ and
  \item $\mathcal P\cup\mathcal T$ satisfies~\ref{links-bounded-remainder}.
  \end{itemize}
 In particular $\mathcal P$ links $\bigcup_{i=1}^3M_i$ and $\mathcal T$ links $M_4$. Let $T \coloneqq M\cup\bigcup_{P\in \mathcal T}P$. By Fact~\ref{factref},  $(\mathcal A, \mathcal P, T, H)$ is a $36\gamma$-absorber, as desired.  Moreover, since $H$ satisfies~\ref{template-flexible-sets} and~\ref{template-buffer-sets}, $M$ satisfies \ref{greedily-choosing-matching}, and $\mathcal P\cup\mathcal T$ satisfies~\ref{links-bounded-remainder}, we have $\bigcup_{A\in\mathcal A}A \cup \bigcup_{P\in \mathcal P}P \cup T$  is contained in $(V', C', G')$ with $\gamma$-bounded remainder, as required.
\endproof

\section{Finding many well-spread absorbing gadgets}\label{switching-section}
The aim of this section is to prove Lemma~\ref{main-switching-lemma}, which states that, for appropriate~$\mu,\eps$, almost all $1$-factorizations of~$K_{n}$ are $\eps$-locally edge-resilient and $\mu$-robustly gadget-resilient.
We will use switchings in~$\cG_{D}^{\text{col}}$ for appropriate $D\subsetneq [n-1]$ to analyse the probability that a uniformly random $G\in \cG_{D}^{\text{col}}$ satisfies the necessary properties, and then use a `weighting factor' (see Corollary~\ref{wf}) to make comparisons to the probability space corresponding to a uniform random choice of $G\in\cG_{[n-1]}^{\text{col}}$.
\subsection{Switchings}
We begin by analysing the property of $\eps$-local edge-resilience.
\begin{lemma}\label{localedge}
Suppose $1/n \ll \eps\ll1$, and let $D\subseteq[n-1]$ have size $|D|=\eps n$.
Suppose $\mathbf{G}\in\cG_{D}^{\text{col}}$ is chosen uniformly at random.
Then $\prob{\mathbf{G}\,\text{is}\,\,\eps\text{-locally}\,\text{edge-resilient}\,}\geq 1- \exp(-\eps^{3}n^{2}/1000)$.
\end{lemma}
\begin{proof}
Note that if $G\in\cG_{D}^{\text{col}}$ has at least~$\eps^{3}n^{2}/100$ edges with endpoints in~$V'$ for all choices of $V'\subseteq V$ of size precisely~$\eps n$, then~$G$ is $\eps$-locally edge-resilient.\COMMENT{In any larger sets $V'\subseteq V$, we can just find all the necessary edges in any subset $V''\subseteq V'$ of size precisely~$\eps n$. Further, of course, $D$ is the only subset of $D$ of size at least $\eps n$.}
Fix $V'\subseteq V$ of size precisely $\eps n$.
For any $G\in\cG_{D}^{\text{col}}$, we say that a subgraph $H\subseteq G$ together with a labelling of its vertices $V(H)=\{u,v,w,x,y,z\}$ is a \textit{spin system} of~$G$ if $E(H)=\{vw, xy, zu\}$, where $u,v\in V'$, $w,x,y,z\in V\setminus V'$, $uv, wx, yz\notin E(G)$, and $\phi_{G}(vw)=\phi_{G}(xy)=\phi_{G}(zu)$.
(Note that different labellings of a subgraph $H\subseteq G$ that both satisfy these conditions will be considered to correspond to different spin systems of~$G$.)
We now define the \textit{spin} switching operation.
Suppose $G\in\cG_{D}^{\text{col}}$ and $H\subseteq G$ is a spin system.
Then\COMMENT{In this case it may have been easier to just define $\text{spin}_{(u,v,w,x,y,z)}(G)$ for a sequence of vertices $u,v,w,x,y,z$, and avoid the need for defining a `spin system'. However, this approach won't work for our main switching later where the switching system has 14 vertices. It seems easiest to refer to a switching in terms of a subgraph like this, and to be consistent about it so I'm using it here. (And of course, a subgraph could correspond to several different switchings depending on which vertices play what role, so the labelling is important.)} we define~$\text{spin}_{H}(G)$ to be the coloured graph obtained from~$G$ by deleting the edges $vw,xy,zu$, and adding the edges $uv, wx, yz$, each with colour~$\phi_{G}(vw)$.
Writing $G'\coloneqq\text{spin}_{H}(G)$, we have $G'\in\cG_{D}^{\text{col}}$ and $e_{V',D}(G')=e_{V',D}(G)+1$.

We define a partition~$\{M_{s}\}_{s=0}^{\binom{\eps n}{2}}$ of~$\cG_{D}^{\text{col}}$ by setting $M_{s}\coloneqq \{G\in\cG_{D}^{\text{col}}\colon e_{V',D}(G)=s\}$, for each $s\in[\binom{\eps n}{2}]_{0}$.
For each $s\in[\binom{\eps n}{2}-1]_{0}$ we define an auxiliary bipartite multigraph\COMMENT{We can have two edges between $G\in M_{s}$ and $G'\in M_{s+1}$ if say there is an $H\subseteq G$ such that $\text{spin}_{H}(G)=G'$. In this case we find a second spin system $H'$ such that $\text{spin}_{H'}(G)=G'$, by swapping the roles of $u$ and $v$, swapping the roles of $w$ and $z$, and swapping the roles of $x$ and $y$.}~$B_{s}$ with vertex bipartition $(M_{s}, M_{s+1})$, where for each $G\in M_{s}$ and each spin system $H\subseteq G$ we put an edge in~$B_{s}$ with endpoints $G\in M_{s}$ and $\text{spin}_{H}(G)\in M_{s+1}$.
Define $\delta_{s}\coloneqq \min_{G\in M_{s}}d_{B_{s}}(G)$ and $\Delta_{s+1}\coloneqq \max_{G\in M_{s+1}}d_{B_{s}}(G)$.
Observe, by double counting~$e(B_{s})$, that $|M_{s}|/|M_{s+1}|\leq \Delta_{s+1}/\delta_{s}$.
To bound~$\Delta_{s+1}$ from above, we fix $G'\in M_{s+1}$ and bound the number of pairs~$(G,H)$, where $G\in M_{s}$ and~$H$ is a spin system of~$G$ such that $\text{spin}_{H}(G)=G'$.
There are~$s+1$ choices for the edge $e\in E_{V',D}(G')$ created by a spin operation, and~$2$ choices for which endpoint of~$e$ played the role of~$u$ in a spin, and which played the role of~$v$.
Now there are at most~$(n/2)^{2}$ choices for two edges with colour~$\phi_{G'}(e)$ in~$G'$ with both endpoints outside of~$V'$, and at most~$8$ choices for which endpoints of these edges played the roles of $w,x,y,z$ in a spin operation yielding~$G'$.
We deduce that $\Delta_{s+1}\leq 4(s+1)n^{2}$.

Suppose that $s\leq \eps^{3}n^{2}/80$.
To bound~$\delta_{s}$ from below, we fix $G\in M_{s}$ and find a lower bound for the number of spin systems\COMMENT{Even if spinning on two different spin systems yields the same graph, they will correspond to two different edges of~$B_{s}$.} $H\subseteq G$.
For a vertex $v\in V'$, let $D_{G}^{*}(v)\subseteq D$ denote the set of colours~$c\in D$ such that the $c$-neighbour of~$v$ is not in~$V'$, in~$G$.
Let $V_{G}^{*}\coloneqq\{v\in V'\colon |D_{G}^{*}(v)|\geq 9\eps n/10\}$, and suppose for a contradiction that $|V_{G}^{*}|< 9\eps n/10$.
Then there are at least~$\eps n/10$ vertices $v\in V'$ for which there are at least~$\eps n/10$ colours $c\in D$ such that the $c$-neighbour of~$v$ is in~$V'$, in~$G$, whence $s=e_{V',D}(G)\geq \eps^{2}n^{2}/200> \eps^{3}n^{2}/80\geq s$, a contradiction.
Note further that, since $s\leq \eps^{3}n^{2}/80$, there are at least $\binom{9\eps n/10}{2}-\eps^{3}n^{2}/80\geq \eps^{2}n^{2}/4$ pairs~$\{a,b\}\in\binom{V_{G}^{*}}{2}$ such that $ab\notin E(G)$.
For each such choice of~$\{a,b\}$, there are two choices of which vertex will play the role of~$u$ and which will play the role of~$v$ in a spin system.
Since $u,v\in V_{G}^{*}$, there are at least~$4\eps n/5$ colours $c\in D$ such that the $c$-neighbour~$z$ of~$u$, and the $c$-neighbour~$w$ of~$v$, are such that $w,z\in V\setminus V'$, in~$G$.
Finally, there are at least~$n/2-3\eps n\geq n/4$ edges coloured~$c$ in~$G$ with neither endpoint in~$V'\cup N_{G}(w)\cup N_{G}(z)$, and two choices of which endpoint of such an edge will play the role of~$x$, and which will play the role of~$y$.
We deduce that $\delta_{s}\geq \eps^{3}n^{4}/5$.
Altogether, we conclude that if $s\leq \eps^{3}n^{2}/80$ and~$M_{s}$ is non-empty, then~$M_{s+1}$ is non-empty and $|M_{s}|/|M_{s+1}|\leq 20(s+1)n^{2}/\eps^{3}n^{4} \leq 1/2$.

Now, fix $s\leq \eps^{3}n^{2}/100$.
If~$M_{s}$ is empty, then\COMMENT{Note $\prob{e_{V',D}(\mathbf{G})=s}=|M_{s}|/|\cG_{D}^{\text{col}}|$, so we need to know $\cG_{D}^{\text{col}}$ is non-empty to know that we are not dividing by zero. But this follows from the usual existence results of $1$-factorizations of complete graphs. (Just restrict such a $1$-factorization to our colour set of interest)} $\prob{e_{V',D}(\mathbf{G}) =s}=0$.
If~$M_{s}$ is non-empty, then
\[
\prob{e_{V',D}(\mathbf{G}) =s} = \frac{|M_{s}|}{|\cG_{D}^{\text{col}}|}\leq \frac{|M_{s}|}{|M_{\eps^{3}n^{2}/80}|}= \prod_{j=s}^{\eps^{3}n^{2}/80-1}\frac{|M_{j}|}{|M_{j+1}|} \leq \left(\frac{1}{2}\right)^{\eps^{3}n^{2}/80-s},
\]
and thus\COMMENT{Note the middle expression is bounded above by $\left(\frac{\eps^{3}n^{2}}{100}+1\right)\exp\left(-\frac{\eps^{3}n^{2}}{400}\ln2\right)$.} 
\[
\prob{e_{V',D}(\mathbf{G}) \leq \eps^{3}n^{2}/100} \leq \sum_{s=0}^{\eps^{3}n^{2}/100}\exp(-(\eps^{3}n^{2}/80-s)\ln2) \leq \exp\left(-\frac{\eps^{3}n^{2}}{800}\right).
\]
A union bound over all choices of $V'\subseteq V$ of size~$\eps n$ now completes the proof.\COMMENT{$\prob{\mathbf{G}\,\text{is}\,\text{not}\,\eps\text{-locally}\,\text{edge-resilient}}\leq \binom{n}{\eps n}\exp(-\eps^{3}n^{2}/800)\leq\exp(-\eps^{3}n^{2}/1000)$.}
\end{proof}
We now turn to showing that for suitable $D\subseteq [n-1]$, almost all $G\in\cG_{D}^{\text{col}}$ are robustly gadget-resilient, which turns out to be a much harder property to analyse than local edge-resilience, and we devote the rest of this section to it.
We first need to show that almost all $G\in\cG_{D}^{\text{col}}$ are `quasirandom', in the sense that small sets of vertices do not have too many crossing edges.
\begin{defin}
Let $D\subseteq [n-1]$.
We say that $G\in\cG_{D}^{\text{col}}$ is \textit{quasirandom} if for all sets $A,B\subseteq V$, not necessarily distinct, such that $|A|=|B|=|D|$, we have that $e_{G}(A,B)<8(|D|-1)^{3}/n$.
We define $\cQ_{D}^{\text{col}}\coloneqq\{G\in\cG_{D}^{\text{col}}\colon G\,\text{is}\,\text{quasirandom}\}$.\COMMENT{We are using this quasirandomness definition that doesn't see the colours, but investigating it in the context of coloured graphs, because we wish to know about the probability of obtaining this quasirandomness property in the space where we pick a coloured graph u.a.r., which is a fundamentally different probability space to the one in which we just pick regular graphs u.a.r.}
\end{defin}
When we are analysing switchings to study the property of robust gadget-resilience (see Lemma~\ref{masterswitch}), it will be important to condition on this quasirandomness.
One can use another switching argument to show that almost all $G\in\cG_{D}^{\text{col}}$ are quasirandom.
\begin{lemma}\label{quasirandom}
Suppose that $1/n\ll\mu\ll1$, let $D\subseteq [n-1]$ have size $|D|=\mu n+1$. 
Suppose that $\mathbf{G}\in\cG_{D}^{\text{col}}$ is chosen uniformly at random.
Then $\prob{\mathbf{G}\in\cQ_{D}^{\text{col}}}\geq 1- \exp(-\mu^{3}n^{2})$.
\end{lemma}
\begin{proof}
Fix $A,B\subseteq V$ satisfying $|A|=|B|=\mu n+1$.
For any $G\in\cG_{D}^{\text{col}}$, we say that a subgraph $H\subseteq G$ together with a labelling of its vertices $V(H)=\{a,b,v,w\}$ is a \textit{rotation system} of~$G$ if $E(H)=\{ab, vw\}$, where $a\in A$, $b\in B$, $v,w\notin A\cup B$, $aw, bv\notin E(G)$, and $\phi_{G}(ab)=\phi_{G}(vw)$. 
We now define the \textit{rotate} switching operation.
Suppose $G\in\cG_{D}^{\text{col}}$ and $H\subseteq G$ is a rotation system.
Then we define~$\text{rot}_{H}(G)$ to be the coloured graph obtained from~$G$ by deleting the edges $ab,vw$, and adding the edges $aw, bv$, each with colour~$\phi_{G}(ab)$.
Writing $G'\coloneqq\text{rot}_{H}(G)$, notice that $G'\in\cG_{D}^{\text{col}}$ and $e_{G'}(A,B)=e_{G}(A,B)-1$.

Lemma~\ref{quasirandom} follows by analysing the degrees of auxiliary bipartite multigraphs~$B_{s}$ in a similar way as in the proof of Lemma~\ref{localedge}. We omit the details.\COMMENT{We define a partition~$\{M_{s}\}_{s=0}^{(\mu n+1)^{2}}$ of~$\cG_{D}^{\text{col}}$ by setting $M_{s}\coloneqq\{G\in\cG_{D}^{\text{col}}\colon e_{G}(A,B)=s\}$ for each $s\in[(\mu n+1)^{2}]_{0}$.\COMMENT{Depending on $A,B$ (say for example $A=B$), it might be that some of $M_{(\mu n+1)^{2}}, M_{(\mu n+1)^{2}-1},\dots$ are empty, but this would just mean the partition has some empty parts. Regardless, $(\mu n+1)^{2}$ is of course a universal upper bound for~$e_{G}(A,B)$.}
For each $s\in[(\mu n+1)^{2}]$ we define an auxiliary bipartite multigraph\COMMENT{We can have two edges between $G\in M_{s}$ and $G'\in M_{s-1}$ if say there is an~$H\subseteq G$ such that $\text{rot}_{H}(G)=G'$ and $a,b\in A\cap B$. In this case we find a second rotation system $H'$ such that $\text{rot}_{H'}(G)=G'$, by swapping the roles of $a$ and $b$, and swapping the roles of $v$ and $w$.}~$B_{s}$ with vertex bipartition~$(M_{s-1}, M_{s})$, where for each $G\in M_{s}$ and each rotation system $H\subseteq G$ we put an edge in~$B_{s}$ with endpoints~$G\in M_{s}$ and $\text{rot}_{H}(G)\in M_{s-1}$.
Define $\delta_{s}\coloneqq \min_{G\in M_{s}}d_{B_{s}}(G)$ and $\Delta_{s-1}\coloneqq \max_{G\in M_{s-1}}d_{B_{s}}(G)$.
Thus $|M_{s}|/|M_{s-1}|\leq \Delta_{s-1}/\delta_{s}$.
To bound~$\Delta_{s-1}$ from above, we fix $G'\in M_{s-1}$ and bound the number of pairs~$(G,H)$, where $G\in M_{s}$ and~$H$ is a rotation system of~$G$ such that $\text{rot}_{H}(G)=G'$.
There are at most $(\mu n+1)^{2}$ choices for which pair of vertices $a\in A, b\in B$ were caused to be a non-edge in~$G'$ by a rotation.
Then there are at most~$\mu n+1$ choices for a colour $d\in D$ such that the $d$-neighbour of~$b$ and the $d$-neighbour of~$a$ in~$G'$ could have played the roles of~$v$ and~$w$ respectively in a rotation yielding~$G'$.
We deduce that $\Delta_{s-1}\leq (\mu n+1)^{3}$.
%
%
%
To bound~$\delta_{s}$ from below, we fix $G\in M_{s}$ and find a lower bound for the number of rotation systems $H\subseteq G$.
Notice that there are~$s$ choices of an edge~$e\in E(G)$ with an endpoint in~$A$ and an endpoint in~$B$, and at least one choice of which of these endpoints will play the role of $a\in A$ in a rotation system~$H$, and which will play the role of $b\in B$.
Observe that $|A\cup B \cup N_{G}(a)\cup N_{G}(b)|\leq 4\mu n+4$, and therefore there are at least $n/2-4\mu n-4\geq n/4$ edges $f\in E_{\phi_{G}(e)}(G)$ such that both endpoints of~$f$ are in $V\setminus(A\cup B\cup N_{G}(a)\cup N_{G}(b))$, and there are two choices of which endpoint of~$f$ will play the role of~$v$, and which will play the role of~$w$.
Thus $\delta_{s}\geq sn/2$.
We conclude that if $s\geq 5\mu^{3}n^{2}$ and~$M_{s}$ is non-empty, then~$M_{s-1}$ is non-empty and $|M_{s}|/|M_{s-1}|\leq 2(\mu n+1)^{3}/sn \leq 1/2$.
%
Now, fix $s\geq 8\mu^{3}n^{2}$.
If~$M_{s}$ is empty, then\COMMENT{Again, we need to know $\cG_{D}^{\text{col}}$ is non-empty to know that we are not dividing by zero. But this follows from the usual existence results of $1$-factorizations of complete graphs. (Just restrict such a $1$-factorization to our colour set of interest)} $\prob{e_{\mathbf{G}}(A,B)=s}=0$.
If~$M_{s}$ is non-empty, then
\[
\prob{e_{\mathbf{G}}(A,B)=s}=\frac{|M_{s}|}{|\cG_{D}^{\text{col}}|}\leq \frac{|M_{s}|}{|M_{5\mu^{3}n^{2}}|}=\prod_{j=0}^{s-5\mu^{3}n^{2}-1}\frac{|M_{s-j}|}{|M_{s-j-1}|}\leq \left(\frac{1}{2}\right)^{s-5\mu^{3}n^{2}},
\]
and thus
\begin{eqnarray*}
\prob{e_{\mathbf{G}}(A,B)\geq 8\mu^{3}n^{2}} & \leq & \sum_{s=8\mu^{3}n^{2}}^{(\mu n+1)^{2}}\exp(-(s-5\mu^{3}n^{2})\ln 2)\leq (\mu n+1)^{2}\exp\left(-3\mu^{3}n^{2}\ln 2\right) \\ & \leq & \exp\left(-2\mu^{3}n^{2}\right).
\end{eqnarray*}
Finally, by a union bound over all choices of $A,B\subseteq V$, each of size $\mu n+1$, we conclude that $\prob{\mathbf{G}\notin\cQ_{D}^{\text{col}}}\leq\binom{n}{\mu n+1}^{2}\exp(-2\mu^{3}n^{2})\leq\exp(-\mu^{3}n^{2})$, which completes the proof of the lemma.}
\end{proof}
Next we will use a switching argument to find a large set of well-spread absorbing gadgets (cf.\ Definition~\ref{spread}).
For this, we consider slightly more restrictive substructures than the absorbing gadgets defined in Definition~\ref{def:absorbing-gadget}.
These additional restrictions (an extra edge~$f$ as well as an underlying partition~$\cP$ of the colours) give us better control over the switching process: they allow us to argue that we do not create more than one additional gadget per switch.
Let $D\subseteq [n-1]$, $c\in[n-1]\setminus D$, write $D^{*}\coloneqq D\cup\{c\}$, and let $G\in\cG_{D^{*}}^{\text{col}}$.
Suppose that~$\cP=\{D_{i}\}_{i=1}^{4}$ is an (ordered) partition of~$D$ into four subsets, and let~$x\in V$.
\begin{defin}
An $(x,c,\cP)$\textit{-gadget} in~$G$ is a subgraph~$J=A\cup\{f\}$ of~$G$ the following form (see Figure~\ref{fig:xcp}):
\begin{enumerate}[label=\upshape(\roman*)]
\item $A$ is an $(x,c)$-absorbing gadget in~$G$;
\item there is an edge $e_{1}\in\partial_{A}(x)$ such that $\phi(e_{1})\in D_{1}$, and the remaining edge $e_{2}\in\partial_{A}(x)$ satisfies $\phi(e_{2})\in D_{2}$;
\item the edge~$e_{3}$ of~$A$ which is not incident to~$x$ but shares an endvertex with~$e_{1}$ and an endvertex with~$e_{2}$ satisfies $\phi(e_{3})\in D_{3}$;
\item $f=xv$ is an edge of~$G$, where $v$ is the unique vertex of~$A$ such that $\phi(\partial_{A}(v))=\{c,\phi(e_{1})\}$;
\item $\phi(f)\in D_{4}$.
\end{enumerate}
\end{defin}
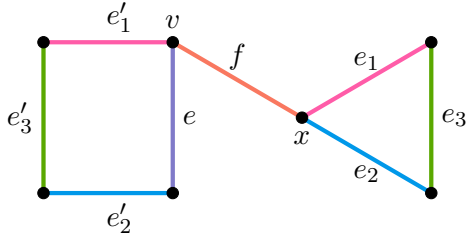
\begin{figure}
\centering
\begin{tikzpicture} [scale=0.5]
\draw [OwlRed,ultra thick] (0,4)--(3.4,4);
\draw [OwlBlue,ultra thick] (0,0)--(3.4,0);
\draw [OwlGreen,ultra thick] (0,0)--(0,4);
\draw [OwlViolet,ultra thick] (3.4,0)--(3.4,4);
\draw [OwlOrange,ultra thick] (3.4,4)--(6.8,2);
\draw [OwlRed,ultra thick] (6.8,2)--(10.2,4);
\draw [OwlBlue,ultra thick] (6.8,2)--(10.2,0);
\draw [OwlGreen,ultra thick] (10.2,0)--(10.2,4);
\draw [fill] (0,0) circle [radius=0.15];
\draw [fill] (0,4) circle [radius=0.15];
\draw [fill] (3.4,0) circle [radius=0.15];
\draw [fill] (3.4,4) circle [radius=0.15];
\draw [fill] (6.8,2) circle [radius=0.15];
\draw [fill] (10.2,0) circle [radius=0.15];
\draw [fill] (10.2,4) circle [radius=0.15];
\node [below] at (6.8,1.9) {$x$};
\node [left] at (0,2) {$e_{3}'$};
\node [below] at (2,0) {$e_{2}'$};
\node [above] at (2,4) {$e_{1}'$};
\node [right] at (3.4,2) {$e$};
\node [above] at (5.1,3) {$f$};
\node [above] at (8.5,3) {$e_{1}$};
\node [below] at (8.5,1) {$e_{2}$};
\node [right] at (10.2,2) {$e_{3}$};
\node [above] at (3.4,4.1) {$v$};
\end{tikzpicture}
\caption{An $(x,c,\cP)$-gadget. Here, $\phi(f)\in D_{4}$, $\phi(e)=c$, and $\phi(e_{i})=\phi(e_{i}')\in D_{i}$ for each $i\in[3]$.}
\label{fig:xcp}
\end{figure}
%
We now define some terminology that will be useful for analysing how many $(x,c,\cP)$-gadgets there are in a graph $G\in\cG_{D^{*}}^{\text{col}}$, and how well-spread these gadgets are.
Each of the terms we define here will have a dependence on the choice of the triple~$(x,c,\cP)$, but since this triple will always be clear from context, for presentation we omit the $(x,c,\cP)$-notation.
\begin{defin}
We say that an $(x,c,\cP)$-gadget~$J$ in~$G$ is \textit{distinguishable} in~$G$ if the edges $e_{3}, e_{3}'$ of~$J$ such that $\phi(e_{3})=\phi(e_{3}')\in D_{3}$ are such that there is no other $(x,c,\cP)$-gadget $J'\neq J$ in~$G$ such that $e_{3}\in E(J')$ or $e_{3}'\in E(J')$.
\end{defin}
We will aim only to count distinguishable $(x,c,\cP)$-gadgets, which will ensure the collection of gadgets we find is well-spread across the set of edges in $G\in\cG_{D^{*}}^{\text{col}}$ that can play the roles of~$e_{3}$,~$e_{3}'$.
We also need to ensure that the collection of gadgets we find is well-spread across the $c$-edges of~$G$.
\begin{defin}~
\begin{itemize}
\item For each $c$-edge~$e$ of~$G\in\cG_{D^{*}}^{\text{col}}$, we define the \textit{saturation} of~$e$ in~$G$, denoted~$\text{sat}_{G}(e)$, or simply~$\text{sat}(e)$ when~$G$ is clear from context, to be the number of distinguishable $(x,c,\cP)$-gadgets of~$G$ which contain~$e$.
We say that~$e$ is \textit{unsaturated} in~$G$ if $\text{sat}(e)\leq |D|-1$, \textit{saturated} if $\text{sat}(e)\geq |D|$, and \textit{supersaturated} if $\text{sat}(e)\geq |D|+6$.
We define~$\text{Sat}(G)$ to be the set of saturated $c$-edges of~$G$, and~$\text{Unsat}(G)\coloneqq E_{c}(G)\setminus\text{Sat}(G)$. 
\item We define the function\COMMENT{Actually, since (later)~$\cP$ is equitable,~$r(G)$ cannot exceed $n|D|/16$ as there are only $n|D|/8$ edges of~$G$ with colours in~$D_{3}$, and each distinguishable $(x,c,\cP)$-gadget uses precisely two of these, and no such edge is in more than one distinguishable $(x,c,\cP)$-gadget by definition.} $r\colon \cG_{D^{*}}^{\text{col}}\rightarrow [n|D|/2]_{0}$ by
\[
r(G)\coloneqq |D||\text{Sat}(G)|+\sum_{e\in\text{Unsat}(G)}\text{sat}(e).
\]
\end{itemize}
\end{defin}

In Lemma~\ref{masterswitch}, we will use switchings to show that~$r(G)$ is large (for some well-chosen~$\cP$) in almost all quasirandom $G\in\cG_{D^{*}}^{\text{col}}$.
In Lemma~\ref{justgadgets}, we use distinguishability, saturation, and the fact that any non-$x$ vertex in an $(x,c,\cP)$-gadget must be incident to an edge playing the role of either~$e_{3}$,~$e_{3}'$, or the $c$-edge, to show that~$r(G)$ being large means that there are many well-distributed $(x,c,\cP)$-gadgets in~$G$, and thus many well-spread $(x,c)$-absorbing gadgets.
We now define a relaxation of~$\cQ_{D^{*}}^{\text{col}}$, which will be a convenient formulation for ensuring that quasirandomness is maintained when we use switchings to find $(x,c,\cP)$-gadgets.
For each $s\in[n|D|/2]_{0}$, we write $A_{s}^{D^{*}}\coloneqq\{G\in\cG_{D^{*}}^{\text{col}}\colon r(G)=s\}$, and
\[
Q_{s}^{D^{*}}\coloneqq\{G\in\cG_{D^{*}}^{\text{col}}\colon e_{G}(A,B)< 8|D|^{3}/n + 6s\,\,\, \text{for}\,\text{all}\,A,B\subseteq V\,\text{such}\,\text{that}\,|A|=|B|=|D|\}.
\]
We also define $T_{s}^{D^{*}}\coloneqq A_{s}^{D^{*}}\cap Q_{s}^{D^{*}}$ and $\widetilde{\cQ}_{D^{*}}^{\text{col}}\coloneqq \bigcup_{s=0}^{n|D|/2}T_{s}^{D^{*}}$.
Notice that\COMMENT{In the definition of~$\cQ_{D^{*}}^{\text{col}}$, we look at $A$, $B$ of size $|D^{*}|=|D|+1$, while in the definition of~$\widetilde{\cQ}_{D^{*}}^{\text{col}}$, we look at $A$, $B$ of size~$|D|$, but this inclusion still holds.}
\begin{equation}\label{eq:qsubset}
\cQ_{D^{*}}^{\text{col}}\subseteq \widetilde{\cQ}_{D^{*}}^{\text{col}}.
\end{equation}

Finally, we discuss the switching operation that we will use in Lemma~\ref{masterswitch}.
\begin{defin}~
For any $G\in\cG_{D^{*}}^{\text{col}}$, we say that a subgraph $H\subseteq G$ together with a labelling of its vertices $V(H)=\{x,u_{1},u_{2},\dots,u_{14}\}$ is a \textit{twist system} of~$G$ if (see Figure~\ref{fig:ts}):
\begin{enumerate}[label=\upshape(\roman*)]
    \item $E(H)=\{u_{1}u_{2}, u_{3}u_{5}, u_{4}u_{6}, u_{5}u_{7}, u_{6}u_{8}, u_{7}u_{8}, u_{7}x, xu_{9}, xu_{10}, u_{9}u_{11}, u_{10}u_{12}, u_{13}u_{14}\}$;
    \item $\phi(u_{5}u_{7})=\phi(xu_{9})\in D_{1}$;
    \item $\phi(u_{6}u_{8})=\phi(xu_{10})\in D_{2}$;
    \item $\phi(u_{1}u_{2})=\phi(u_{3}u_{5})=\phi(u_{4}u_{6})=\phi(u_{9}u_{11})=\phi(u_{10}u_{12})=\phi(u_{13}u_{14})\in D_{3}$;
    \item $\phi(u_{7}x)\in D_{4}$;
    \item $\phi(u_{7}u_{8})=c$;
     \item $u_{1}u_{3}, u_{2}u_{4}, u_{5}u_{6}, u_{9}u_{10}, u_{11}u_{13}, u_{12}u_{14}\notin E(G)$.
\end{enumerate}
For a twist system~$H$ of~$G$, we define~$\text{twist}_{H}(G)$ to be the coloured graph obtained from~$G$ by deleting the edges $u_{1}u_{2}$, $u_{3}u_{5}$, $u_{4}u_{6}$, $u_{9}u_{11}$, $u_{10}u_{12}$, $u_{13}u_{14}$, and adding the edges $u_{1}u_{3}$, $u_{2}u_{4}$, $u_{5}u_{6}$, $u_{9}u_{10}$, $u_{11}u_{13}$, $u_{12}u_{14}$, each with colour~$\phi_{G}(u_{1}u_{2})$.\COMMENT{Writing $G'=\text{twist}_{H}(G)$, notice that since~$\phi_{G}$ is a $1$-factorization of~$G$, we have that the colouring~$\phi_{G'}$ of~$G'$ is a $1$-factorization of~$G'$.}
The $(x,c,\cP)$-gadget in~$\text{twist}_{H}(G)$ with edges $u_{5}u_{6}$, $u_{5}u_{7}$, $u_{6}u_{8}$, $u_{7}u_{8}$, $u_{7}x$, $xu_{9}$, $xu_{10}$, $u_{9}u_{10}$ is called the \textit{canonical} $(x,c,\cP)$\textit{-gadget} of the twist.
\end{defin}
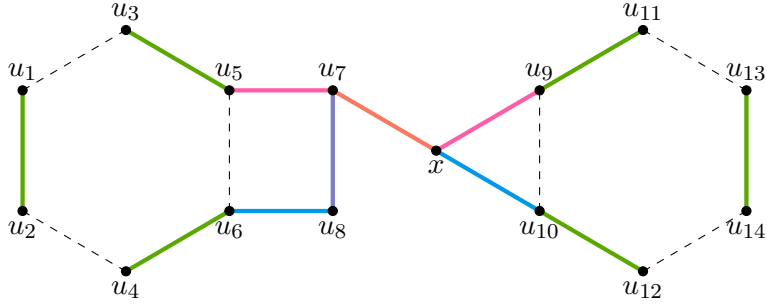
\begin{figure}
\centering
\begin{tikzpicture} [scale=0.4]
\draw [OwlGreen,ultra thick] (0,0)--(0,4);
\draw [OwlGreen,ultra thick] (3.4,-2)--(6.8,0);
\draw [OwlGreen,ultra thick] (3.4,6)--(6.8,4);
\draw [OwlRed,ultra thick] (6.8,4)--(10.2,4);
\draw [OwlBlue,ultra thick] (6.8,0)--(10.2,0);
\draw [OwlViolet,ultra thick] (10.2,0)--(10.2,4);
\draw [OwlOrange,ultra thick] (10.2,4)--(13.6,2);
\draw [OwlRed,ultra thick] (13.6,2)--(17,4);
\draw [OwlBlue,ultra thick] (13.6,2)--(17,0);
\draw [OwlGreen,ultra thick] (17,0)--(20.4,-2);
\draw [OwlGreen,ultra thick] (17,4)--(20.4,6);
\draw [OwlGreen,ultra thick] (23.8,0)--(23.8,4);
\draw [dashed] (0,0)--(3.4,-2);
\draw [dashed] (0,4)--(3.4,6);
\draw [dashed] (6.8,0)--(6.8,4);
\draw [dashed] (17,0)--(17,4);
\draw [dashed] (20.4,-2)--(23.8,0);
\draw [dashed] (20.4,6)--(23.8,4);
\draw [fill] (0,0) circle [radius=0.15];
\draw [fill] (0,4) circle [radius=0.15];
\draw [fill] (3.4,-2) circle [radius=0.15];
\draw [fill] (3.4,6) circle [radius=0.15];
\draw [fill] (6.8,0) circle [radius=0.15];
\draw [fill] (6.8,4) circle [radius=0.15];
\draw [fill] (10.2,0) circle [radius=0.15];
\draw [fill] (10.2,4) circle [radius=0.15];
\draw [fill] (13.6,2) circle [radius=0.15];
\draw [fill] (17,0) circle [radius=0.15];
\draw [fill] (17,4) circle [radius=0.15];
\draw [fill] (20.4,-2) circle [radius=0.15];
\draw [fill] (20.4,6) circle [radius=0.15];
\draw [fill] (23.8,0) circle [radius=0.15];
\draw [fill] (23.8,4) circle [radius=0.15];
\node [below] at (0,0) {$u_{2}$};
\node [above] at (0,4) {$u_{1}$};
\node [above] at (3.4,6) {$u_{3}$};
\node [below] at (3.4,-2) {$u_{4}$};
\node [above] at (6.8,4) {$u_{5}$};
\node [below] at (6.8,0) {$u_{6}$};
\node [below] at (10.2,0) {$u_{8}$};
\node [above] at (10.2,4) {$u_{7}$};
\node [below] at (13.6,2) {$x$};
\node [above] at (17,4) {$u_{9}$};
\node [below] at (17,0) {$u_{10}$};
\node [above] at (20.4,6) {$u_{11}$};
\node [below] at (20.4,-2) {$u_{12}$};
\node [above] at (23.8,4) {$u_{13}$};
\node [below] at (23.8,0) {$u_{14}$};
\end{tikzpicture}
\caption{A twist system of~$G$. Here, dashed edges represent non-edges of~$G$, and the colours of the edges satisfy (ii)--(vi) in the definition of twist system.}
\label{fig:ts}
\end{figure}
We simultaneously switch two edges into the positions~$u_{5}u_{6}$ and~$u_{9}u_{10}$ because it is much easier to find structures as in Figure~\ref{fig:ts} than it is to find such a structure with one of these edges already in place.
Moreover, the two `switching cycles' we use have three edges and three non-edges (rather than two of each, as in the rotation switching) essentially because of the extra freedom this gives us when choosing the edges~$u_{1}u_{2}$ and~$u_{13}u_{14}$.
This extra freedom allows us to ensure that in almost all twist systems, one avoids undesirable issues like inadvertently creating more than one new gadget when one performs the twist.

The proof of Lemma~\ref{masterswitch} proceeds with a similar strategy to those of Lemmas~\ref{localedge} and~\ref{quasirandom}, but it is much more challenging this time to show that graphs with low~$r(G)$-value admit many ways to switch to yield a graph $G'\in\cG_{D^{*}}^{\text{col}}$ satisfying $r(G')=r(G)+1$.
\begin{lemma}\label{masterswitch}
Suppose that $1/n\ll\mu\ll1$, and let $D\subseteq [n-1]$ have size $|D|=\mu n$.
Let $x\in V$, let $c\in [n-1]\setminus D$, and let $\cP=\{D_{i}\}_{i=1}^{4}$ be an equitable partition of~$D$.
Suppose that $\mathbf{G}\in\cG_{D\cup\{c\}}^{\text{col}}$ is chosen uniformly at random.
Then
\[
\prob{r(\mathbf{G})\leq\frac{\mu^{4}n^{2}}{2^{23}}\biggm| \mathbf{G}\in\widetilde{\cQ}_{D\cup\{c\}}^{\text{col}}}\leq\exp\left(-\frac{\mu^{4}n^{2}}{2^{24}}\right).
\]
\end{lemma}
\begin{proof}
Write $D^{*}\coloneqq D\cup\{c\}$.
Consider the partition~$\{T_{s}^{D^{*}}\}_{s=0}^{nk/2}$ of~$\widetilde{\cQ}_{D^{*}}^{\text{col}}$, where $k\coloneqq |D|$.
For each $s\in[nk/2-1]_{0}$, we define an auxiliary bipartite multigraph\COMMENT{Suppose~$H$ is a twist system with colours~$i\in D_{i}$ and vertices as we usually label them. It is possible there is another twist system~$H'$ using colours $1'\neq 1$, $2'\neq 2$, $3$, $4'\neq 4$, with vertices $\{x,v_{1},\dots, v_{14}\}$ (here making the natural numerical correspondence of vertex roles in this order) such that $v_{1}=u_{6}$, $v_{2}=u_{4}$, $v_{3}=u_{5}$, $v_{4}=u_{2}$, $v_{5}=u_{3}$, $v_{6}=u_{1}$, $v_{9}=u_{13}$, $v_{10}=u_{11}$, $v_{11}=u_{14}$, $v_{12}=u_{9}$, $v_{13}=u_{12}$, $v_{14}=u_{10}$, and using an entirely different $c$-edge. These two twists actually operate on the same $3$-edges and non-edges, so that $\text{twist}_{H}(G)=\text{twist}_{H'}(G)$. It is possible that the saturation of both $c$-edges increases by precisely one (and is below the saturation threshold), while some other $c$-edge loses saturation one, so that $G'\in T_{s+1}$, so that indeed~$B_{s}$ may be a multigraph. It would be possible to rule this out in the definition of adjacency in~$B_{s}$, but I tried to make the definition no more convoluted than it already needed to be - plus this added complication will not bother us.}~$B_{s}$ with vertex bipartition $(T_{s}^{D^{*}},T_{s+1}^{D^{*}})$ and an edge between~$G$ and~$\text{twist}_{H}(G)$ whenever:
\begin{enumerate}
    \item[(a)] $G\in T_{s}^{D^{*}}$;
    \item[(b)] $H$ is a twist system in~$G$ for which~$G'\coloneqq\text{twist}_{H}(G)\in T_{s+1}^{D^{*}}$  and~$G'$ satisfies $\text{sat}_{G'}(e)=\text{sat}_{G}(e)+1\leq k$ for the $c$-edge $e=u_{7}u_{8}$ of~$H$, with the canonical $(x,c,\cP)$-gadget of the twist~$G'$ being the only additional distinguishable $(x,c,\cP)$-gadget using this $c$-edge.
\end{enumerate}
Define $\delta_{s}\coloneqq \min_{G\in T_{s}^{D^{*}}}d_{B_{s}}(G)$ and $\Delta_{s+1}\coloneqq \max_{G\in T_{s+1}^{D^{*}}}d_{B_{s}}(G)$.
Thus $|T_{s}^{D^{*}}|/|T_{s+1}^{D^{*}}|\leq \Delta_{s+1}/\delta_{s}$.
To bound~$\Delta_{s+1}$ from above, we fix $G'\in T_{s+1}^{D^{*}}$ and bound the number of pairs~$(G,H)$, where $G\in T_{s}^{D^{*}}$ and~$H$ is a twist system of~$G$ such that $\text{twist}_{H}(G)=G'$ and~(b) holds.
Firstly, note that
\[
\sum_{\substack{e\in E_{c}(G') \\ \text{sat}_{G'}(e)\leq k}}\text{sat}_{G'}(e) \leq r(G')= s+1.
\]
Thus, it follows from condition (b) that there are at most~$s+1$ choices for the canonical $(x,c,\cP)$-gadget of a twist yielding~$G'$ for which we record an edge in~$B_{s}$.
Fixing this $(x,c,\cP)$-gadget fixes the vertices of~$V$ which played the roles of $x$, $u_{5}$, $u_{6}$, $\dots$, $u_{10}$ in a twist yielding~$G'$.
To determine all possible sets of vertices playing the roles of $u_{1}$, $u_{2}$, $u_{3}$, $u_{4}$, $u_{11}$, $u_{12}$, $u_{13}$, $u_{14}$ (thus determining~$H$ and~$G$ such that $\text{twist}_{H}(G)=G'$), it suffices to find all choices of four edges of~$G'$ with colour~$\phi_{G'}(u_{5}u_{6})$ satisfying the necessary non-adjacency conditions.
There are at most~$(n/2)^{4}$ choices for these four edges, and at most~$4!\cdot2^{4}$ choices for which endpoints of these edges play which role.\COMMENT{We need to choose which of the four edges is playing which role - there are~$4!$ choices for this assignment. Then we need to choose which endpoints are playing which role for each edge. There are~$2^{4}$ choices for this assignment.}
We deduce that $\Delta_{s+1}\leq 24n^{4}(s+1)$.

Suppose that $s\leq k^{4}/2^{22}n^{2}$.
To bound~$\delta_{s}$ from below, we fix $G\in T_{s}^{D^{*}}$ and find a lower bound for the number of twist systems $H\subseteq G$ for which we record an edge between~$G$ and~$\text{twist}_{H}(G)$ in~$B_{s}$.
To do this, we will show that there are many choices for a set of four colours and two edges, such that each of these sets uniquely identifies a twist system in~$G$ for which we record an edge in~$B_{s}$.
Note that since $s\leq k^{4}/2^{22}n^{2}$ and $G\in Q_{s}^{D^{*}}$, we have\COMMENT{$e_{G}(A,B)<8|D|^{3}/n+6s\leq 8k^{3}/n+6k^{4}/2^{22}n^{2}\leq 10k^{3}/n$.}
\begin{equation}\label{eq:quas}
e_{G}(A,B)\leq 10k^{3}/n\hspace{5mm}\text{for}\hspace{1mm}\text{all}\hspace{1mm}\text{sets}\hspace{2mm}A,B\subseteq V\hspace{2mm}\text{of}\hspace{1mm}\text{sizes}\hspace{2mm}|A|=|B|=k.
\end{equation}
We begin by finding subsets of~$D_{3}$ and~$D_{4}$ with some useful properties in~$G$.
\begin{claim}\label{d3good}
There is a set $D_{3}^{\text{good}}\subseteq D_{3}$ of size $|D_{3}^{\text{good}}|\geq k/8$ such that for all $d\in D_{3}^{\text{good}}$ we have
\begin{enumerate}[label=\upshape(\roman*)]
\item $|E_{d}(N_{D_{1}}(x), N_{D_{2}}(x))|\leq 200k^{2}/n$;
\item there are at most~$64k^{3}/n^{2}$ $d$-edges~$e$ in~$G$ with the property that~$e$ lies in some distinguishable $(x,c,\cP)$-gadget in~$G$ whose $c$-edge is not supersaturated.
\end{enumerate}
\end{claim}
\claimproof{}
Observe that $|N_{D_{1}}(x)|=|N_{D_{2}}(x)|=k/4$.
Then, by (arbitrarily extending $N_{D_{1}}(x)$, $N_{D_{2}}(x)$ and) applying~(\ref{eq:quas}), we see that $e(N_{D_{1}}(x),N_{D_{2}}(x))\leq 10k^{3}/n$.\COMMENT{Suppose now that for at least~$k/16$ colours $d\in D_{3}$, there are at least~$200k^{2}/n$ edges of~$G$ in the set~$E_{d}(N_{D_{1}}(x),N_{D_{2}}(x))$.
Then $e(N_{D_{1}}(x),N_{D_{2}}(x))\geq 200k^{3}/16n$, a contradiction.}
Thus there is a set $\hat{D}_{3}\subseteq D_{3}$ of size $|\hat{D}_{3}|\geq 3k/16$ such that each $d\in\hat{D}_{3}$ satisfies~(i).
Next, notice that, since $r(G)=s$, there are at most~$s/k\leq k^{3}/2^{22}n^{2}$ saturated $c$-edges in~$G$\COMMENT{$r(G)\coloneqq k|\text{Sat}(G)|+\sum_{e\in\text{Unsat}(G)}\text{sat}(e)=s$ implies that $k|\text{Sat}(G)|\leq s$.}.
Suppose for a contradiction that at least~$k/16$ colours $d\in D_{3}$ are such that there are at least~$64k^{3}/n^{2}$ $d$-edges~$e$ in~$G$ with the property that~$e$ lies in some distinguishable $(x,c,\cP)$-gadget in~$G$ whose $c$-edge is not supersaturated.
Then, by considering the contribution of these distinguishable $(x,c,\cP)$-gadgets to~$r(G)$, and accounting for saturated $c$-edges, we obtain that\COMMENT{The total number of these distinguishable gadgets is at least $(k/16) \cdot 32k^{3}/n^{2}$, where we have divided $64k^{3}/n^{2}$ by~$2$ because each distinguishable gadget can contain up to~$2$ of these $d$-edges. Then it may be that some of these gadgets do not `contribute' to~$r(G)$ because their~$c$-edge is saturated. Since this occurs only if the saturation of such a $c$-edge is in $\{k+1, k+2, k+3, k+4, k+5\}$ and there are at most $k^{3}/2^{22}n^{2}$ saturated $c$-edges, subtracting $5k^{3}/2^{22}n^{2}$ accounts for this.} $r(G)\geq (k/16)\cdot 32k^{3}/n^{2} - 5k^{3}/2^{22}n^{2} >s$, a contradiction.
Thus there is a set $\Tilde{D}_{3}\subseteq D_{3}$ of size $|\Tilde{D}_{3}|\geq 3k/16$ such that each $d\in\Tilde{D}_{3}$ satisfies~(ii).
We define $D_{3}^{\text{good}}\coloneqq \hat{D}_{3}\cap\Tilde{D}_{3}$, and note that $|D_{3}^{\text{good}}|\geq k/8$.
\endclaimproof{}
We also define $D_{4}^{\text{good}}\subseteq D_{4}$ to be the set of colours $d_{4}\in D_{4}$ such that the $c$-edge~$e$ incident to the $d_{4}$-neighbour of~$x$ in~$G$ satisfies $\text{sat}(e)\leq k-1$.
Observe that $|D_{4}^{\text{good}}|\geq k/8$, since otherwise there are at least~$k/16$ saturated $c$-edges\COMMENT{There are at least $k/8$ colours $d_{4}\in D_{4}$ which are not good, and any fixed $c$-edge may be incident to the $d_{4}$-neighbour of~$x$ for up to two of these bad $d_{4}$.} in~$G$, whence $r(G)\geq k^{2}/16>s$, a contradiction.

We now show that there are many choices of a vector~$(d_{1},d_{2},d_{3},d_{4},\overrightarrow{f_{1}},\overrightarrow{f_{2}})$ where each $d_{i}\in D_{i}$ and each $\overrightarrow{f_{j}}$ is an edge $f_{j}\in E_{d_{3}}(G)$ together with an identification of which endpoints will play which role, such that each vector uniquely gives rise to a candidate of a twist system $H\subseteq G$.
We can begin to construct such a candidate by choosing $d_{4}\in D_{4}^{\text{good}}$ and letting~$u_{7}$ denote the $d_{4}$-neighbour of~$x$ in~$G$, and letting~$u_{8}$ denote the $c$-neighbour of~$u_{7}$.
Secondly, we choose $d_{1}\in D_{1}$, avoiding the colour of the edge~$xu_{8}$ (if it is present), and let~$u_{5}$ denote the $d_{1}$-neighbour of~$u_{7}$, and let~$u_{9}$ denote the $d_{1}$-neighbour of~$x$.
Next, we choose $d_{2}\in D_{2}$, avoiding the colours of the edges $u_{5}u_{8}$, $u_{5}x$, $u_{8}x$, $u_{8}u_{9}$ in~$G$ (if they are present), and let~$u_{6}$ denote the $d_{2}$-neighbour of~$u_{8}$, and let~$u_{10}$ denote the $d_{2}$-neighbour of~$x$.
Then, we choose $d_{3}\in D_{3}^{\text{good}}$, avoiding the colours of all edges in~$E_{G}(\{x,u_{5},u_{6},\dots,u_{10}\})$.
We let~$u_{3},u_{4},u_{11},u_{12}$ denote the $d_{3}$-neighbours of~$u_{5},u_{6},u_{9},u_{10}$, respectively.
Finally, we choose two distinct edges $f_{1},f_{2}\in E_{d_{3}}(G)$ which are not incident to any vertex in $\{x, u_{3}, u_{4},\dots,u_{12}\}$, and we choose which endpoint of~$f_{1}$ will play the role of~$u_{1}$ and which will play the role of~$u_{2}$, and choose which endpoint of~$f_{2}$ will play the role of~$u_{13}$ and which will play the role of~$u_{14}$.
Let~$\Lambda$ denote the set of all possible vectors~$(d_{1},d_{2},d_{3},d_{4},\overrightarrow{f_{1}},\overrightarrow{f_{2}})$ that can be chosen in this way, so that~$|\Lambda|\geq\frac{k}{8}\cdot\frac{3k}{16}\cdot\frac{k}{8}\cdot\frac{k}{16}\cdot\frac{n}{4}\cdot2\cdot\frac{n}{4}\cdot2=3k^{4}n^{2}/2^{16}$.
Further, let~$H(\lambda)\subseteq G$ denote the labelled subgraph of~$G$ corresponding to~$\lambda\in\Lambda$ in the above way.
If~$H(\lambda)$ is a twist system, then we sometimes say that we `twist on~$\lambda$' to mean that we perform the twist operation to obtain~$\text{twist}_{H(\lambda)}(G)$ from~$G$.

It is clear that~$H(\lambda)$ is unique for all vectors $\lambda\in\Lambda$, and that~$H(\lambda)$ satisfies conditions (i)--(vi) of the definition of a twist system.
However, some~$H(\lambda)$ may fail to satisfy~(vii), and some may fail to satisfy condition (b) in the definition of adjacency in~$B_{s}$.
We now show that only for a small proportion of~$\lambda\in\Lambda$ do either of these problems occur.
We begin by ensuring that most $\lambda\in\Lambda$ give rise to twist systems.
\begin{claim}
There is a subset $\Lambda_{1}\subseteq\Lambda$ such that $|\Lambda_{1}|\geq9|\Lambda|/10$ and~$H(\lambda)$ is a twist system for all $\lambda\in\Lambda_{1}$.
\end{claim}
\claimproof{}
Fix any choice of $d_{3}\in D_{3}^{\text{good}}$, $d_{4}\in D_{4}^{\text{good}}$ and $\overrightarrow{f_{1}}, \overrightarrow{f_{2}}$ appearing concurrently\COMMENT{So that we can assume the subgraph these colours determine is not `degenerate'} in some $\lambda\in\Lambda$, and note that there are at most~$(k/4)^{2}\cdot n^{2}$ such choices.
Here and throughout the remainder of the proof of Lemma~\ref{masterswitch}, we write~$u_{7}$ for the $d_{4}$-neighbour of~$x$, we write~$u_{8}$ for the $c$-neighbour of~$u_{7}$, and so on, where the choice of $d_{1}$, $d_{2}$, $d_{3}$, $d_{4}$, $\overrightarrow{f_{1}}$, $\overrightarrow{f_{2}}$ will always be clear from context.
Note that fixing $d_{3}$, $d_{4}$ only fixes the vertices $x$, $u_{7}$, $u_{8}$.
There are at most $10k^{3}/n$ pairs $(d_{1},d_{2})$ with each $d_{i}\in D_{i}$ such that there is an edge $u_{5}u_{6}\in E(G)$, since otherwise $e(N_{D_{1}}(u_{7}), N_{D_{2}}(u_{8}))>10k^{3}/n$, contradicting~(\ref{eq:quas})\COMMENT{for any extension of $N_{D_{1}}(u_{7})$, $N_{D_{2}}(u_{8})$ to sets of size~$k$. This technicality has already been mentioned once so now I reduce it to a comment. And from now on I even omit it from comments.}.
Similarly, there are at most $10k^{3}/n$ pairs $(d_{1}, d_{2})$ with each $d_{i}\in D_{i}$ such that $u_{9}u_{10}$ is an edge of~$G$\COMMENT{Otherwise, $e(N_{D_{1}}(x), N_{D_{2}}(x))>10k^{3}/n$, contradicting~(\ref{eq:quas}).}.
We deduce that there are at most $(20k^{3}/n)\cdot(k/4)^{2}\cdot n^{2}=5k^{5}n/4$ vectors $\lambda\in\Lambda$ for which~$H(\lambda)$ is such that either $u_{5}u_{6}$ or $u_{9}u_{10}$ is an edge of~$G$.
Now fix instead $d_{1}$, $d_{2}$, $d_{3}$, $d_{4}$, $\overrightarrow{f_{2}}$.
Note that\COMMENT{$|D^{*}|=k+1$.} $|N_{G}(u_{3})\cup N_{G}(u_{4})|\leq 2k+2$ so that there are at most $4k+4$ choices of $\overrightarrow{f_{1}}$ such that either $u_{1}u_{3}$ or $u_{2}u_{4}$ is an edge of~$G$.
Analysing the pairs $u_{11}u_{13}$ and $u_{12}u_{14}$ similarly\COMMENT{Fix $d_{1}$, $d_{2}$, $d_{3}$, $d_{4}$, $\overrightarrow{f_{1}}$. Note that $|N(u_{11})\cup N(u_{12})|\leq 2k+2$ so that there are at most $4k+4$ choices of $\overrightarrow{f_{2}}$ such that either $u_{11}u_{13}$ or $u_{12}u_{14}$ is an edge of~$G$.}, we deduce that altogether, there are at most $5k^{5}n/4+2((k/4)^{4}\cdot n\cdot(4k+4)) \leq 2k^{5}n\leq|\Lambda|/10$ vectors $\lambda\in\Lambda$ for which~$H(\lambda)$ fails to be a twist system.
\endclaimproof{}

We now show that only for a small proportion of $\lambda\in\Lambda_{1}$ does~$H(\lambda)$ fail to give rise to an edge in~$B_{s}$, by showing that most~$H(\lambda)$ satisfy the following properties:
\begin{itemize}
    \item [(P1)]$\text{twist}_{H(\lambda)}(G)\in Q_{s+1}^{D^{*}}$;  \item [(P2)] Deletion of the six $d_{3}$-edges in~$H(\lambda)$ does not decrease~$r(G)$;
    \item [(P3)] The canonical $(x,c,\cP)$-gadget of the twist~$\text{twist}_{H(\lambda)}(G)$ is distinguishable, and it is the only $(x,c,\cP)$-gadget which is in~$\text{twist}_{H(\lambda)}(G)$ but not in~$G$.\COMMENT{Recall that $\text{sat}_{G}(e)\leq k-1$ for the $c$-edge~$e$ of any~$H(\lambda)$ (since we choose $d_{4}\in D_{4}^{\text{good}}$) so this ensures~$r(G)$ increases, and by not more than one.}
\end{itemize}
Firstly, since~$G\in Q_{s}^{D^{*}}$ and we only create six new edges in any twist, it is clear that~$H(\lambda)$ satisfies~(P1) for all $\lambda\in\Lambda_{1}$.
\begin{claim}\label{rgnodec}
There is a subset $\Lambda_{2}\subseteq \Lambda_{1}$ such that $|\Lambda_{2}|\geq9|\Lambda_{1}|/10$ and~$H(\lambda)$ satisfies property~(P2) for all $\lambda\in\Lambda_{2}$.
\end{claim}
\claimproof{}
Fix $d_{1}\in D_{1}$, $d_{3}\in D_{3}^{\text{good}}$, $d_{4}\in D_{4}^{\text{good}}$, $\overrightarrow{f_{1}}$, $\overrightarrow{f_{2}}$ appearing concurrently in some $\lambda\in\Lambda_{1}$.
Let~$F_{d_{3}}(G)\subseteq E_{d_{3}}(G)$ be the set of $d_{3}$-edges~$e$ in~$G$ with the property that~$e$ is in some distinguishable $(x,c,\cP)$-gadget in~$G$ whose $c$-edge is not supersaturated.
Recall that $|F_{d_{3}}(G)|\leq 64k^{3}/n^{2}$ since $d_{3}\in D_{3}^{\text{good}}$.
Observe then that there are at most $128k^{3}/n^{2}$ colours $d_{2}\in D_{2}$ such that~$u_{10}$ is the endpoint of an edge in~$F_{d_{3}}(G)$.
Thus for all but at most $(k/4)^{3}\cdot n^{2}\cdot 128k^{3}/n^{2}=2k^{6}$ choices of $\lambda=(d_{1},d_{2},d_{3},d_{4}, \overrightarrow{f_{1}}, \overrightarrow{f_{2}})\in\Lambda_{1}$, the edge~$u_{10}u_{12}$ is not in~$F_{d_{3}}(G)$.
Now fix instead $d_{1}$, $d_{2}$, $d_{3}$, $d_{4}$, $\overrightarrow{f_{2}}$.
Then since $d_{3}\in D_{3}^{\text{good}}$, there are at most $128k^{3}/n^{2}$ choices of~$\overrightarrow{f_{1}}$ such that~$f_{1}\in F_{d_{3}}(G)$, so that for all but at most $(k/4)^{4}\cdot n\cdot 128k^{3}/n^{2}=k^{7}/2n$ vectors $\lambda \in\Lambda_{1}$,~$H(\lambda)$ is such that $f_{1}\notin F_{d_{3}}(G)$.
Similar analyses\COMMENT{Fix $d_{2}\in D_{2}$, $d_{3}\in D_{3}^{\text{good}}$, $d_{4}\in D_{4}^{\text{good}}$, $\overrightarrow{f_{1}}$, $\overrightarrow{f_{2}}$. At most $128k^{3}/n^{2}$ colours $d_{1}\in D_{1}$ are such that~$u_{9}$ is the endpoint of an edge in~$F_{d_{3}}(G)$. Thus for all but at most $(k/4)^{3}\cdot n^{2}\cdot 128k^{3}/n^{2}=2k^{6}$ choices of $\lambda\in\Lambda$, the edge~$u_{9}u_{11}$ is not in~$F_{d_{3}}(G)$. \newline Fix $d_{2}\in D_{2}$, $d_{3}\in D_{3}^{\text{good}}$, $d_{4}\in D_{4}^{\text{good}}$, $\overrightarrow{f_{1}}$, $\overrightarrow{f_{2}}$. In particular, we have fixed which vertex of~$V$ plays the role of~$u_{7}$. At most $128k^{3}/n^{2}$ colours $d_{1}\in D_{1}$ are such that~$u_{5}$ is the endpoint of an edge in~$F_{d_{3}}(G)$. Thus for all but at most $(k/4)^{3}\cdot n^{2}\cdot 128k^{3}/n^{2}=2k^{6}$ choices of $\lambda\in\Lambda$, the edge~$u_{3}u_{5}$ is not in~$F_{d_{3}}(G)$. \newline Fix $d_{1}\in D_{2}$, $d_{3}\in D_{3}^{\text{good}}$, $d_{4}\in D_{4}^{\text{good}}$, $\overrightarrow{f_{1}}$, $\overrightarrow{f_{2}}$. In particular, we have fixed which vertex of~$V$ plays the role of~$u_{8}$. At most $128k^{3}/n^{2}$ colours $d_{2}\in D_{2}$ are such that~$u_{6}$ is the endpoint of an edge in~$F_{d_{3}}(G)$. Thus for all but at most $(k/4)^{3}\cdot n^{2}\cdot 128k^{3}/n^{2}=2k^{6}$ choices of $\lambda\in\Lambda$, the edge~$u_{4}u_{6}$ is not in~$F_{d_{3}}(G)$. \newline Fix $d_{1}$, $d_{2}$, $d_{3}$, $d_{4}$, $\overrightarrow{f_{1}}$.
Then since $d_{3}\in D_{3}^{\text{good}}$, there are at most $128k^{3}/n^{2}$ choices of~$\overrightarrow{f_{2}}$ such that~$f_{2}\in F_{d_{3}}(G)$, so that for all but at most $(k/4)^{4}\cdot n\cdot 128k^{3}/n^{2}=k^{7}/2n$ vectors $\lambda \in\Lambda$, the corresponding labelled subgraph of~$G$ is such that $f_{2}\notin F_{d_{3}}(G)$.} show that there are at most $8k^{6}+k^{7}/n\leq 9k^{6}\leq|\Lambda_{1}|/10$ choices of $\lambda\in\Lambda_{1}$ such that $\{u_{1}u_{2},u_{3}u_{5},u_{4}u_{6},u_{9}u_{11},u_{10}u_{12},u_{13}u_{14}\}\cap F_{d_{3}}(G)\neq\emptyset$.
By definition of~$F_{d_{3}}(G)$ and supersaturation of a $c$-edge, we deduce that for all remaining $\lambda\in \Lambda_{1}$,~$H(\lambda)$ is such that deleting the edges $u_{1}u_{2}$, $u_{3}u_{5}$, $u_{4}u_{6}$, $u_{9}u_{11}$, $u_{10}u_{12}$, $u_{13}u_{14}$ does not decrease~$r(G)$.\COMMENT{For any of these six edges, say~$e$, either~$e$ is not in a distinguishable $(x,c,\cP)$-gadget, whence deleting~$e$ does not decrease~$r(G)$, or the unique (by distinguishability) distinguishable $(x,c,\cP)$-gadget that~$e$ is in is such that the $c$-edge of this gadget is supersaturated.
Then deletion of~$e$ decreases the saturation of this $c$-edge by precisely one, which does not decrease~$r(G)$. The worst case is that all six edges $u_{1}u_{2}$, $u_{3}u_{5}$, $u_{4}u_{6}$, $u_{9}u_{11}$, $u_{10}u_{12}$, $u_{13}u_{14}$ are each in separate distinguishable $(x,c,\cP)$-gadgets each using the same supersaturated $c$-edge of~$G$. But by the definition of supersaturation, deletion of these six edges still does not decrease~$r(G)$.}
\endclaimproof{}

When we perform a twist operation on a twist system~$H$ in~$G$, since the only new edges we add have some colour in~$D_{3}$, we have that for any new distinguishable $(x,c,\cP)$-gadget~$J$ we create in the twist, one of the new edges $u_{1}u_{3}$, $u_{2}u_{4}$, $u_{5}u_{6}$, $u_{9}u_{10}$, $u_{11}u_{13}$, $u_{12}u_{14}$ of the twist is playing the role of either $v_{5}v_{6}$ or $v_{9}v_{10}$ in~$J$. 
(Here and throughout the rest of the proof, we imagine completed $(x,c,\cP)$-gadgets~$J$ as having vertices labelled $x, v_{5}, \dots, v_{10}$, where the role of~$v_{i}$ corresponds to the role of~$u_{i}$ in Figure~\ref{fig:ts}.)
We now show that for most $\lambda\in\Lambda_{2}$,~$H(\lambda)$ satisfies property~(P3).
This is the most delicate part of the argument, and we break it into three more claims.

\begin{claim}\label{p2claim1}
There is a subset $\Lambda_{3}\subseteq\Lambda_{2}$ such that $|\Lambda_{3}|\geq9|\Lambda_{2}|/10$ and all $\lambda\in\Lambda_{3}$ are such that if~$J$ is an $(x,c,\cP)$-gadget that is in~$\text{twist}_{H(\lambda)}(G)$ but not in~$G$, then the pair~$u_{9}u_{10}$ of~$H(\lambda)$ plays the role of~$v_{9}v_{10}$.
\end{claim}
\claimproof{}
Since the only edges added by any twist operation all have colour in~$D_{3}$, it suffices to show that at most~$|\Lambda_{2}|/10$ vectors $\lambda\in\Lambda_{2}$ are such that twisting on~$\lambda$ creates an $(x,c,\cP)$-gadget~$J$ for which either
\begin{enumerate}[label=\upshape(\roman*)]
\item one of the pairs $u_{1}u_{3}$, $u_{2}u_{4}$, $u_{5}u_{6}$, $u_{11}u_{13}$, $u_{12}u_{14}$ of~$H(\lambda)$ plays the role of~$v_{9}v_{10}$, or
\item the edge~$v_{9}v_{10}$ of~$J$ is present in~$G$.
\end{enumerate}
To address~(i), we show that $u_{1},u_{2},u_{5},u_{11},u_{12}\notin N_{G}(x)$ for all but at most~$|\Lambda_{2}|/20$ vectors $\lambda\in\Lambda_{2}$.
Note firstly that at most $10k^{3}/n$ pairs $(d_{1},d_{4})$ where $d_{1}\in D_{1}$, $d_{4}\in D_{4}^{\text{good}}$ are such that $u_{5}\in N_{G}(x)$, since otherwise $e(N_{D_{4}}(x), N_{G}(x))>10k^{3}/n$, contradicting~(\ref{eq:quas}).
Thus, at most $(k/4)^{2}\cdot n^{2}\cdot 10k^{3}/n=5k^{5}n/8$ choices of $\lambda\in\Lambda_{2}$ are such that $u_{5}\in N_{G}(x)$.
Now fix $d_{1}$, $d_{2}$, $d_{3}$, $d_{4}$, $\overrightarrow{f_{2}}$ appearing concurrently in some $\lambda\in\Lambda_{2}$.
Notice that there are at most $2k+2$ choices of~$\overrightarrow{f_{1}}$ such that~$f_{1}$ has at least one endpoint in~$N_{G}(x)$.
Analysing~$\overrightarrow{f_{2}}$ similarly\COMMENT{Fix $d_{1}$, $d_{2}$, $d_{3}$, $d_{4}$, $\overrightarrow{f_{1}}$. Notice that there are at most $2k+2$ choices of~$\overrightarrow{f_{2}}$ such that~$f_{2}$ has either endpoint in~$N_{G}(x)$. Now $5k^{5}n/8 +2((k/4)^{4}\cdot n\cdot(2k+2)\leq k^{5}n$.}, we deduce that there are at most $5k^{5}n/8+2(k/4)^{4}(2k+2)n\leq|\Lambda_{2}|/20$ choices of $\lambda\in\Lambda_{2}$ such that at least one of $u_{1}$, $u_{2}$, $u_{5}$, $u_{13}$, $u_{14}$ lies in~$N_{G}(x)$.

Turning now to~(ii), we show that at most~$|\Lambda_{2}|/20$ vectors~$\lambda\in\Lambda_{2}$ are such that twisting on~$\lambda$ creates any $(x,c,\cP)$-gadgets~$J$ for which the edge~$v_{9}v_{10}$ of~$J$ is present in~$G$ (and thus one of the pairs $u_{1}u_{3},u_{2}u_{4},u_{5}u_{6},u_{9}u_{10},u_{11}u_{13},u_{12}u_{14}$ of~$H(\lambda)$ plays the role of~$v_{5}v_{6}$).
To do this, we use some of the properties of~$D_{3}^{\text{good}}$.
Fix $d_{2}\in D_{2}$, $d_{3}\in D_{3}^{\text{good}}$, $d_{4}\in D_{4}^{\text{good}}$, $\overrightarrow{f_{1}}$, $\overrightarrow{f_{2}}$ appearing concurrently in some $\lambda \in \Lambda_{3}$.
Note that since $d_{3}\in D_{3}^{\text{good}}$, there are at most $200k^{2}/n$ pairs $(d_{1}', d_{2}')$ where $d_{1}'\in D_{1}$, $d_{2}'\in D_{2}$, such that there is a~$K_{3}$ in~$G$ with vertices $x$, $w_{1}$, $w_{2}$, where~$w_{i}$ is the $d_{i}'$-neighbour of~$x$ for $i\in\{1,2\}$, and the edge $w_{1}w_{2}$ is coloured~$d_{3}$.
Let the set of these pairs $(d_{1}',d_{2}')$ be denoted~$L(d_{3})$.
For each pair $\ell=(d_{1}',d_{2}')\in L(d_{3})$, let~$z_{\ell}^{1}$ be the end of the $d_{1}'cd_{2}'$-walk starting at~$u_{10}$.
Similarly, let~$z_{\ell}^{2}$ denote the end of the $d_{2}'cd_{1}'$-walk starting at~$u_{10}$.
Define~$M\coloneqq\bigcup_{\ell\in L(d_{3})}\{z_{\ell}^{1}, z_{\ell}^{2}\}$, so that $|M|\leq 400k^{2}/n$.
Since there are at most~$400k^{2}/n$ choices of $d_{1}\in D_{1}$ for which we obtain $u_{9}\in M$, we deduce that for all but at most $(k/4)^{3}\cdot n^{2}\cdot 400k^{2}/n=25k^{5}n/4$ vectors $\lambda\in\Lambda_{2}$,~$H(\lambda)$ is such that adding the edge $u_{9}u_{10}$ in colour~$d_{3}$ does not create a new $(x,c,\cP)$-gadget~$J$ where $u_{9}u_{10}$ plays the role of $v_{5}v_{6}$ in~$J$ and the edge playing the role of $v_{9}v_{10}$ in~$J$ is already present in~$G$ before the twist\COMMENT{Only when we choose $d_{1}\in D_{1}$ such that $u_{9}\in M$ is it true that adding the edge $u_{9}u_{10}$ with colour $d_{3}$ completes a $C_{4}$ with colours $d_{1}', c, d_{2}', d_{3}$ for which there is already a $K_{3}$ with colours $d_{1}', d_{2}', d_{3}$ using the $d_{1}'$ and $d_{2}'$ edges at $x$. In fact only those structures amongst these for which there is a $d_{4}'$-edge (for some $d_{4}'\in D_{4}$) in the necessary place actually create one of the new gadgets we are trying to rule out, but this only helps us.}.
One can observe similarly\COMMENT{Now, fix $d_{2}$, $d_{3}$, $d_{4}$, $\overrightarrow{f_{1}}$, $\overrightarrow{f_{2}}$.
Since $d_{3}\in D_{3}^{\text{good}}$, observe that $|L(d_{3})|\leq 200k^{2}/n$.
Now for each $\ell=(d_{1}', d_{2}')\in L(d_{3})$, let $z_{\ell}^{1}$ be the end of the $d_{1}'cd_{2}'$-walk starting at~$u_{6}$, and let $z_{\ell}^{2}$ be the end of the $d_{2}'cd_{1}'$-walk starting at~$u_{6}$.
Define~$M\coloneqq\bigcup_{\ell\in L(d_{3})}\{z_{\ell}^{1}, z_{\ell}^{2}\}$, and notice that $|M|\leq 400k^{2}/n$.
Since there are at most~$400k^{2}/n$ choices of $d_{1}\in D_{1}$ for which we obtain $u_{5}\in M$, we deduce that for all but at most $(k/4)^{3}\cdot n^{2}\cdot 400k^{2}/n=25k^{5}n/4$ vectors $\lambda\in\Lambda_{2}$,~$H(\lambda)$ is such that adding the edge $u_{5}u_{6}$ in colour~$d_{3}$ does not create a new $(x,c,\cP)$-gadget~$J$ where $u_{5}u_{6}$ plays the role of $v_{5}v_{6}$ in~$J$ and the edge playing the role of $v_{9}v_{10}$ in~$J$ is already present in~$G$ before the twist.} that for all but at most~$25k^{5}n/4$ vectors~$\lambda\in\Lambda$,~$H(\lambda)$ is such that adding the edge $u_{5}u_{6}$ in colour~$d_{3}$ does not create a new $(x,c,\cP)$-gadget~$J$ where $u_{5}u_{6}$ plays the role of $v_{5}v_{6}$ in~$J$ and the edge playing the role of $v_{9}v_{10}$ in~$J$ is already present in~$G$ before the twist.

Now fix instead $d_{1}$, $d_{2}$, $d_{3}$, $d_{4}$, $\overrightarrow{f_{2}}$ appearing concurrently in some $\lambda\in\Lambda_{2}$.
For each $\ell=(d_{1}', d_{2}')\in L(d_{3})$, let $y_{\ell}^{1}$ be the end of the $d_{1}'cd_{2}'$-walk starting at~$u_{4}$, let $y_{\ell}^{2}$ be the end of the $d_{2}'cd_{1}'$-walk starting at~$u_{4}$, let $z_{\ell}^{1}$ be the end of the $d_{1}'cd_{2}'$-walk starting at~$u_{3}$, and let $z_{\ell}^{2}$ be the end of the $d_{2}'cd_{1}'$-walk starting at~$u_{3}$.
Define $M\coloneqq\bigcup_{\ell\in L(d_{3})}\{y_{\ell}^{1}, y_{\ell}^{2}, z_{\ell}^{1}, z_{\ell}^{2}\}$, and notice that $|M|\leq 800k^{2}/n$.
We deduce that there are at most $1600k^{2}/n$ choices of~$\overrightarrow{f_{1}}$ such that~$f_{1}$ has an endpoint in~$M$, and that for all remaining choices of~$\overrightarrow{f_{1}}$, twisting on $\lambda=(d_{1},d_{2},d_{3},d_{4},\overrightarrow{f_{1}},\overrightarrow{f_{2}})$ cannot create a new $(x,c,\cP)$-gadget~$J$ where the new $d_{3}$-edges $u_{1}u_{3}$ or $u_{2}u_{4}$ play the role of $v_{5}v_{6}$ in~$J$ and the edge~$v_{9}v_{10}$ of~$J$ is present in~$G$. 
Analysing~$\overrightarrow{f_{2}}$ similarly\COMMENT{Fix $d_{1}$, $d_{2}$, $d_{3}$, $d_{4}$, $\overrightarrow{f_{1}}$. Since $d_{3}\in D_{3}^{\text{good}}$, observe that $|L(d_{3})|\leq 200k^{2}/n$. Now for each $\ell=(d_{1}', d_{2}')\in L(d_{3})$, let $y_{\ell}^{1}$ be the end of the $d_{1}'cd_{2}'$-walk starting at~$u_{11}$, let $y_{\ell}^{2}$ be the end of the $d_{2}'cd_{1}'$-walk starting at~$u_{11}$, let $z_{\ell}^{1}$ be the end of the $d_{1}'cd_{2}'$-walk starting at~$u_{12}$, and let $z_{\ell}^{2}$ be the end of the $d_{2}'cd_{1}'$-walk starting at~$u_{12}$. Define $M\coloneqq\bigcup_{\ell\in L(d_{3})}\{y_{\ell}^{1}, y_{\ell}^{2}, z_{\ell}^{1}, z_{\ell}^{2}\}$, and notice that $|M|\leq 800k^{2}/n$. We deduce that there are at most $1600k^{2}/n$ choices of~$\overrightarrow{f_{2}}$ such that~$f_{2}$ has an endpoint in~$M$, and that for all remaining choices of~$\overrightarrow{f_{2}}$, twisting on $\lambda=(d_{1},d_{2},d_{3},d_{4},\overrightarrow{f_{1}},\overrightarrow{f_{2}})$ cannot create a new $(x,c,\cP)$-gadget~$J$ where the new $d_{3}$-edges $u_{11}u_{13}$ or $u_{12}u_{14}$ play the role of $v_{5}v_{6}$ in~$J$ and the edge~$v_{9}v_{10}$ of~$J$ is present in~$G$.}, we conclude that for all but at most\COMMENT{$25k^{5}n/4$ for $u_{9}u_{10}$, $25k^{5}n/4$ for $u_{5}u_{6}$, and $2((k/4)^{4}\cdot n\cdot 2000k^{2}/n)$ for $u_{1}u_{3}$, $u_{2}u_{4}$, $u_{11}u_{13}$, $u_{12}u_{14}$.} $13k^{5}n\leq|\Lambda_{2}|/20$ choices of $\lambda\in\Lambda_{2}$, twisting on~$\lambda$ cannot create a new $(x,c,\cP)$-gadget~$J$ for which the edge~$v_{9}v_{10}$ of~$J$ is present in~$G$.
\endclaimproof{}
\begin{claim}
There is a subset $\Lambda_{4}\subseteq\Lambda_{3}$ such that $|\Lambda_{4}|\geq9|\Lambda_{3}|/10$ and all $\lambda\in\Lambda_{4}$ are such that if~$J$ is an $(x,c,\cP)$-gadget that is in~$\text{twist}_{H(\lambda)}(G)$ but not in~$G$, then the pair~$u_{5}u_{6}$ of~$H(\lambda)$ plays the role of~$v_{5}v_{6}$.
\end{claim}
\claimproof{}
By Claim~\ref{p2claim1}, it will suffice to show that at most~$|\Lambda_{3}|/10$ vectors $\lambda\in\Lambda_{3}$ are such that twisting on~$\lambda$ creates an $(x,c,\cP)$-gadget~$J$ for which either
\begin{enumerate}[label=\upshape(\roman*)]
\item one of the pairs~$u_{1}u_{3},u_{2}u_{4},u_{11}u_{13},u_{12}u_{14}$ of~$H(\lambda)$ plays the role of~$v_{5}v_{6}$, and~$u_{9}u_{10}$ plays the role of~$v_{9}v_{10}$, or
\item the edge~$v_{5}v_{6}$ of~$J$ is present in~$G$ and the pair~$u_{9}u_{10}$ of~$H(\lambda)$ plays the role of~$v_{9}v_{10}$.
\end{enumerate}
To address~(i), fix $d_{1}$, $d_{2}$, $d_{3}$, $d_{4}$, $\overrightarrow{f_{2}}$ appearing concurrently in some $\lambda\in\Lambda_{3}$.
Let~$a_{1}$ be the end of the $d_{1}cd_{2}$-walk starting at~$u_{4}$, let~$a_{2}$ be the end of the $d_{2}cd_{1}$-walk starting at~$u_{4}$, let~$b_{1}$ be the end of the $d_{1}cd_{2}$-walk starting at~$u_{3}$, and let~$b_{2}$ be the end of the $d_{2}cd_{1}$-walk starting at~$u_{3}$.
Since there are at most~$8$ choices of~$\overrightarrow{f_{1}}$ with an endpoint in~$\{a_{1}, a_{2}, b_{1}, b_{2}\}$, we deduce that for all remaining choices of~$\overrightarrow{f_{1}}$, twisting on $\lambda=(d_{1},d_{2},d_{3},d_{4},\overrightarrow{f_{1}},\overrightarrow{f_{2}})$ cannot create an $(x,c,\cP)$-gadget~$J$ for which the new $d_{3}$-edges~$u_{1}u_{3}$ or~$u_{2}u_{4}$ play the role of~$v_{5}v_{6}$ in~$J$ and~$u_{9}u_{10}$ plays the role of~$v_{9}v_{10}$.
Analysing~$\overrightarrow{f_{2}}$ similarly, we conclude that we must discard at most~$k^{4}n/16\leq|\Lambda_{3}|/20$ vectors $\lambda\in\Lambda_{3}$ to account for~(i).

Turning now to~(ii), write $D_{4}=\{d_{4}^{1}, d_{4}^{2},\dots, d_{4}^{k/4}\}$.
For each $d_{4}^{i}\in D_{4}$, let~$y_{i}$ be the $d_{4}^{i}$-neighbour of~$x$, let~$z_{i}$ be the $c$-neighbour of~$y_{i}$, define $R_{i}\coloneqq N_{D_{1}}(y_{i})$ and $S_{i}\coloneqq N_{D_{2}}(z_{i})$.
Notice that $\sum_{i=1}^{k/4}e(R_{i}, S_{i})\leq 5k^{4}/2n$, since otherwise we obtain a contradiction to~(\ref{eq:quas}) for some pair $(R_{i}, S_{i})$.
We deduce that there are at most $5k^{4}/2n$ triples $(d_{1},d_{2},d_{3})$ with each $d_{i}\in D_{i}$ for which adding the edge $u_{9}u_{10}$ in colour~$d_{3}$ creates an $(x,c,\cP)$-gadget~$J$ for which $u_{9}u_{10}$ plays the role of $v_{9}v_{10}$ in~$J$ and the edge playing the role of $v_{5}v_{6}$ is already present in~$G$, whence at most $(5k^{4}/2n)\cdot (k/4)\cdot n^{2}=5k^{5}n/8\leq|\Lambda_{3}|/20$ choices of $\lambda\in\Lambda_{3}$ are such that twisting on~$\lambda$ creates an $(x,c,\cP)$-gadget of this type.
\endclaimproof{}
\begin{claim}\label{p2claim3}
There is a subset $\Lambda_{5}\subseteq\Lambda_{4}$ such that $|\Lambda_{5}|\geq9|\Lambda_{4}|/10$ and all $\lambda\in\Lambda_{5}$ are such that if~$J$ is an $(x,c,\cP)$-gadget that is in~$\text{twist}_{H(\lambda)}(G)$ but not in~$G$ and the pairs~$u_{5}u_{6},u_{9}u_{10}$ of~$H(\lambda)$ play the roles of the edges~$v_{5}v_{6},v_{9}v_{10}$ of~$J$ respectively, then~$J$ is the canonical $(x,c,\cP)$-gadget of the twist.
\end{claim}
\claimproof{}
Fix $d_{3}$, $d_{4}$, $\overrightarrow{f_{1}}$, $\overrightarrow{f_{2}}$ appearing concurrently in some $\lambda\in\Lambda_{4}$.
By~(\ref{eq:quas}), we have that $e(N_{D_{2}}(u_{8}), N_{D_{4}}(x))\leq 10k^{3}/n$.
We deduce that there are at most $10k^{3}/n$ choices of the pair $(d_{1},d_{2})$ such that the $d_{1}$-neighbour of~$u_{6}$ lies in~$N_{D_{4}}(x)$, whence\COMMENT{The situation we were trying to avoid was the one in which~$u_{5}$ is the endpoint of the $d_{1}cd_{2}$-walk starting at~$u_{6}$, such that the $d_{1}$-neighbour of~$u_{6}$ is also a $d_{4}'$-neighbour of~$x$ for some $d_{4}'\in D_{4}$. In this situation, when we twist, we create an unwanted second $(x,c,\cP)$-gadget where $u_{5}u_{6}$ and $u_{9}u_{10}$ play the roles of $v_{5}v_{6}$ and $v_{9}v_{10}$ respectively.} for all but at most $5k^{5}n/8\leq|\Lambda_{4}|/10$ choices of $\lambda\in\Lambda_{4}$, the canonical $(x,c,\cP)$-gadget of the twist is the only new $(x,c,\cP)$-gadget for which $u_{5}u_{6},u_{9}u_{10}$ play the roles of $v_{5}v_{6},v_{9}v_{10}$ respectively. 
\endclaimproof{}
Note that, by Claims~\ref{p2claim1}--\ref{p2claim3}, the canonical $(x,c,\cP)$-gadget of a twist on $\lambda\in\Lambda_{5}$ is clearly distinguishable in~$\text{twist}_{H(\lambda)}(G)$ since its edges~$v_{5}v_{6}$ and~$v_{9}v_{10}$ with colours in~$D_{3}$ were added by the twist and performing this twist creates no other $(x,c,\cP)$-gadgets.
Thus Claims~\ref{p2claim1}--\ref{p2claim3} imply that~$H(\lambda)$ satisfies~(P3) for all $\lambda\in\Lambda_{5}$.
Recalling that~$\text{sat}_{G}(e)\leq k-1$ for the $c$-edge~$e$ of~$H(\lambda)$ for all $\lambda\in\Lambda$ and also using Claim~\ref{rgnodec}, we now deduce that $r(\text{twist}_{H(\lambda)}(G))=r(G)+1$, and thus $\text{twist}_{H(\lambda)}(G)\in A_{s+1}^{D^{*}}$, for all $\lambda\in\Lambda_{5}$.
Since~$H(\lambda)$ satisfies~(P1) for all $\lambda\in\Lambda_{1}$, we deduce that $\text{twist}_{H(\lambda)}(G)\in T_{s+1}^{D^{*}}$ for all $\lambda\in\Lambda_{5}$, and that $\delta_{s}\geq|\Lambda_{5}|\geq|\Lambda|/2\geq 3k^{4}n^{2}/2^{17}$.
We conclude that
if $s\leq k^{4}/2^{22}n^{2}$ and~$T_{s}^{D^{*}}$ is non-empty, then~$T_{s+1}^{D^{*}}$ is non-empty and $|T_{s}^{D^{*}}|/|T_{s+1}^{D^{*}}|\leq 2^{17}\cdot 24n^{4}(s+1)/3k^{4}n^{2}\leq 1/2$.
Now, fix $s\leq \mu^{4}n^{2}/2^{23}$.
If~$T_{s}^{D^{*}}$ is empty, then\COMMENT{Note $\prob{r(\mathbf{G})=s\mid\mathbf{G}\in\widetilde{\cQ}_{D^{*}}^{\text{col}}}=|T_{s}^{D^{*}}|/|\widetilde{\cQ}_{D^{*}}^{\text{col}}|$, so we need to know $\widetilde{\cQ}_{D^{*}}^{\text{col}}$ is non-empty to know that we are not dividing by zero. But this follows from the fact that $\cQ_{D^{*}}^{\text{col}}$ is non-empty, which in turn follows from the usual existence results of $1$-factorizations of complete graphs, restricting such $1$-factorizations to see that $\cG_{D^{*}}^{\text{col}}$ is non-empty, and then applying Lemma~\ref{quasirandom}.} $\prob{r(\mathbf{G})=s\mid\mathbf{G}\in\widetilde{\cQ}_{D^{*}}^{\text{col}}}=0$.
If~$T_{s}^{D^{*}}$ is non-empty, then
\[
\prob{r(\mathbf{G})=s\mid\mathbf{G}\in\widetilde{\cQ}_{D^{*}}^{\text{col}}} = \frac{|T_{s}^{D^{*}}|}{|\widetilde{\cQ}_{D^{*}}^{\text{col}}|} \leq  \frac{|T_{s}^{D^{*}}|}{|T_{k^{4}/2^{22}n^{2}}^{D^{*}}|} =\prod_{j=s}^{k^{4}/2^{22}n^{2}-1}\frac{|T_{j}^{D^{*}}|}{|T_{j+1}^{D^{*}}|} \leq (1/2)^{k^{4}/2^{22}n^{2}-s},
\]
and thus\COMMENT{Note that the middle expression is bounded above by $\left(\frac{\mu^{4}n^{2}}{2^{23}}+1\right)\exp\left(-\frac{\mu^{4}n^{2}}{2^{23}}\ln 2\right)$.},
\[
\prob{r(\mathbf{G})\leq \mu^{4}n^{2}/2^{23}\mid\mathbf{G}\in\widetilde{\cQ}_{D^{*}}^{\text{col}}} \leq \sum_{s=0}^{\mu^{4}n^{2}/2^{23}}\exp(-(k^{4}/2^{22}n^{2}-s)\ln 2)\leq \exp\left(-\frac{\mu^{4}n^{2}}{2^{24}}\right),
\]
which completes the proof of the lemma.
\end{proof}
Next, we show that in order to find many well-spread $(x,c)$-absorbing gadgets in~$G\in\cG_{D\cup\{c\}}^{\text{col}}$, it suffices to show that~$r(G)$ is large for some equitable partition~$\cP$ of~$D$ into four parts.
(Recall that `well-spread' was defined in Definition~\ref{spread}.)
\begin{lemma}\label{justgadgets}
Suppose that $1/n\ll\mu$, and let $D\subseteq[n-1]$ be such that $|D|\leq\mu n$.
Let $x\in V$, let $c\in [n-1]\setminus D$, and let $\cP=\{D_{i}\}_{i=1}^{4}$ be an equitable partition of~$D$.
Then for any integer $t\geq 0$ and any $G\in\cG_{D\cup\{c\}}^{\text{col}}$, if $r(G)\geq t$, then~$G$ contains a $5\mu n/4$-well-spread collection of~$t$ distinct $(x,c)$-absorbing gadgets.
\end{lemma}
\begin{proof}
Let $G\in\cG_{D\cup\{c\}}^{\text{col}}$, let $t\geq 0$ be an integer, and suppose that $r(G)\geq t$.
Then, since $|D|\leq\mu n$ and by definition of~$r$, we deduce that there is a collection~$\cA_{(x,c,\cP)}$ of~$t$ distinct $(x,c,\cP)$-gadgets satisfying the following conditions:
\begin{enumerate}[label=\upshape(\roman*)]
    \item Each edge of~$G$ with colour in~$D_{3}$ is contained in at most one $(x,c,\cP)$-gadget $J\in\cA_{(x,c,\cP)}$;
    \item Each $c$-edge of~$G$ is contained in at most $\mu n$ $(x,c,\cP)$-gadgets $J\in\cA_{(x,c,\cP)}$.
\end{enumerate}
Fix $v\in V\setminus\{x\}$.
Let~$e$ be the $c$-edge of~$G$ incident to~$v$ and for each $d\in D_{3}$ let~$f_{d}$ be the $d$-edge of~$G$ incident to~$v$.
Then by conditions~(i) and~(ii) there are at most $5\mu n/4$ $(x,c,\cP)$-gadgets $J\in\cA_{(x,c,\cP)}$ containing any of the edges in $\{e\}\cup\bigcup_{d\in D_{3}}\{f_{d}\}$.
Note that if~$v$ is contained in some $J\in\cA_{(x,c,\cP)}$, then~$v$ is incident to either the $c$-edge in~$J$, or to one of the edges in~$J$ with colour in~$D_{3}$.
We thus conclude\COMMENT{Of course, it may be that there are gadgets $J\in\cA_{(x,c,\cP)}$ such that~$J$ contains one of the edges coloured with some colour in~$D_{1}$ (say) incident to~$v$. But by the above observation,~$J$ must use one of the edges $\{e\}\cup\bigcup_{d\in D_{3}}\{f_{d}\}$, so we have already counted~$J$.} that~$v$ is contained in at most $5\mu n/4$ $(x,c,\cP)$-gadgets $J\in\cA_{(x,c,\cP)}$.
It immediately follows that no edge of~$G$ is contained in more than~$5\mu n/4$ $(x,c,\cP)$-gadgets $J\in\cA_{(x,c,\cP)}$.\COMMENT{Indeed the number of times an edge is used can be bounded above by the number of times one of its endpoints is used.}

For each $d\in D_{1}\cup D_{2}\cup D_{4}$, there are at most~$5\mu n/4$ $J\in\cA_{(x,c,\cP)}$ with $d\in\phi(J)$ since each such~$J$ must contain the $d$-neighbour of~$x$ in~$G$.
For each $d\in D_{3}$, there are at most~$\mu n/2$ $d$-edges~$f$ in~$G$ such that both endpoints of~$f$ are neighbours of~$x$.
Any $J\in\cA_{(x,c,\cP)}$ for which $d\in\phi(J)$ must contain one of these edges~$f$.
Thus by~(i), there are at most $\mu n/2$ $J\in\cA_{(x,c,\cP)}$ such that $d\in\phi(J)$.

Finally, define a function~$g$ on~$\cA_{(x,c,\cP)}$ by setting $g(J)\coloneqq J-f$, where~$f$ is the unique edge of~$J$ with colour in~$D_{4}$, for each $J\in\cA_{(x,c,\cP)}$.
Then it is clear that~$g$ is injective\COMMENT{Suppose $g(J)=g(J')$. Then~$J$ and $J'$ can both be obtained by adding back in the edge of~$G$ with endpoints~$x$ and~$u$, where $u$ is the unique vertex of~$g(J)=g(J')$ having $\phi_{G}(\partial_{J}(u))=\{c, d_{1}\}$, for some $d_{1}\in D_{1}$. So $J=J'$.} and that~$g(J)$ is an $(x,c)$-absorbing gadget, for each $J\in\cA_{(x,c,\cP)}$.
Thus,~$g(\cA_{(x,c,\cP)})$ is a $5\mu n/4$-well-spread collection of~$t$ distinct $(x,c)$-absorbing gadgets in~$G$, as required.
\end{proof}
\subsection{Weighting factor}
We now state two results on the number of $1$-factorizations in dense $d$-regular graphs~$G$, where a $1$-factorization of~$G$ consists of an ordered set of~$d$ perfect matchings in~$G$.
We will use these results to find a `weighting factor' (see Corollary~\ref{wf}), which we will use to compare the probabilities of particular events occurring in different probability spaces.
For any graph~$G$, let~$M(G)$ denote the number of distinct $1$-factorizations of~$G$, and for any $n,d\in\bN$, let $\cG_{d}^{n}$ denote the set of $d$-regular graphs on~$n$ vertices.
Firstly, the Kahn-Lov\'{a}sz Theorem (see e.g.~\cite{AF08}) states that a graph with degree sequence $r_{1},\dots,r_{n}$ has at most~$\prod_{i=1}^{n}(r_{i}!)^{1/2r_{i}}$ perfect matchings.
In particular, an $n$-vertex $d$-regular graph has at most~$(d!)^{n/2d}$ perfect matchings.
To determine an upper bound for the number of $1$-factorizations of a $d$-regular graph~$G$, one can simply apply the Kahn-Lov\'{a}sz Theorem repeatedly to obtain $M(G)\leq \prod_{r=1}^{d}(r!)^{n/2r}$.
Using Stirling's approximation, we obtain the following result.\COMMENT{Using Stirling's approximation as $r!=\exp(r\ln r -r +O(\ln r))$ and the well-known asymptotic expansion of the Harmonic number~$H_{d}=\frac{1}{1}+\frac{1}{2}+\dots+\frac{1}{d}=\ln d +\gamma+O(1/d)$, where $\gamma\leq 0.58$ is the Euler-Mascheroni constant, and the Taylor series expansion $e^{x}=1+x+x^{2}/2+\dots$, one obtains
\begin{eqnarray*}
\prod_{r=1}^{d}(r!)^{n/2r} & = & \prod_{r=1}^{d}\exp\left(\frac{n}{2}\ln r -\frac{n}{2}+O\left(\frac{n}{2r}\ln r\right)\right) =\exp\left(\frac{n}{2}(\ln 1+\dots+\ln d)-\frac{dn}{2}+\sum_{r=1}^{d}O\left(\frac{n}{2r}\ln r\right)\right) \\ & = & \exp\left(\frac{n}{2}\ln d! -\frac{dn}{2}+O\left(\sum_{r=1}^{d}\frac{n}{2r}\ln r\right)\right) \\ & = & \exp\left(\frac{n}{2}(d\ln d - d +O(\ln d))-\frac{dn}{2}+O\left(\sum_{r=1}^{d}\frac{n}{2r}\ln r\right)\right) \\ & = & \exp\left(\frac{nd}{2}\ln d -nd +O(n\ln d) + O\left(\sum_{r=1}^{d}\frac{n}{2r}\ln r\right)\right) \\ & = & \exp\left(\frac{nd}{2}\left(\ln d - 2 +O\left(\frac{\ln d}{d}\right) +O\left(\sum_{r=1}^{d}\frac{1}{rd}\ln r\right)\right)\right) \\ & \leq & \exp\left(\frac{nd}{2}\left(\ln d - 2 +O\left(\frac{\ln d}{d}\right) +O\left(\sum_{r=1}^{d}\frac{1}{rd}\ln d\right)\right)\right) \\ & = & \exp\left(\frac{nd}{2}\left(\ln d - 2 +O\left(\frac{\ln d}{d}\right) +O\left(\frac{\ln d}{d}\sum_{r=1}^{d}\frac{1}{r}\right)\right)\right) \\ & = & \exp\left(\frac{nd}{2}\left(\ln d - 2 +O\left(\frac{\ln^{2}d}{d}\right)\right)\right) \leq \left(\left(1+O\left(\frac{\ln^{2}d}{d}\right)+O\left(\frac{\ln^{4}d}{d^{2}}\right)\right)\frac{d}{e^{2}}\right)^{dn/2} \\ & \leq & \left(\left(1+o\left(\frac{\ln^{3}d}{d}\right)\right)\frac{d}{e^{2}}\right)^{dn/2} \leq \left(\left(1+o\left(\frac{\ln^{3}n}{n}\right)\right)\frac{d}{e^{2}}\right)^{dn/2} \leq \left(\left(1+\frac{1}{\sqrt{n}}\right)\frac{d}{e^{2}}\right)^{dn/2},
\end{eqnarray*}
where we have used $d=\Theta(n)$ (say) a couple times towards the end.}
\begin{theorem}\label{linlur}
Suppose $n\in\bN$ is even with $1/n\ll1$\COMMENT{Sufficiently large needed for us to make the simplifications we make in applying Stirling's (and the asymptotic formula for the Harmonic number)}, and $d\geq n/2$.\COMMENT{We don't necessarily need~$d$ to be this large, but $d=\Theta(n)$ does seem helpful in the previous comment.}
Then every $G\in\cG_{d}^{n}$ satisfies
%
%
\[
M(G)\leq\left(\left(1+n^{-1/2}\right)\frac{d}{e^{2}}\right)^{dn/2}.
\]
\end{theorem}
On the other hand, Ferber, Jain, and Sudakov~\cite{FJS20} proved the following lower bound for the number of distinct $1$-factorizations in dense regular graphs.
\begin{theorem}[{\cite[Theorem 1.2]{FJS20}}]\label{Ferb}
Suppose $C>0$ and $n\in\bN$ is even with $1/n\ll1/C\ll1$, and $d\geq(1/2+n^{-1/C})n$.
Then every $G\in\cG_{d}^{n}$ satisfies
\[
M(G)\geq\left(\left(1-n^{-1/C}\right)\frac{d}{e^{2}}\right)^{dn/2}.
\]
\end{theorem}
Theorems~\ref{linlur} and~\ref{Ferb} immediately yield\COMMENT{Define $C$ to be the universal constant from Theorem~\ref{Ferb}. Then, applying Theorems~\ref{Ferb} and~\ref{linlur}, for all sufficiently large even integers~$n$, all $d\geq(1/2+n^{-1/C})n$, and all $G,H\in\cG_{d}^{n}$, we have
\[
\frac{M(G)}{M(H)} \leq \frac{((1+n^{-1/2})d/e^{2})^{dn/2}}{((1-n^{-1/C})d/e^{2})^{dn/2}} \leq \left(1+\frac{2n^{-1/C}}{1-n^{-1/C}}\right)^{dn/2} \leq (1+4n^{-1/C})^{dn/2} \leq \exp(2n^{1-1/C}d).
\]}
the following corollary:
\begin{cor}\label{wf}
Suppose $C>0$ and $n\in\bN$ is even with $1/n\ll1/C\ll1$, and $d\geq(1/2+n^{-1/C})n$.
Then
\[
\frac{M(G)}{M(H)}\leq\exp\left(2n^{1-1/C}d\right),
\]
for all $G,H\in\cG_{d}^{n}$.
\end{cor}
Recall that for $G\in\cG_{[n-1]}^{\text{col}}$ and a set of colours $D\subseteq [n-1]$, $G|_{D}$ is be the spanning subgraph of~$G$ containing precisely those edges of~$G$ which have colour in~$D$.
We now have all the tools we need to prove Lemma~\ref{main-switching-lemma}.
\lateproof{Lemma~\ref{main-switching-lemma}}
Let $C>0$ be the constant given by Corollary~\ref{wf} and suppose that $1/n\ll1/C, \mu, \eps$.
Let~$\pr$ denote the probability measure for the space corresponding to choosing $\mathbf{G}\in\cG_{[n-1]}^{\text{col}}$ uniformly at random.
Fix $D\subseteq [n-1]$ such that $|D|=\eps n$, and let~$\pr_{D}$ denote the probability measure for the space corresponding to choosing $\mathbf{H}\in\cG_{D}^{\text{col}}$ uniformly at random.
Let~$\cG_{D}^{\text{bad}}$ denote the set of $H\in\cG_{D}^{\text{col}}$ such that~$H$ is not $\eps$-locally edge-resilient.
For $H\in\cG_{D}^{\text{col}}$, write~$N_{H}$ for the number of distinct completions of~$H$ to an element $G\in\cG_{[n-1]}^{\text{col}}$; that is,~$N_{H}$ is the number of $1$-factorizations of the complement of~$H$.\COMMENT{A little lazy but hopefully clear that the complement has no interest in the colours of~$H$.}
Then
\begin{eqnarray*}
\prob{\mathbf{G}|_{D}\,\text{is}\,\text{not}\,\eps\text{-locally}\,\text{edge-resilient}} & = & \frac{\sum_{H\in\cG_{D}^{\text{bad}}}N_{H}}{\sum_{H'\in\cG_{D}^{\text{col}}}N_{H'}} \\ & \leq & \probd{\mathbf{H}\in\cG_{D}^{\text{bad}}}\cdot\exp\left(2n^{2-1/C}\right) \\ & \leq & \exp\left(-\eps^{3}n^{2}/2000\right),
\end{eqnarray*}
where we have used Lemma~\ref{localedge} and Corollary~\ref{wf}.
Then, union bounding over choices of~$D$, we deduce that\COMMENT{Let $C'\subseteq [n-1]$, $V'\subseteq V$ each have size $|V'|, |C'|\geq\eps n$. Let $V^{*}\subseteq V'$, $C^{*}\subseteq C'$ be arbitrary subsets of size exactly~$\eps n$. If $\mathbf{G}|_{C^{*}}$ is $\eps$-locally edge-resilient, then $\mathbf{G}$ has at least $\eps^{3}n^{2}/100$ edges with colour in $C^{*}$ and endpoints in $V^{*}$, whence~$\mathbf{G}$ has at least $\eps^{3}n^{2}/100$ edges with colour in $C'$ and endpoints in $V'$.}
\begin{equation}\label{eq:edgeresil}
\prob{\mathbf{G}\,\text{is}\,\text{not}\,\eps\text{-locally}\,\text{edge-resilient}}\leq \binom{n-1}{\eps n}\exp\left(-\frac{\eps^{3}n^{2}}{2000}\right)\leq\exp\left(-\frac{\eps^{3}n^{2}}{4000}\right).
\end{equation}
Now, fix $x\in V$, and fix $c\in [n-1]$.
Choose $F\subseteq[n-1]\setminus\{c\}$ of size $|F|=\mu n$ arbitrarily.
Write $F^{*}\coloneqq F\cup\{c\}$, and let~$\pr_{F^{*}}$ denote the probability measure for the space~$\cS$ corresponding to choosing $\mathbf{H}\in\cG_{F^{*}}^{\text{col}}$ uniformly at random.
Let~$\cP$ be an equitable (ordered) partition of~$F$ into four subsets.
Let~$A_{F^{*}}^{(x,c)}\subseteq\cG_{F^{*}}^{\text{col}}$ be the set of $H\in\cG_{F^{*}}^{\text{col}}$ such that~$H$ has a $5\mu n/4$-well-spread collection of at least $\mu^{4}n^{2}/2^{23}$ $(x,c)$-absorbing gadgets.
Then, considering $A_{F^{*}}^{(x,c)}$, $\cQ_{F^{*}}^{\text{col}}$, $\widetilde{\cQ}_{F^{*}}^{\text{col}}$ as events in~$\cS$, observe that
\begin{eqnarray*}
\probfc{\overline{A_{F^{*}}^{(x,c)}}} & \leq & \probfc{\widetilde{\cQ}_{F^{*}}^{\text{col}}}\probfc{\overline{A_{F^{*}}^{(x,c)}} \biggm| \widetilde{\cQ}_{F^{*}}^{\text{col}}}+\probfc{\overline{\widetilde{\cQ}_{F^{*}}^{\text{col}}}} \\ & \stackrel{(\ref{eq:qsubset})}{\leq} & \probfc{\overline{A_{F^{*}}^{(x,c)}} \biggm| \widetilde{\cQ}_{F^{*}}^{\text{col}}} + \probfc{\overline{\cQ_{F^{*}}^{\text{col}}}}.
\end{eqnarray*}
Thus, applying\COMMENT{Lemma~\ref{justgadgets} is applied in contrapositive form} Lemma~\ref{justgadgets}, Lemma~\ref{quasirandom}, and Lemma~\ref{masterswitch}, we obtain
\begin{eqnarray*}
\probfc{\overline{A_{F^{*}}^{(x,c)}}} & \leq & \probfc{r(\mathbf{H})\leq \mu^{4}n^{2}/2^{23} \biggm| \mathbf{H}\in\widetilde{\cQ}_{F^{*}}^{\text{col}}} +\probfc{\mathbf{H}\notin\cQ_{F^{*}}^{\text{col}}} \\ & \leq & \exp\left(-\frac{\mu^{4}n^{2}}{2^{24}}\right)+\exp\left(-\mu^{3}n^{2}\right) \leq \exp\left(-\frac{\mu^{4}n^{2}}{2^{25}}\right).
\end{eqnarray*}
Then by Corollary~\ref{wf},
\begin{eqnarray*}
\prob{\mathbf{G}|_{F^{*}}\notin A_{F^{*}}^{(x,c)}} & = & \frac{\sum_{H\in\overline{A_{F^{*}}^{(x,c)}}}N_{H}}{\sum_{H'\in\cG_{F^{*}}^{\text{col}}}N_{H'}} \leq \probfc{\mathbf{H}\notin A_{F^{*}}^{(x,c)}}\cdot\exp\left(2n^{2-1/C}\right) \\ & \leq & \exp\left(-\frac{\mu^{4}n^{2}}{2^{26}}\right).
\end{eqnarray*}
In particular, with probability at least $1-\exp(-\mu^{4}n^{2}/2^{26})$,~$\mathbf{G}$ has a~$5\mu n/4$-well-spread collection of at least~$\mu^{4}n^{2}/2^{23}$ $(x,c)$-absorbing gadgets.
Now, union bounding over all vertices $x\in V$ and all colours $c\in [n-1]$, we deduce that
\begin{equation}\label{eq:gadgres}
\prob{\mathbf{G}\,\text{is}\,\text{not}\,\mu\text{-robustly}\,\text{gadget-resilient}} \leq n^{2}\cdot\exp\left(-\frac{\mu^{4}n^{2}}{2^{26}}\right)\leq \exp\left(-\frac{\mu^{4}n^{2}}{2^{27}}\right).
\end{equation}
The result now follows by combining~(\ref{eq:edgeresil}) and~(\ref{eq:gadgres}).
\endproof

\section{Modifications and Corollaries}\label{corollary-section}

In this section we show how to derive the~$n$ odd case of Theorem~\ref{mainthm} from the case when $n$ is even.
We also show how Theorem~\ref{mainthm}\ref{mainthm:rainbow-cycle} implies Corollary~\ref{oddsym}.

\subsection{A rainbow Hamilton cycle for $n$ odd}\label{odd}

We actually derive the~$n$ odd case of Theorem~\ref{mainthm} from the following slightly stronger version of Theorem~\ref{mainthm}\ref{mainthm:rainbow-cycle} in the case when $n$ is even.  

\begin{theorem}\label{rainbow-cycle-missing-vtx}
If $n$ is even and $\phi$ is a uniformly random 1-factorization of $K_n$, then for every vertex $v$, with high probability, $\phi$ admits a rainbow cycle containing all of the colours and all of the vertices except $v$.
\end{theorem}

We now argue that our proof of Theorem~\ref{mainthm} for $n$ even is sufficiently robust to also obtain this strengthening.  In particular, we can strengthen Lemma~\ref{main-absorber-lemma} so that the absorber does not contain $v$, since~\ref{flexible-sets-right-size}--\ref{many-covers-survive} in Lemma~\ref{absorbing-template-lemma}, \ref{absorbing-sets-right-size}--\ref{weak-pseudorandomness-in-slice} in Lemma~\ref{greedy-absorber-lemma}, and~\ref{linking-sets-right-size}--\ref{common-linking-nbrhood-large-reserve} in Lemma~\ref{linking-lemma} all hold after deleting $v$ from any part in the absorber partition.  The proof of Lemma~\ref{long-rainbow-path-lemma} is also sufficiently robust to guarantee that the rainbow path from the lemma does not contain $v$, but we do not need this strengthening, since we can instead strengthen Proposition~\ref{main-absorbing-proposition} to obtain a rainbow cycle containing $P' - v$ and all of the colours, as follows.  If $v\in V(P')$, then we replace $v$ in $P'$ with a $(V_{\mathrm{flex}}, C_{\mathrm{flex}}, G_{\mathrm{flex}})$-cover by deleting $v$ and adding a $(V_{\mathrm{flex}}, C_{\mathrm{flex}}, G_{\mathrm{flex}})$-cover of $w$, $w'$, and $\phi(vw)$, where $w$ and $w'$ are the vertices adjacent to $v$ in $P'$.  The remainder of the proof proceeds normally, letting $v_\ell \coloneqq v$ to ensure $v \notin V(P''_1)$.  In this procedure, we need to assume that $P'$ is contained in $(V\setminus V', C\setminus C', G')$ with $\delta/19$-bounded remainder (rather than $\delta/18$), but in Lemma~\ref{main-absorber-lemma} we can find a $38\gamma$-absorber, which completes the proof.

Now we show how Theorem~\ref{rainbow-cycle-missing-vtx} implies the odd $n$ case of Theorem~\ref{mainthm}.
\lateproof{Theorem~\ref{mainthm},~$n$ odd case}
  When~$n$ is odd, any optimal edge-colouring of~$K_{n}$ has~$n$ colour classes, each containing precisely $(n-1)/2$ edges.
For every colour~$c$, there is a unique vertex which has no incident edges of colour~$c$, and for every vertex~$v$, there is a unique colour such that~$v$ has no incident edges of this colour.  Thus, we can obtain a 1-factorization $\phi'$ of $K_{n+1}$ from an optimal edge-colouring $\phi$ of $K_n$ in the following way.  We add a vertex~$z$, and for every other vertex $v$, we add an edge~$zv$, where $\phi'(zv)$ is the unique colour $c$ such that $v$ is not incident to a $c$-edge in $K_n$.  Note that this operation produces a bijection from the set of $n$-edge-colourings of $K_n$ to the set of 1-factorizations of $K_{n+1}$.  Thus, if $n$ is odd and $\phi$ is a uniformly random optimal edge-colouring of $K_n$, then $\phi'$ is a uniformly random optimal edge-colouring of $K_{n+1}$.  By Theorem~\ref{rainbow-cycle-missing-vtx}, with high probability there is a rainbow cycle $F$ in $K_{n+1}$ containing all of the colours and all of the vertices except $z$, so $F$ is a rainbow Hamilton cycle in $K_n$, satisfying Theorem~\ref{mainthm}\ref{mainthm:rainbow-cycle}.  Deleting any edge from $F$ gives a rainbow Hamilton path, as required in Theorem~\ref{mainthm}\ref{mainthm:rainbow-path}.
\endproof

\subsection{Symmetric Latin squares}
Now we use Theorem~\ref{mainthm} to prove Corollary~\ref{oddsym}.

\lateproof{Corollary~\ref{oddsym}}
Suppose that~$n\in \mathbb N$ is odd.
Firstly, note that there is a one-to-one correspondence between the set~$\cL_{n}^{\text{sym}}$ of symmetric~$n\times n$ Latin squares with symbols in~$[n]$ (say) and the set~$\Phi_{n}$ of optimal edge-colourings of~$K_{n}$ on vertices~$[n]$ and with colours in~$[n]$.
Indeed, let~$\phi\in\Phi_{n}$.
Then we can construct a unique symmetric Latin square~$L_{\phi}\in \cL_{n}^{\text{sym}}$ by putting the symbol~$\phi(ij)$ in position~$(i,j)$ for all edges $ij\in E(K_{n})$, and for each position $(i,i)$ on the leading diagonal we now enter the unique symbol still missing from row~$i$.
Conversely, let $L\in\cL_{n}^{\text{sym}}$.
We can obtain a unique element $\phi_{L}\in\Phi_{n}$ from~$L$ in the following way.
Colour each edge~$ij$ of the complete graph~$K_{n}$ on vertex set~$[n]$ with the symbol in position~$(i,j)$ of~$L$.
It is clear that~$\phi_{L}$ is proper, and thus~$\phi_{L}$ is optimal.
Moreover, it is clear that we can uniquely recover~$L$ from~$\phi_{L}$.

Now, let $K_n^{\circ}$ be the graph obtained from $K_n$ by adding a loop $ii$ at every vertex $i\in[n]$, and for every $\phi\in\Phi_n$, let $\phi^\circ$ be the unique proper $n$-edge-colouring of $K_n^\circ$ such that the restriction of $\phi^\circ$ to the underlying simple graph is $\phi$.  The rainbow 2-factors in $K_n^\circ$ admitted by $\phi^\circ$ correspond to transversals in $L_\phi$ in the following way.  If $L\in\cL_{n}^{\text{sym}}$ and $T$ is a transversal of $L$, then the subgraph of $K_n^{\circ}$ induced by the edges $ij$ where $(i, j) \in T$ is a rainbow 2-factor.  If $\sigma$ is the underlying permutation of $T$, then the cycles of this rainbow 2-factor are precisely the cycles in the cycle decomposition of $\sigma$, up to orientation.  Therefore a rainbow Hamilton cycle in $K_n^\circ$ corresponds to two disjoint Hamilton transversals in $L_\phi$.

By these correspondences, for $n$ odd, if $\mathbf L \in \cL_{n}^{\text{sym}}$ is a uniformly random symmetric $n\times n$ Latin square, then $\phi_{\mathbf L}$ is a uniformly random optimal edge-colouring of $K_n$.  By Theorem~\ref{mainthm}\ref{mainthm:rainbow-cycle}, $\phi_{\mathbf{L}}$ admits a rainbow Hamilton cycle $F$ with high probability.  Since $F$ is also a rainbow Hamilton cycle in $K_n^{\circ}$, the corresponding transversals in $\mathbf L$ are Hamilton, as desired. 
\endproof

Note that, if~$n$ is odd, the leading diagonal of any $L\in\cL_{n}^{\text{sym}}$ is also a transversal, disjoint from any Hamilton transversal.
Indeed, by symmetry all symbols appear an even number of times off of the leading diagonal, and therefore an odd number of times (and thus exactly once) on the leading diagonal.

\bibliographystyle{amsplainv2}
\bibliography{References}
\vspace{20mm}
\end{document}